\documentclass[12pt, leqno]{amsart}

\usepackage{graphicx}
\usepackage{amssymb,amsmath,amsthm,amscd}
\usepackage{mathrsfs}
\usepackage{enumerate}
\usepackage{verbatim}
\usepackage[usenames,dvipsnames]{color}
\usepackage[colorlinks=true, pdfstartview=FitV,
 linkcolor=blue,citecolor=blue,urlcolor=blue]{hyperref}
\usepackage[all]{xy}

\allowdisplaybreaks[4]

\newcommand{\nc}{\newcommand}
\numberwithin{equation}{subsection}

\newcommand{\er}{\end{red}}

\theoremstyle{plain}
\newtheorem{lemma}{Lemma}[subsection]
\newtheorem{prop}[lemma]{Proposition}
\newtheorem{theorem}[lemma]{Theorem}
\newcommand{\Prop}{\begin{prop}}
\newcommand{\enprop}{\end{prop}}
\newcommand{\Lemma}{\begin{lemma}}
\newcommand{\enlemma}{\end{lemma}}
\newcommand{\Th}{\begin{theorem}}
\newcommand{\enth}{\end{theorem}}
\newtheorem{corollary}[lemma]{Corollary}
\newcommand{\Cor}{\begin{corollary}}
\newcommand{\encor}{\end{corollary}}
\newtheorem{definition}[lemma]{Definition}

\newcommand{\Def}{\begin{definition}}
\newcommand{\edf}{\end{definition}}
\newtheorem{sublemma}[lemma]{Sublemma}
\newcommand{\Sublemma}{\begin{sublemma}}
\newcommand{\ensub}{\end{sublemma}}

\theoremstyle{definition}
\newtheorem{remark}[lemma]{Remark}
\newtheorem{example}[lemma]{Example}
\newtheorem{Convention}[lemma]{Convention}
\newcommand{\Conv}{\begin{Convention}}
\newcommand{\enconv}{\end{Convention}}
\nc{\Rem}{\begin{remark}}
\nc{\enrem}{\end{remark}}
\newcommand{\C}{{\mathbb C}}
\newcommand{\Q}{\mathbb {Q}}

\newcommand{\Z}{{\mathbb Z}}
\newcommand{\B}{{\mathbf{B}}}
\newcommand{\A}{{\mathbf A}}

\newcommand{\R}{{\rm R}}

\newcommand{\CC}{{\mathscr{C}}}
\newcommand{\one}{{\bf{1}}}
\newcommand{\seteq}{\mathbin{:=}}

\newcommand{\hd}{{\operatorname{hd}}}
\newcommand{\soc}{{\operatorname{soc}}}

\newcommand{\g}{{\mathfrak{g}}}

\newcommand{\Tor}{\operatorname{Tor}}
\newcommand{\Hom}{\operatorname{Hom}}
\newcommand{\End}{\operatorname{End}}

\newcommand{\isoto}[1][]{\mathop{\xrightarrow%
[{\raisebox{.3ex}[0ex][.3ex]{$\scriptstyle{#1}$}}]%
{{\raisebox{-.6ex}[0ex][-.6ex]{$\mspace{2mu}\sim\mspace{2mu}$}}}}}
\newcommand{\tensor}{\otimes}

\newcommand{\eq}{\begin{eqnarray}}
\newcommand{\eneq}{\end{eqnarray}}

\newcommand{\hs}{\hspace*}

\newcommand{\To}[1][{\hs{2ex}}]{\xrightarrow{\,#1\,}}

\newcommand{\eqn}{\begin{eqnarray*}}
\newcommand{\eneqn}{\end{eqnarray*}}
\newcommand{\on}{\operatorname}
\newcommand{\Ker}{\on{Ker}}

\newcommand{\bni}{\be[{\rm(i)}]}
\newcommand{\bna}{\be[{\rm(a)}]}

\newcommand{\QED}{\end{proof}}
\newcommand{\Proof}{\begin{proof}}

\newcommand{\soplus}{\mathop{\mbox{\normalsize$\bigoplus$}}\limits}

\newcommand{\id}{\on{id}}
\newcommand{\ba}{\begin{array}}
\newcommand{\ea}{\end{array}}

\newcommand{\Coker}{{\operatorname{Coker}}}

\newcommand{\bi}{\begin{enumerate}[{\rm(i)}]}

\newcommand{\monoto}{\rightarrowtail}

\newcommand{\indlim}{\varinjlim\limits}
\newcommand{\set}[2]{\left\{#1 \mid #2 \right\}}

\newcommand{\Mod}{\operatorname{Mod}}
\newcommand{\Modg}{\operatorname{{Mod}_{\mathrm{gr}}}}

\newcommand{\eqsub}{\begin{subequations}\begin{eqnarray}}
\newcommand{\eneqsub}{\end{eqnarray}\end{subequations}}

\newcommand{\ol}{\overline}

\nc{\la}{\lambda}
\nc{\lam}{\lambda}
\nc{\U}[1][\g]{U_q(#1)}
\nc{\te}{\tilde{e}}
\nc{\tei}{\tilde{e}_i}
\nc{\tf}{\tilde{f}}
\nc{\tfi}{\tilde{f}_i}
\nc{\tU}{\widetilde U_q(\g)}
\nc{\tE}{\tilde{E}}
\nc{\tF}{\widetilde{\F}}
\nc{\tK}{\widetilde{K}}

\nc{\tk}{\tilde{k}}
\nc{\tkone}{\tk_{\ol{1}}}
\nc{\teone}{\tilde{e}_{\ol{1}}}
\nc{\tfone}{\tilde{f}_{\ol{1}}}

\nc{\teibar}{\tilde{e}_{\ol{i}}} \nc{\tfibar}{\tilde{f}_{\ol{i}}}
\nc{\tki}{{\tk}_{\ol {i}}}

\nc{\BZ}{{\mathbb{Z}}}
\nc{\al}{\alpha}
\nc{\qs}{{q}}
\nc{\lan}{\langle}
\nc{\ran}{\rangle}
\nc{\re}{{\mathrm{re}}}
\nc{\wt}{\operatorname{wt}}
\nc{\ch}{\operatorname{ch}}
\nc{\Uf}[1][\g]{U^-_q(#1)}
\nc{\Ue}{U^+_q(\g)}
\nc{\eps}{\varepsilon}
\nc{\vphi}{\varphi}
\nc{\sphi}{\varphi^*}
\nc{\seps}{\varepsilon^*}

\nc{\nn}{\nonumber}
\def\max{{\mathop{\mathrm{max}}}}
\nc{\vp}{\varpi}
\nc{\cls}{{\operatorname{cl}}}
\nc{\Wt}{{\operatorname{Wt}}}
\nc{\Us}{U'_q(\g)}
\nc{\La}{\Lambda}
\nc{\ro}{{\rm(}}
\nc{\rf}{{\rm)}}
\nc{\norm}{{\mathrm{norm}}}
\nc{\qbox}{\quad\mbox}
\nc{\braid}{{\mathfrak{B}}}
\nc{\Ad}{\operatorname{Ad}}
\nc{\Aut}{\operatorname{Aut}}
\nc{\dt}[1]{\tilde{\tilde #1}}
\nc{\Sn}{S^{{\mathrm{norm}}}}
\nc{\aff}{{\rm{aff}}}
\nc{\rk}{{\mathrm{rk}}}
\nc{\tP}{\widetilde{P}}
\nc{\tW}{\widetilde{W}}
\nc{\Dyn}{\mathrm{Dyn}}
\nc{\tD}{\widetilde{\Delta}}
\nc{\height}{{\operatorname{ht}}}
\nc{\bl}{\bigl(}
\nc{\br}{\bigr)}
\nc{\Hecke}{\mathrm{H}}
\nc{\HA}{\Hecke^{\mathrm{A}}}
\nc{\HB}{\Hecke^{\mathrm{B}}}
\newcommand{\scbul}{{\,\raise1pt\hbox{$\scriptscriptstyle\bullet$}\,}}
\nc{\vac}{{\phi}}
\nc{\Bt}{\B_\theta(\g)}
\nc{\be}{\begin{enumerate}}
\nc{\ee}{\end{enumerate}}
\nc{\low}{{\mathrm{low}}}
\nc{\upper}{{\mathrm{up}}}
\nc{\Zodd}{\Z_{\mathrm{odd}}}
\nc{\Ft}[1][n]{\mathbb{P}\mathrm{ol}_{#1}}
\nc{\Ftf}[1][n]{\widetilde{\mathbb{P}\mathrm{ol}}_{#1}}
\nc{\KA}{\on{K}^{\mathrm{A}}}
\nc{\KB}{\on{K}^{\mathrm{B}}}
\nc{\Res}{\on{Res}}
\nc{\Fc}[1][{n,m}]{\mathbf{F}_{#1}}
\nc{\tphi}{\tilde{\varphi}}
\nc{\CO}{\mathscr{O}}
\nc{\inte}{\mathrm{int}}
\nc{\Oint}{\mathcal{O}^{\ge0}_{\inte}}
\nc{\vs}{\vspace*}
\nc{\tL}{\widetilde{L}}
\nc{\tu}{\tilde{u}}
\nc{\noi}{\noindent}
\nc{\heigh}{\mathfrak{t}}
\nc{\lowest}{\mathfrak{l}}
\nc{\rootl}{\mathsf{Q}}
\nc{\cl}{{\rm{cl}}}
\nc{\uqpg}{U'_q(\mathfrak g)}
\nc{\Oh}{\widehat{\mathcal{O}}}

\nc{\KLR}{Khovanov-Lauda-Rouquier algebra}
\nc{\KLRs}{Khovanov-Lauda-Rouquier algebras}
\nc{\cor}{\mathbf{k}}
\nc{\cora}{{\cor(A)}}
\nc{\haut}{\mathrm{ht}}
\nc{\tens}{\mathop\otimes}
\nc{\gmod}{\mbox{-$\mathrm{gmod}$}}
\nc{\proj}{\mbox{-$\mathrm{proj}$}}
\nc{\gproj}{\mbox{-$\mathrm{gproj}$}}
\nc{\smod}{\mbox{-$\mathrm{mod}$}}

\nc{\h}{\mathfrak h}
\nc{\Rnorm}{R^{\rm{norm}}}
\nc{\K}{\C(q)}
\nc{\Vhat}{\widehat{V}}
\nc{\F}{\mathcal{F}}
\def\AA{{\mathcal A}}

\def\T{{\mathcal T}}

\nc{\fd}[1][A]{\on{\mathrm{flat.dim}_{#1}}}
\nc{\bP}{{\mathbb{P}}}
\nc{\bPh}{\widehat{\mathbb{P}}}
\nc{\bK}[1][{n}]{\widehat{\mathbb{K}}_{#1}}
\nc{\bV}[1][{n}]{\widehat{V}^{\otimes{#1}}}
\nc{\bVK}[1][{n}]{\widehat{V}^{\otimes{#1}}_{\widehat{\mathbb{K}}}}
\nc{\hV}{\widehat{V}}
\nc{\opp}{\mathrm{opp}}
\nc{\col}{\colon}
\nc{\bnum}{\be[{\rm(i)}]}
\nc{\oep}{\epsilon}
\nc{\qtext}{\quad\text}
\nc{\qtextq}[1]{\quad\text{#1}\quad}
\nc{\longtwoheadrightarrow}[1][]{\xymatrix{\ar@{->>}[r]^-{{#1}}&}}
\nc{\epiTo}[1][]{\longtwoheadrightarrow[{#1}]}
\nc{\epito}{\twoheadrightarrow}
\nc{\monoTo}[1][]{\xymatrix{\ar@{>->}[r]^-{{#1}}&}}
\nc{\sym}{\mathfrak{S}}
\nc{\inp}[1]{{({#1})_{\mathrm{n}}}}
\nc{\rtl}{\rootl}
\nc{\wtd}{\widetilde}
\nc{\etens}{\boxtimes}
\nc{\ds}[1]{\mathrm{d}(#1)}
\nc{\rmat}[1]{{r}_{\mspace{-2mu}\raisebox{-.5ex}{${\scriptstyle{#1}}$}}}
\nc{\shc}{\mathcal{C}}
\nc{\Fct}{{\on{Fct}}}
\nc{\tC}{\widetilde{\shc}}
\nc{\Zp}{\Z_{\ge0}}
\nc{\tPhi}{\widetilde{\Phi}}
\nc{\tT}{{\widetilde{\T}}}
\nc{\Ob}{\on{Ob}}
\nc{\bwr}{\mbox{\large$\wr$}}
\nc{\Img}{\on{Im}}
\nc{\Ab}{\mathcal{A}^{\mathrm{big}}}
\nc{\Sb}{\mathcal{S}^{\mathrm{big}}}
\nc{\As}{\mathcal{A}}
\nc{\Ss}{\mathcal{S}}
\nc{\ntens}{\widetilde{\otimes}}
\nc{\hR}{\widehat{R}}
\nc{\nconv}{\star}
\nc{\ts}{\tilde{s}}
\nc{\sho}{\mathcal{O}}
\nc{\bc}{\begin{cases}}
\nc{\ec}{\end{cases}}
\nc{\slnh}{{\widehat{\mathfrak{sl}}_N}}
\nc{\UA}{U_q'(\slnh)}
\nc{\KR}{R_K}
\nc{\cQ}{\mathcal{Q}}
\nc{\Irr}{\mathcal{I}rr}
\nc{\tQ}{\widetilde{\cQ}}
\nc{\bs}{\mathbf{s}}
\nc{\bL}{\mathbb{L}}
\nc{\tg}{\tilde{g}}

\newlength{\mylength}
\title[Symmetric quiver Hecke algebras and R-matrices]
{Symmetric quiver Hecke algebras\\ and \\R-matrices of quantum affine algebras}

\author[S.-J. Kang, M. Kashiwara, M. Kim]{Seok-Jin Kang$^{1}$, Masaki Kashiwara$^{2}$, Myungho Kim$^{3}$}

\address[S.-J. Kang]{Korea Research Institute of Arts and Mathematics, Asan, Korea}
         \email{soccerkang@hotmail.com}

\address[M. Kashiwara]{%
Kyoto University Institute for Advanced Study, Research Institute
for Mathematical Sciences, Kyoto University, Kyoto 606-8502, Japan
\& Korea Institute for Advanced Study, Seoul 02455, Korea }

 \address[M. Kim]{Department of Mathematics, Kyung Hee University, Seoul 02447, Korea}
\email{mkim@khu.ac.kr}

\thanks{$^1$ This work was supported by NRF Grant \# 2012-005700 and
 NRF Grant \# 2011-0027952.}
\thanks{$^2$ This work was partially supported by Grant-in-Aid for
Scientific Research (B) 22340005, Japan Society for the Promotion of
Science.}
\thanks{$^{3}$ This work was partially supported by NRF Grant \# 2012-005700.}


\date{\today}

\begin{document}

\begin{abstract}
Let $J$ be a set of pairs consisting of good $\uqpg$-modules and
invertible elements in the base field $\mathbb C(q)$. The
distribution of poles of normalized R-matrices yields
Khovanov-Lauda-Rouquier algebras $R^J(\beta)$ for each $\beta \in \rootl^+$.
 We define a functor $\mathcal F_\beta$  from the category of graded
$R^J(\beta)$-modules to the category of $\uqpg$-modules. The functor
$\mathcal F= \bigoplus_{\beta \in \rootl^+} \mathcal F_\beta$ sends convolution products of finite-dimensional graded
$R^J(\beta)$-modules to tensor products of finite-dimensional
$\uqpg$-modules. It is exact if $R^J$ is of finite type $A,D,E$. If
$V(\varpi_1)$ is the fundamental representation of $\UA$ of weight
$\varpi_1$ and $J=\set{\bl V(\varpi_1), q^{2i} \br}{i \in \mathbb
Z}$, then $R^J$ is the \KLR\ of type $A_{\infty}$.
The corresponding functor $\F$ sends
a finite-dimensional graded $R^J$-module to a module in $\mathcal C_J$, where
$\mathcal C_J$ is the category of finite-dimensional integrable $\UA$-modules $M$ such that
every composition factor of $M$ appears as a composition factor of a
tensor product of modules of the form $V(\varpi_1)_{q^{2s}}$ $(s \in
\Z)$.
Focusing on this
case, we obtain an abelian rigid graded tensor category $\T_J$ by
localizing the category of finite-dimensional graded $R^J$-modules.
The functor $\mathcal F$ factors through $\T_J$. Moreover, the
Grothendieck ring of the category $\mathcal C_J$ is isomorphic to the Grothendieck ring of $\T_J$ at
$q=1$.
\end{abstract}

\maketitle

\tableofcontents

\section*{Introduction}

The {\em \KLRs}\ (sometimes called the {\em quiver Hecke algebras})  are
a family of $\mathbb Z$-graded algebras which {\em categorifies}
the negative half of a quantum group (\cite{KL09, KL11, R08}). More
precisely, if $U_q(\g)$ is a quantum group associated with a
symmetrizable Cartan datum, then there exists a  family of algebras
$\{R(n)\}_{n \in \Z_{\ge 0}}$ such that the Grothendieck group of
the direct sum $\soplus\nolimits_{n\in\Z_{\ge0}}R(n)\gproj$ of
the categories of finitely generated projective graded $R(n)$-modules
is isomorphic to the integral form
$U^-_{\A}(\g)$ of the negative half of $U_q(\g)$. Here, the
multiplication of the
Grothendieck group is given by the convolution product and
the action of $q$ is given by the
grading shift functor. Moreover, the cyclotomic
quotient $R^{\Lambda}(n)$ of $R(n)$ provides a categorification of the
integrable highest weight module $V(\Lambda)$ of $U_q(\g)$
(\cite{KK12}).

One of the motivations of these categorification theorems originated
from the so-called {\em LLT-Ariki theory}. In 1996,
Lascoux-Leclerc-Thibon (\cite{LLT}) conjectured that the irreducible
representations of Hecke algebras of type $A$ are controlled by the
upper global basis (\cite{Kas90, Kas91})
 (or dual canonical basis (\cite{Lus93}))
 of the basic representation of the quantum affine algebra $U_q(A^{(1)}_{N-1})$.
Soon after, Ariki (\cite{Ariki})  proved this conjecture by showing that the
cyclotomic quotients of affine Hecke algebras categorify the
irreducible highest weight modules over $U(A^{(1)}_{N-1})$, the
universal enveloping algebra of
 affine Kac-Moody algebra of type $A^{(1)}_{N-1}$.
 In \cite{BK09, R08}, Brundan-Kleshchev and Rouquier showed that
the affine Hecke algebra of type $A$ is isomorphic to the \KLR\ of
type A or of type $A_\infty$ up to a specialization and a
localization. Thus the \KLRs\ can be understood as a graded version
of the affine Hecke algebras of type $A$ and Kang-Kashiwara's
cyclotomic categorification theorem is a generalization of Ariki's
theorem on type $A$ and $A_\infty$ to all symmetrizable Cartan
datum.

The purpose of this paper is to show that the \KLR\ can be regarded
as a generalization of the affine Hecke algebra of type $A$ in
another context: the {\em quantum affine Schur-Weyl duality}.
In \cite{Jimbo86}, M.\ Jimbo  extended the classical
Schur-Weyl duality to the quantum case:
the duality between the
category of finite-dimensional modules over the Hecke algebra $H_q(n)$ and
the category of finite-dimensional modules over the quantum group
$U_q(\mathfrak{gl}_N)$.
In
\cite{CP96, Che, GRV94}, Chari-Pressley, Cherednik and
Ginzburg-Reshetikhin-Vasserot introduced a quantum affine version of
Schur-Weyl duality: they
defined a functor between the category of
finite-dimensional modules over the affine Hecke algebra $H_q^{\aff}(n)$
and the category of finite-dimensional
integrable modules over $\UA$.

There are two key ingredients in defining this functor: (1) the
R-matrices on the $n$-fold tensor product of the fundamental
representation $V(\varpi_1)$ which satisfy the  Yang-Baxter
equations,  (2) a set of elements in $H_q^{\aff}(n)$, called the {\em
intertwiners}, which satisfy the braid relations. Roughly speaking,
by assigning  intertwiners to  R-matrices, we obtain the quantum
affine Schur-Weyl duality functor.

Now it is quite natural to ask whether this functor can be
generalized to the case of quantum affine algebras $\uqpg$ of other
types or not.
Our answer to this question can be explained in the following way.

Let $\{V_s\}_{s\in \mathcal{S}}$ be a family of good $\uqpg$-modules
and let $J$ be a subset of $\mathcal{S}\times \C(q)^{\times}$. An
element $i \in J$ is denoted by $i=({S(i)}, X(i))$. We define a
quiver $\Gamma_J$ as follows:
\be[{1)}]
\item The set of vertices is taken to be $J$.

\item For $i,j\in  J $, let $$\Rnorm_{V_{S(i)},V_{S(j)}}\colon
(V_{S(i)})_u\tens (V_{S(j)})_v \to  (V_{S(j)})_v\tens(V_{S(i)})_u$$
be the normalized R-matrix. We put $d$ many arrows from $i$ to $j$
where $d$ is the order of poles of $\Rnorm_{V_{S(i)},V_{S(j)}}(v/u)$
at $v/u=X(j)/X(i)$.
\ee

\noi Then the quiver $\Gamma_J$ defines a Cartan datum
$(A=(a_{ij})_{i,j \in J}, P, \Pi, P^{\vee}, \Pi^{\vee})$, where $A$
is determined by \eqref{eq:quiver}. Hence we obtain a
Khovanov-Lauda-Rouquier algebra $R^J(\beta)$ $(\beta \in \rtl^+)$
with the parameters  $Q_{i,j}(u,v) =\delta(i\not=j)
(u-v)^{d_{ij}}(v-u)^{d_{ji}} \in \cor[u,v]$. Here,
$\rtl^+=\sum_{i\in J} \Z_{\ge 0} \alpha_i$ denotes the positive root
lattice.

Let $\beta = \sum_{i \in J} k_i \alpha_i \in \rtl^{+}$ with
$|\beta|:=\sum_{i \in J} k_i =n$ and let $J^{\beta} = \{ \nu =
(\nu_1, \ldots, \nu_n) \in J^n \,;\, \alpha_{\nu_1} + \cdots +
\alpha_{\nu_n} = \beta \}$. For each sequence
$\nu=(\nu_1,\ldots,\nu_n) \in J^{\beta}$, set
$$V_\nu = (V_{S(\nu_1)})_{\rm aff} \otimes \cdots \otimes (V_{S(\nu_n)})_{\rm aff}$$
and let $\widehat{V}^{\otimes \beta}$ be a certain completion of
$\soplus_{\nu \in J^{\beta}} V_{\nu}$. Then $\widehat{V}^{\otimes
\beta}$ is endowed with a structure of $(U'_q(\g),
R^J(\beta))$-bimodule. Hence we can define a functor
$$\mathcal{F}_{\beta} \colon  \Modg(R^J(\beta)) \rightarrow \Mod(U'_q(\g)) $$
by
$$M \longmapsto \widehat{V}^{\otimes \beta}  \otimes_{R^J(\beta)} M,$$
where $\Modg(R^J(\beta))$ denotes the category of graded
$R^J(\beta)$-modules and $\Mod(U'_q(\g))$ denotes the category of
$U'_q(\g)$-modules. We would like to emphasize that, in general, the
type of the quiver $\Gamma_J$ is  irrelevant to  the type of
the affine Lie algebra $\g$ or its underlying finite-dimensional
simple Lie algebra $\g_0$.

We denote by $R^{J}(\beta)\gmod$  the category of finite-dimensional
graded $R^{J}(\beta)$-modules and set $R^{J}\gmod = \soplus_{\beta
\in \rtl^{+}} R^{J}(\beta)\gmod$. We also denote by $\CC_{\g}$ the
category of finite-dimensional integrable $U'_q(\g)$-modules.

The functor $\mathcal F_{\beta}$ enjoys two important properties:

\begin{itemize}
\item[(1)] $\F:=\soplus_{\beta \in \rtl^{+}}\F_{\beta}$ is a tensor functor; i.e.,
$\F$ sends a convolution product of objects in $R^{J}\gmod$ to a
tensor product of $\uqpg$-modules in $\CC_{\g}$.

\item[(2)] $\mathcal F_{\beta}$ is exact if $R^J(\beta)$ is of finite type $A,D,E$.
\end{itemize}

\vskip1em

One of the most important features of the functor $\mathcal F$ is that the category $R^{J}\gmod$, domain of the functor, has a $\Z$-grading.
For many years, it has been asked if the category  $\CC_{\g}$, the codomain of $\mathcal F$, has a $\Z$-grading or not.
This question has arisen after the works of Varagnolo-Vasserot \cite{VV02}, Nakajima \cite{Nak04} and Hernandez \cite{Her04},
in which it was shown that the Grothendieck ring $K(\CC_{\g})$ of $\CC_{\g}$  has interesting $t$-deformations. In particular, Nakajima provided an algorithm of calculating  {\em $q$-characters} of simple representations \cite{Nak04} using $t$-deformation.
Thus one would expect that there might be a $\Z$-graded tensor category whose Grothendieck ring is isomorphic to one of such  $t$-deformations of $K(\CC_{\g})$.
We will see in the last half of this paper that the functor $\mathcal F$ helps us construct such a category  in type $A_{N-1}^{(1)}$ case.

 Precisely speaking, this is  the case when $J=\set{(V(\varpi_1), q^{2i})}{i\in \Z}$,
where $V(\varpi_1)$ is the fundamental representation of
$\UA$ of weight $\varpi_1$. Then, by the construction, the
corresponding \KLR\  is of type $A_\infty$ and the functor
$\F_{n}:=\soplus_{|\beta|=n} \F_{\beta}$ recovers the quantum affine
Schur-Weyl duality functor if  we identify the objects in
$R^{J}(n)\gmod:= \soplus_{|\beta|=n} R^{J}(\beta)\gmod$ with
finite-dimensional  $H_q^{\rm aff}(n)$-modules in an appropriate way
(\cite{Kim12, R08}).

 Let $\As = R^{J}\gmod$ and let ${\mathcal C}_{J}$ be the full
subcategory of $\CC_{\g}$ consisting of $\UA$-modules $M$ such that
every composition factor of $M$ appears as a composition factor of a
tensor product of modules of the form $V(\varpi_1)_{q^{2s}}$ $(s \in
\Z)$. Then ${\mathcal C}_{J}$ is an abelian category and is stable
under taking submodules, quotients, extensions and tensor products.
Moreover, ${\mathcal C}_{J}$ contains all the modules of the form
$V(\varpi_i)_{(-q)^{s}}$ for $1 \le i \le N-1$ and $s \in i-1+2\Z$
(see \S\;\ref{subsec:CJ}). Then  $\F$ restricts to a functor from
$\As$ to $\mathcal C_J$.

Next,  we construct a tensor category $\T_J$
in two steps.
 Let $\Ss$ be the kernel of the functor $\F$.
Then $\Ss$ is a \ Serre subcategory of $\AA$. 
 The quotient category $\As/\Ss$
has a structure of a tensor category  with the convolution as tensor product.
Secondly, we localize the category $\AA
/\Ss$ one step further to obtain a category $\T_J$, where the
$R^J$-modules mapped  to  the trivial representations of $\UA$
 by $\F$  are  isomorphic to the unit object in $\T_J$.
 We then show that the category $\T_J$ is an abelian
rigid tensor category; i.e., every object in $\T_J$ has a right dual
and a left dual. To prove our assertions, we develop a general
process of localization of tensor categories through a {\em
commuting family of central objects}.  We summarize this process in
 Appendix A.

Moreover, the functor  $\mathcal F\col \As\to\mathcal C_J$
factors through the category $\T_J$. Hence, roughly speaking, $\T_J$
can be regarded as a graded version of $\mathcal C_J$. Note that the
category $\T_J$ is $\Z$-graded (i.e., endowed with a grading shift
functor), while the category $\mathcal C_J$ is not. Thus the
structure of the category $\T_J$ is definitely richer than that of
$\mathcal C_J$.

 One of our main results is  that $\mathcal F$ induces  a ring isomorphism
between $K(\T_J)_{q=1}$, the Grothendieck ring of $\T_J$ forgetting
the grading and  the Grothendieck ring $K(\mathcal C_J)$ of the
category $\mathcal C_J$.  
 Indeed, as proved in Theorem~\ref{th:def},
$K(\T_J)$ is isomorphic to the $t$-deformation $\mathcal K_t$ of 
$K(\mathcal C_J)$ introduced by Hernandez-Leclerc in \cite{HL11}
(cf.\ \cite{VV02, Nak04, Her04}). 
Therefore $\T_J$ provides a
$\Z$-graded lifting of ${\mathcal C}_{J}$ as a rigid
tensor category.

As we have seen so far, even the simplest case yields many
interesting and deep applications of the functor $\mathcal F$. We
expect  that a lot more exciting developments will follow as
we take various choices of  the affine algebra $\g$  and $J$.

This paper is organized as follows. In the first section, we recall
the notions of quantum groups and \KLRs.\ Next, we introduce the
R-matrices for \KLRs\ and study their basic properties. In the
second section, we recall the quantum affine algebras and their
representation theory. The notion of normalized R-matrices are also
recalled. In the third section, we define an $(R^J(\beta)),\uqpg
)$-bimodule $\widehat{V}^{\otimes \beta}$ and a functor $\F_{\beta}$
between the category of graded $R^J(\beta)$-modules and the category
of $\uqpg$-modules. In the last section, we construct a category
$\T_J$ and investigate its fundamental properties. In the
appendices, we gather necessary lemmas and propositions on the
localization of tensor categories and the quotient categories of an
abelian category by a Serre subcategory.

\smallskip

\noi
{\bf Acknowledgements.} We would like to express our gratitude to
Bernard Leclerc for his kind explanations of his works
and many fruitful discussions.
The first and the third author gratefully acknowledge
the hospitality of RIMS (Kyoto)
 during their visit in 2011 and 2012.

\section*{Convention}

\bnum
\item All the algebras and rings in this paper are assumed to have a unit,
and modules over them are unitary.
\item For a ring $A$, $A^\opp$ denotes the opposite ring of $A$.
\item For a ring $A$,  an $A$-module means a left $A$-module.
\item For a statement $P$, $\delta(P)$ is $1$ if $P$ is true and $0$
if $P$ is false.
\item For a ring $A$, we denote by $\Mod(A)$ the category of $A$-modules.
We denote by $A\proj$ the category of finitely generated projective $A$-modules.
When $A$ is an algebra over a field $\cor$, we denote by $A\smod$
the category of $A$-modules which are finite-dimensional over $\cor$.

If $A$ is a graded ring, then we denote by
$\Modg(A)$, $A\gproj$, $A\gmod$ their graded version with
homomorphism preserving the grading as morphisms.
They are also exact categories.
\item For a ring $A$, we denote by $A^\times$ the set of
invertible elements of $A$.
\ee

\section{Quantum groups and \KLRs}

\subsection{Cartan datum}\label{subsec:Cartan}
In this subsection, we recall the definition of quantum groups. Let $I$
be an index set. A \emph{Cartan datum} is a quintuple $(A,P,
\Pi,P^{\vee},\Pi^{\vee})$ consisting of
\begin{enumerate}[(a)]
\item an integer-valued matrix $A=(a_{ij})_{i,j \in I}$,
called the \emph{symmetrizable generalized Cartan matrix},
 which satisfies
\bni
\item $a_{ii} = 2$ $(i \in I)$,
\item $a_{ij} \le 0 $ $(i \neq j)$,
\item $a_{ij}=0$ if $a_{ji}=0$ $(i,j \in I)$,
\item there exists a diagonal matrix
$D=\text{diag} (\mathsf s_i \mid i \in I)$ such that $DA$ is
symmetric, and $\mathsf s_i$ are positive integers.
\end{enumerate}

\item a free abelian group $P$, called the \emph{weight lattice},
\item $\Pi= \{ \alpha_i \in P \mid \ i \in I \}$, called
the set of \emph{simple roots},
\item $P^{\vee}\seteq\Hom(P, \Z)$, called the \emph{co-weight lattice},
\item $\Pi^{\vee}= \{ h_i \ | \ i \in I \}\subset P^{\vee}$, called
the set of \emph{simple coroots},
\end{enumerate}

satisfying the following properties:
\begin{enumerate}
\item[(i)] $\langle h_i,\alpha_j \rangle = a_{ij}$ for all $i,j \in I$,
\item[(ii)] $\Pi$ is linearly independent,
\item[(iii)] for each $i \in I$, there exists $\Lambda_i \in P$ such that
           $\langle h_j, \Lambda_i \rangle =\delta_{ij}$ for all $j \in I$.
\end{enumerate}
We call $\Lambda_i$ the \emph{fundamental weights}.
The free abelian group $\rootl\seteq\soplus_{i \in I} \Z \alpha_i$ is called the
\emph{root lattice}. Set $\rootl^{+}= \sum_{i \in I} \Z_{\ge 0}
\alpha_i\subset\rootl$ and $\rootl^{-}= \sum_{i \in I} \Z_{\le0}
\alpha_i\subset\rootl$. For $\beta=\sum_{i\in I}m_i\al_i\in\rootl$,
we set
$|\beta|=\sum_{i\in I}|m_i|$.

Set $\mathfrak{h}=\Q \otimes_\Z P^{\vee}$.
Then there exists a symmetric bilinear form $(\quad , \quad)$ on
$\mathfrak{h}^*$ satisfying
$$ (\alpha_i , \alpha_j) =\mathsf s_i a_{ij} \quad (i,j \in I)
\quad\text{and $\lan h_i,\lambda\ran=
\dfrac{2(\alpha_i,\lambda)}{(\alpha_i,\alpha_i)}$ for any $\lambda\in\mathfrak{h}^*$ and $i \in I$}.$$

Let $q$ be an indeterminate. For each $i \in I$, set $q_i = q^{\,\mathsf s_i}$.

\begin{definition} \label{def:qgroup}
The {\em quantum group} $U_q(\g)$ associated with a Cartan datum
$(A,P,\Pi,P^{\vee}, \Pi^{\vee})$ is the  algebra  over
$\mathbb Q(q)$ generated by $e_i,f_i$ $(i \in I)$ and
$q^{h}$ $(h \in P^\vee)$ satisfying following relations:
\begin{equation*}
\begin{aligned}
& q^0=1,\ q^{h} q^{h'}=q^{h+h'} \ \ \text{for} \ h,h' \in P,\\
& q^{h}e_i q^{-h}= q^{\lan h, \alpha_i\ran} e_i, \ \
          \ q^{h}f_i q^{-h} = q^{-\lan h, \alpha_i\ran} f_i \ \ \text{for} \ h \in P^\vee, i \in
          I, \\
& e_if_j - f_je_i = \delta_{ij} \dfrac{K_i -K^{-1}_i}{q_i- q^{-1}_i
}, \ \ \mbox{ where } K_i=q^{\mathsf s_i h_i}, \\
& \sum^{1-a_{ij}}_{r=0} (-1)^r \left[\begin{matrix}1-a_{ij}
\\ r\\ \end{matrix} \right]_i e^{1-a_{ij}-r}_i
         e_j e^{r}_i =0 \quad \text{ if } i \ne j, \\
& \sum^{1-a_{ij}}_{r=0} (-1)^r \left[\begin{matrix}1-a_{ij}
\\ r\\ \end{matrix} \right]_i f^{1-a_{ij}-r}_if_j
        f^{r}_i=0 \quad \text{ if } i \ne j.
\end{aligned}
\end{equation*}
\end{definition}

Here, we set $[n]_i =\dfrac{ q^n_{i} - q^{-n}_{i} }{ q_{i} - q^{-1}_{i} },\quad
  [n]_i! = \prod^{n}_{k=1} [k]_i$ and
  $\left[\begin{matrix}m \\ n\\ \end{matrix} \right]_i= \dfrac{ [m]_i! }{[m-n]_i! [n]_i! }\;$
  for each $m,n \in \Z_{\ge 0}$, $i \in I$.

We have a comultiplication
$\Delta \colon U_q(\g) \rightarrow U_q(\g) \otimes U_q(\g) $
given by
 \begin{align*}
 & \Delta (q^h) = q^h \otimes q^h, &
 & \Delta (e_i) = e_i \otimes K_i^{-1} + 1 \otimes e_i, &
 & \Delta (f_i) = f_i \otimes 1 + K_i \otimes f_i.
  \end{align*}

Let $U_q^{+}(\g)$ (resp.\ $U_q^{-}(\g)$) be the subalgebra of
$U_q(\g)$ generated by $e_i$'s (resp.\ $f_i$'s), and let $U^0_q(\g)$
be the subalgebra of $U_q(\g)$ generated by $q^{h}$ $(h \in
P^{\vee})$. Then we have the \emph{triangular decomposition}
$$ U_q(\g) \simeq U^{-}_q(\g) \otimes U^{0}_q(\g) \otimes U^{+}_q(\g),$$
and the {\em weight space decomposition}
$$U_q(\g) = \bigoplus_{\beta \in \rootl} U_q(\g)_{\beta},$$
where $U_q(\g)_{\beta}\seteq\set{ x \in U_q(\g)}{\text{$q^{h}x q^{-h}
=q^{(h, \beta )}x$ for any $h \in P$}}$.

Let $\A= \Z[q, q^{-1}]$ and set
$$e_i^{(n)} = e_i^n / [n]_i!, \quad f_i^{(n)} =
f_i^n / [n]_i! \ \ (n \in \Z_{\ge 0}).$$
We define the $\A$-form
$U_{\A}(\g)$ to be the $\A$-subalgebra of $U_q(\g)$ generated by
$e_i^{(n)}$, $f_i^{(n)}$ $(i \in I, n \in \Z_{\ge 0})$, $q^h$ ($h\in P^\vee$).
Let $U_{\A}^{+}(\g)$ (resp.\ $U_{\A}^{-}(\g)$) be the
$\A$-subalgebra of $U_q(\g)$ generated by $e_i^{(n)}$ (resp.\
$f_i^{(n)}$) for $i\in I$, $n \in \Z_{\ge 0}$.

\subsection{\KLRs\ }
 \hfill

Now we recall the definition of \KLRs\ associated with a given
Cartan datum $(A, P, \Pi, P^{\vee}, \Pi^{\vee})$.

Let $\cor$ be a commutative ring.
For $i,j\in I$ such that $i\not=j$, set
$$S_{i,j}=\set{(p,q)\in\Z_{\ge0}^2}{(\al_i , \al_i)p+(\al_j , \al_j)q=-2(\al_i , \al_j)}.
$$
Let us take  a family of polynomials $(Q_{ij})_{i,j\in I}$ in $\cor[u,v]$
which are of the form
\begin{equation} \label{eq:Q}
Q_{ij}(u,v) = \begin{cases}\hs{5ex} 0 \ \ & \text{if $i=j$,} \\
\sum\limits_{(p,q)\in S_{i,j}}
t_{i,j;p,q} u^p v^q\quad& \text{if $i \neq j$}
\end{cases}
\end{equation}
with $t_{i,j;p,q}\in\cor$. We assume that
they satisfy $t_{i,j;p,q}=t_{j,i;q,p}$ (equivalently, $Q_{i,j}(u,v)=Q_{j,i}(v,u)$) and
$t_{i,j:-a_{ij},0} \in \cor^{\times}$.

We denote by
$\sym_{n} = \langle s_1, \ldots, s_{n-1} \rangle$ the symmetric group
on $n$ letters, where $s_i\seteq (i, i+1)$ is the transposition of $i$ and $i+1$.
Then $\sym_n$ acts on $I^n$ by place permutations.

For $n \in \Z_{\ge 0}$ and $\beta \in \rootl_+$ such that $|\beta| = n$, we set
$$I^{\beta} = \set{\nu = (\nu_1, \ldots, \nu_n) \in I^{n}}%
{ \alpha_{\nu_1} + \cdots + \alpha_{\nu_n} = \beta }.$$

\begin{definition}
For $\beta \in \rootl^+$ with $|\beta|=n$, the {\em
Khovanov-Lauda-Rouquier algebra}  $R(\beta)$  at $\beta$ associated
with a Cartan datum $(A,P, \Pi,P^{\vee},\Pi^{\vee})$ and a matrix
$(Q_{ij})_{i,j \in I}$ is the  algebra over $\cor$
generated by the elements $\{ e(\nu) \}_{\nu \in  I^{\beta}}$, $
\{x_k \}_{1 \le k \le n}$, $\{ \tau_m \}_{1 \le m \le n-1}$
satisfying the following defining relations:
\begin{equation*} \label{eq:KLR}
\begin{aligned}
& e(\nu) e(\nu') = \delta_{\nu, \nu'} e(\nu), \ \
\sum_{\nu \in  I^{\beta} } e(\nu) = 1, \\
& x_{k} x_{m} = x_{m} x_{k}, \ \ x_{k} e(\nu) = e(\nu) x_{k}, \\
& \tau_{m} e(\nu) = e(s_{m}(\nu)) \tau_{m}, \ \ \tau_{k} \tau_{m} =
\tau_{m} \tau_{k} \ \ \text{if} \ |k-m|>1, \\
& \tau_{k}^2 e(\nu) = Q_{\nu_{k}, \nu_{k+1}} (x_{k}, x_{k+1})
e(\nu), \\
& (\tau_{k} x_{m} - x_{s_k(m)} \tau_{k}) e(\nu) = \begin{cases}
-e(\nu) \ \ & \text{if} \ m=k, \nu_{k} = \nu_{k+1}, \\
e(\nu) \ \ & \text{if} \ m=k+1, \nu_{k}=\nu_{k+1}, \\
0 \ \ & \text{otherwise},
\end{cases} \\
& (\tau_{k+1} \tau_{k} \tau_{k+1}-\tau_{k} \tau_{k+1} \tau_{k}) e(\nu)\\
& =\begin{cases} \dfrac{Q_{\nu_{k}, \nu_{k+1}}(x_{k},
x_{k+1}) - Q_{\nu_{k}, \nu_{k+1}}(x_{k+2}, x_{k+1})} {x_{k} -
x_{k+2}}e(\nu) \ \ & \text{if} \
\nu_{k} = \nu_{k+2}, \\
0 \ \ & \text{otherwise}.
\end{cases}
\end{aligned}
\end{equation*}
\end{definition}

The above relations are homogeneous provided with
\begin{equation*} \label{eq:Z-grading}
\deg e(\nu) =0, \quad \deg\, x_{k} e(\nu) = (\alpha_{\nu_k}
, \alpha_{\nu_k}), \quad\deg\, \tau_{l} e(\nu) = -
(\alpha_{\nu_l} , \alpha_{\nu_{l+1}}),
\end{equation*}
and hence $R( \beta )$ is a $\Z$-graded algebra.

For an element $w$ of the symmetric group $\sym_n$,
let us choose  a  reduced expression $w=s_{i_1}\cdots s_{i_\ell}$, and
set $$\tau_w=\tau_{i_1}\cdots \tau_{i_\ell}.$$
In general, it depends on the choice of
reduced expressions of $w$ (see Lemma~\ref{lem:ind_red}).
Then we have
$$R(\beta)=\soplus_{\nu\in  I^{\beta} , \;w\in\sym_n}\cor[x_1,\ldots, x_n]e(\nu)\tau_w.$$

 For a graded $R(\beta)$-module $M=\bigoplus_{k \in \Z} M_k$, we define
$qM =\bigoplus_{k \in \Z} (qM)_k$, where
 \begin{align*}
 (qM)_k = M_{k-1} & \ (k \in \Z).
 \end{align*}
We call $q$ the \emph{grading-shift functor} on the category of
graded $R(\beta)$-modules.

We denote by $R(\beta)\gproj$
the category of finitely generated projective graded $R(\beta)$-modules and
denote by $R(\beta)\gmod$
the category of graded $R(\beta)$-modules
which are finite-dimensional over $\cor$.
We also denote by  $R(\beta)\smod$ the category of ungraded $R(\beta)$-modules which are
finite-dimensional over $\cor$.

For each $n \in \Z_{\ge 0}$, set
\eqn
&\displaystyle \Mod(R(n)) \seteq \bigoplus_{{\substack{\beta \in \rootl^+, \\ |\beta|=n}}} \Mod(R(\beta)),
&\displaystyle \Modg(R(n))\seteq \bigoplus_{{\substack{\beta \in \rootl^+, \\ |\beta|=n}}} \Modg(R(\beta)),\\
&\displaystyle R(n)\proj \seteq \bigoplus_{{\substack{\beta \in \rootl^+, \\ |\beta|=n}}} R(\beta) \proj,
&\displaystyle R(n)\gproj \seteq \bigoplus_{{\substack{\beta \in \rootl^+, \\ |\beta|=n}}} R(\beta) \gproj,\\
&\displaystyle R(n)\smod \seteq \bigoplus_{{\substack{\beta \in \rootl^+, \\ |\beta|=n}}} R(\beta) \smod,
&\displaystyle R(n)\gmod \seteq \bigoplus_{{\substack{\beta \in \rootl^+, \\ |\beta|=n}}} R(\beta) \gmod.
\eneqn
We sometimes say that an object of $\Mod(R(n))$ is an $R(n)$-module.
Similarly, we say that an object of $\soplus_{\beta \in \rootl^+} \Mod(R(\beta)) $ is an $R$-module, etc.

\medskip
For $\beta, \gamma \in \rootl^+$ with $|\beta|=m$, $|\gamma|= n$,
 set
$$e(\beta,\gamma)=\displaystyle\sum_{\substack{\nu \in I^{\beta+\gamma}, \\ (\nu_1, \ldots ,\nu_m) \in I^{\beta}}} e(\nu) \in R(\beta+\gamma). $$
Then $e(\beta,\gamma)$ is an idempotent.
Let
\eq R( \beta)\tens R( \gamma  )\to e(\beta,\gamma)R( \beta+\gamma)e(\beta,\gamma) \label{eq:embedding}
\eneq
be the $\cor$-algebra homomorphism given by
$e(\mu)\tens e(\nu)\mapsto e(\mu*\nu)$ ($\mu\in I^{\beta}$
and $\nu\in I^{\gamma}$)
$x_k\tens 1\mapsto x_ke(\beta,\gamma)$ ($1\le k\le m$),
$1\tens x_k\mapsto x_{m+k}e(\beta,\gamma)$ ($1\le k\le n$),
$\tau_k\tens 1\mapsto \tau_ke(\beta,\gamma)$ ($1\le k<m$),
$1\tens \tau_k\mapsto \tau_{m+k}e(\beta,\gamma)$ ($1\le k<n$).
Here $\mu*\nu$ is the concatenation of $\mu$ and $\nu$;
i.e., $\mu*\nu=(\mu_1,\ldots,\mu_m,\nu_1,\ldots,\nu_n)$.

For $a\in R(\beta)$ and $b\in R(\gamma)$, we denote by $a\etens b$
the image of $a\tens b$ under this homomorphism. Hence for example,
$\tau_1e(\beta)\etens\tau_1e(\gamma)
=\tau_1\tau_{m+1}e(\beta,\gamma)$. The whole image of $R(
\beta)\tens R( \gamma  )\to e(\beta,\gamma)R(
\beta+\gamma)e(\beta,\gamma)$ is denoted by $R(\beta,\gamma )$.

\medskip
For a graded $R(\beta)$-module $M$ and a graded $R(\gamma)$-module $N$,
we define the \emph{convolution product}
$M\circ N$ by
$$M\circ N=R(\beta + \gamma) e(\beta,\gamma)
\tens_{R(\beta )\otimes R( \gamma)}(M\otimes N). $$
We sometimes denote by $M\etens N$ the module
$M\tens N$ regarded as an
$R(\beta,\gamma)$-submodule of $M\circ N$.

For $M \in R( \beta) \smod$, the dual space
$$M^* \seteq \Hom_\cor (M, \cor)$$
admits an $R(\beta)$-module structure via
\begin{align*}
(r \cdot  f)(u) \seteq f(\psi(r) u) \quad (r \in R( \beta), \ u \in M),
\end{align*}
where $\psi$ denotes the $\cor$-algebra anti-involution on $R(\beta
)$ which fixes the generators $e(\nu)$, $x_m$ and $\tau_k$ for $\nu
\in I^{\beta}, 1 \leq m \leq  |\beta|$ and $1 \leq k \leq
|\beta|-1$.

It is known that (see \cite[Theorem 2.2 (2)]{LV11})
\eq
&&(M_1 \circ M_2)^* \simeq q^{(\beta,\gamma)}
(M_2^* \circ M_1 ^*)\label{eq:dualconv}
\eneq
for any $M_1 \in R(\beta) \gmod$ and $M_2 \in R(\gamma) \gmod$.

\medskip
Let us denote by $K(\R(\beta)\gproj)$ and $K(\R(\beta)\gmod)$ the corresponding Grothendieck groups. Then,
$\soplus_{\beta \in \rootl^+} K(R(\beta)\gproj)$ and
$\soplus_{\beta \in \rootl^+} K(R(\beta)\gmod)$
are $\A$-algebras with
 the multiplication induced by the convolution product and
the $\A$-action induced by the  grading shift functor $q$.

 In \cite{KL09, R08}, it is shown that
\KLRs\ \emph{categorify} the negative half of the corresponding quantum group. More precisely, we have the following theorem.

\begin{theorem}[{\cite{KL09, R08}}] \label{Thm:categorification}
 For a given  Cartan datum $(A,P, \Pi,P^{\vee},\Pi^{\vee})$,
we take a parameter matrix $(Q_{ij})_{i,j \in J}$ satisfying the conditions in \eqref{eq:Q}, and let $U_q(\mathfrak g)$ and $R(\beta) $ be
the associated quantum group   and \KLR, respectively.
Then there exists an $\A$-algebra isomorphism
\begin{align}
  U^-_\A(\mathfrak g) \simeq \bigoplus_{\beta \in \rootl^+} K(R(\beta) \gproj).
\end{align}
By duality, we have
\begin{align}
  U^-_\A(\mathfrak g)^{\vee} \simeq \bigoplus_{\beta \in \rootl^+} K(R(\beta) \gmod),
\end{align}
where $U^-_\A(\mathfrak g)^{\vee} \seteq \{x \in U_q^-(\g) \,|\, (x, U_\A^-(\g))_{-}\subset \A\}$
and $(\quad,\quad)_{-}$ denotes the non-degenerate bilinear form on $U^-_q(\g)$ defined in \cite[\S\;3]{Kas91}.

\end{theorem}

 The \KLRs\ also categorify the global bases.
For the definition of global bases, see \cite{Kas91}.

\begin{theorem} [\cite{VV09, R11}] \label{thm:categorification 2}
Assume that $A$ is symmetric and that $Q_{i,j}(u,v)$ is a polynomial in $u-v$.
Assume further that $\cor$ is a field of characteristic $0$.
 Then under the isomorphism in {\rm Theorem \ref{Thm:categorification}},
the lower global basis \ro respectively, upper global basis\rf\ corresponds to
the set of isomorphism classes of indecomposable projective modules
\ro respectively, the set of isomorphism classes of simple modules\rf.
\end{theorem}

\subsection{R-matrices for \KLRs}
\subsubsection{Intertwiners}\label{subsec:int}
 For $|\beta|=n$ and $1\le a<n$,  we define $\vphi_a\in R( \beta)$
by
\eq&&\ba{l}
  \vphi_a e(\nu)=
\begin{cases}
  \bl\tau_ax_a-x_{a}\tau_a\br e(\nu)\\
\hs{4ex}= \bl x_{a+1}\tau_a-\tau_ax_{a+1}\br e(\nu)\\
\hs{8ex}  =\bl\tau_a(x_a-x_{a+1})+1\br e(\nu)\\
\hs{12ex}=\bl(x_{a+1}-x_{a})\tau_a-1\br e(\nu)
 & \text{if $\nu_a=\nu_{a+1}$,} \\[2ex]
\tau_ae(\nu)& \text{otherwise.}
\end{cases}
\label{def:int} \ea \eneq
They are called the {\em intertwiners}.

The following lemma is well-known (for example, it easily follows
 from  the polynomial representation of \KLRs\
(\cite[Proposition~2.3]{KL09}, \cite[Proposition~3.12]{R08}). \Lemma
\label{lem:ga} \hfill
\bnum
\item $\vphi_a^2e(\nu)=\bl
Q_{\nu_a,\nu_{a+1}}(x_a,x_{a+1})+\delta_{\nu_a,\nu_{a+1}}\br e(\nu)$.
\item
$\{\vphi_k\}_{1\le k<n}$ satisfies the braid relation.
\item
 For $w\in \sym_n$, let $w=s_{a_1}\cdots s_{a_\ell}$ be a reduced expression of $w$ and
set $\vphi_w=\vphi_{a_1}\cdots\vphi_{a_\ell}$.
Then $\vphi_w$ does not depend on the choice of reduced expressions of $w$.

\item For $w\in \sym_n$ and $1\le k\le n$, we have $\vphi_w x_k=x_{w(k)}\vphi_w$.
\item
For $w\in \sym_n$ and $1\le k<n$,
if $w(k+1)=w(k)+1$, then $\vphi_w\tau_k=\tau_{w(k)}\vphi_w$.\label{ga5}
\item
$\vphi_{w^{-1}}\vphi_we(\nu)=\prod\limits_{ \substack{a<b,\\ w(a)>w(b)} }
(Q_{\nu_a,\nu_b}(x_a,x_b)+\delta_{\nu_a,\nu_b})e(\nu)$. \label{ga6}
\ee
\enlemma

We define another
symmetric bilinear form $\inp{{\scbul,\scbul}}$ on $\rtl$ by
\eqn
&&\inp{\al_i,\al_j}=\delta_{ij}.
\eneqn
Then the intertwiner $\vphi_we(\nu)$ has degree
 $$\hs{-1ex}\sum\limits_{\substack{1\le a<b\le n,\\ w(a)>w(b)}}
\bl-(\al_{\nu_a},\al_{\nu_b})+2\inp{\al_{\nu_a},\al_{\nu_b}}\br.$$

For $m,n\in\Z_{\ge0}$,
let us denote by $w[{m,n}]$ the element of $\sym_{m+n}$  defined by
\eq
&&w[{m,n}](k)=\begin{cases}k+n&\text{if $1\le k\le m$,}\\
k-m&\text{if $m<k\le m+n$.}\end{cases}
\eneq

Let $\beta,\gamma\in \rtl^+$ with $|\beta|=m$, $|\gamma|=n$ and let
$M$ be an $R(\beta)$-module and $N$ an $R(\gamma)$-module. Then the
map
$$M\tens N\to q^{(\beta,\gamma)-2\inp{\beta,\gamma}}N\circ M$$ given by
$$u\tens v\longmapsto \vphi_{w[n,m]}(v\tens u)$$
is  $R( \beta,\gamma )$ -linear by the above lemma, and it extends
to an $R( \beta +\gamma)$- module  homomorphism \eq &&R_{M,N}\col
M\circ N\To q^{(\beta,\gamma)-2\inp{\beta,\gamma}}N\circ M. \eneq
 Then we obtain the following commutative diagrams:
$$\xymatrix@C=7ex{L\circ M\circ N\ar[r]^{R_{L,M}}\ar[dr]_{R_{L,M\circ N}}
&M\circ L\circ N\ar[d]^{R_{L.N}}\\
&M\circ N\circ L}\qtext{and}\quad
\xymatrix@C=7ex{L\circ M\circ N\ar[r]^{R_{M,N}}\ar[dr]_{R_{L\circ M,N}}
&L\circ N\circ M\ar[d]^{R_{L.N}}\\
&N\circ L\circ M\,.}
$$
Hence the homomorphisms $R_{M,N}$ satisfy the Yang-Baxter equation,
namely, the following diagram  is commutative:
\eq&&\ba{l}\xymatrix@R=2ex{
&L\circ M\circ N\ar[dl]_{R_{L,M}}\ar[dr]^{R_{M,N}}\\
M\circ L\circ N\ar[dd]_{R_{L,N}}&&L\circ  N\circ M\ar[dd]^{R_{L,N}}\\
\\
M\circ N\circ L\ar[dr]_{R_{M,N}}&&N\circ L\circ M\ar[dl]^{R_{L,M}}\\
&N\circ M\circ L \,.
}\ea\label{dia:RKLR}\eneq In the above
diagrams, we omit the grading shift.

\subsubsection{Spectral parameters}\label{subsec:spec}

\Def A \KLR\ $R( \beta )$ is {\em symmetric} if \eq &&\parbox{70ex}{
\bna\item the Cartan matrix  $A=(a_{ij})_{i,j\in I}$ is symmetric,
\item $(\al_i,\al_j)=a_{ij}$,
\item $Q_{i,j}(u,v)$ is a polynomial in $u-v$.
\ee}\label{cond:sym}
\eneq
\edf

\medskip
In this subsection, we assume that \KLRs\ are symmetric.
Let $z$ be an indeterminate  which is  homogeneous of degree $2$, and
let $\psi_z$ be the algebra homomorphism
\eqn
&&\psi_z\col R( \beta )\to \cor[z]\tens R( \beta )
\eneqn
given by
$$\psi_z(x_k)=x_k+z,\quad\psi_z(\tau_k)=\tau_k, \quad\psi_z(e(\nu))=e(\nu).$$

For an $R( \beta )$-module $M$, we denote by $M_z$
the $\bl\cor[z]\tens R( \beta )\br$-module
$\cor[z]\tens M$ with the action of $R( \beta )$ twisted by $\psi_z$.
Namely,
\eq&&\hs{3ex}e(\nu)(a\tens u)=a\tens e(\nu)u,\;
x_k(a\tens u)=(za)\tens u+a\tens (x_ku), \; \tau_k(a\tens u)=a\tens(\tau_k u)
\label{eq:spt1}
\eneq
for  $\nu\in I^\beta$,  $a\in \cor[z]$ and $u\in M$.
 For $u\in M$, we sometimes denote by $u_z$ the corresponding element $1\tens u$ of
the $R( \beta )$-module  $M_z$.

Observe that this construction is possible only when the \KLR\ is symmetric.

 We can reformulate  the above construction as follows.
Let $R( \beta )_z\seteq\bl\cor[z]\tens R( \beta )\br \one_z$ be the left free
$\cor[z]\tens R( \beta )$-module generated by $\one_z$.
We define the $(\cor[z]\tens R( \beta ),R( \beta ))$
-bimodule structure on $R( \beta )_z$
by
$$\one_ze(\nu)=e(\nu)\one_z,\quad\one_zx_k=(x_k-z)\one_z,
\quad \one_z\tau_k=\tau_k\one_z.$$
Then we have
$$M_z\simeq R( \beta )_z\tens_{R( \beta )}M.$$
The element $a\tens u$ in \eqref{eq:spt1} corresponds
to $a\one_z\tens u$ under this notation.

\subsection{R-matrices with spectral parameters}

\Lemma Let $\Delta^\pm$ be the set of positive
\ro resp.\ negative\rf\  roots associated
with a generalized Cartan matrix, as in {\rm\S\,\ref{subsec:Cartan}}.
Let $\set{\al_i}{i\in I}$ be the set of simple roots,
$W$ the Weyl group, and $s_i\in W$ the simple reflection.
Let $G$ be the monoid generated by $\ts_i$ with the defining relation:
\eq&&\text{$\ts_i\ts_j= \ts_j\ts_i   $ if  $(\al_i,\al_j)=0$.}\eneq
Assume that an element $w$ of $W$ satisfies the following conditions:
\eq&&
 \text{if $\al,\beta\in\Delta^+$, $\al \not=\beta$, and $w\al,\;w\beta\in \Delta^-$, then ${(\alpha,\beta)}\ge0$.}
\label{cond:tauw0}
\eneq
Then $\ts_{i_1}\cdots \ts_{i_\ell}\in G$ does not depend on the choice
of a reduced expression  $w=s_{i_1}\cdots s_{i_\ell}$.
\enlemma
\Proof
(a) We shall first prove
\eq&&
\text{if $w$ satisfies \eqref{cond:tauw0} and
$s_iw<w$, then $s_iw$ also satisfies \eqref{cond:tauw0}.}
\eneq
Set $w'\seteq  s_iw$.
Assume that $\al,\beta\in\Delta^+$, $w'\al,w'\beta\in \Delta^-$
and $(\alpha,\beta)<0$.
Then  either  $w\al$ or $w\beta$ belongs to
$\Delta^+$, say, $w\al\in \Delta^+$.
Then $w\al=\al_i$, and hence $w^{-1}(\al_i)\in\Delta^+$, which contradicts
$s_i w<w$.

\medskip\noi
(b)\quad  We will use  induction on the length of $w$.
Let $w=s_{i_1}\cdots s_{i_\ell}$ and $w=s_{j_1}\cdots s_{j_\ell}$ be
reduced expressions of $w$.
If $i_1=j_1$, then applying the induction hypothesis to
$s_{i_1}w$, we obtain the desired result.

Assume $i_1\not=j_1$.
Then $\al\seteq-w^{-1}(\al_{i_1})$ and $\beta\seteq-w^{-1}(\al_{j_1})$
belong to $\Delta^+$ and $w\al, w\beta$ belong to
$\Delta^-$.
Hence we have
$(\al_{i_1},\al_{j_1}) \ge 0 $ , which implies that $(\al_{i_1},\al_{j_1})=0$.

Note that $s_{i_1}w$ and $s_{j_1}w$ satisfy \eqref{cond:tauw0} by (a).
Set $w'=s_{i_1}s_{j_1}w$.
Then we have $w'<s_{j_1} w<w$ because $(s_{j_1}w)^{-1}\al_{i_1}=
w^{-1}\al_{i_1}\in\Delta^-$.
Taking a reduced expression $w'=s_{k_3}\cdots s_{k_\ell}$,
and applying the induction hypothesis to $s_{i_1}w$ and $s_{j_1}w$, we conclude
$$\ts_{i_1}\ts_{i_2}\cdots\ts_{i_\ell}=
\ts_{i_1}\ts_{j_1}\ts_{k_3}\cdots \ts_{k_\ell}
=\ts_{j_1}\ts_{i_1}\ts_{k_3}\cdots \ts_{k_\ell}
=\ts_{j_1}\ts_{j_2}\cdots \ts_{j_\ell}.
$$
\QED

 We apply the above lemma  to the case of the symmetric group.
\Lemma\label{lem:ind_red}
Assume that $w\in \sym_n$ satisfies the condition
\eq
&&\hs{10ex}\parbox{70ex}{there exists no triple of integers $(i,j,k)$ such that
$$\text{$1\le i<j<k\le n$ and $w(i)>w(j)>w(k)$.}$$}\label{cond:tauw}
\eneq
Then $\tau_w\seteq\tau_{i_1}\cdots \tau_{i_\ell}\in R( \beta )$ does
not depend on the choice of  a reduced expression
 of $w=s_{i_1}\cdots s_{i_\ell}$.
\enlemma

 For $m,n\in\Z_{\ge0}$,
let
\eqn
&&\sym_{m,n} \seteq \set{w\in\sym_{m+n}}{\text{ $w(i)<w(i+1)$ for any $i\not=m$}}.
\eneqn

This condition is equivalent to saying that $w\,\vert_{[1,m]}$
and $w\,\vert_{[m+1,m+n]}$ are increasing.
It is easy to see that
any element of $\sym_{m,n}$ satisfies the condition
\eqref{cond:tauw}. Hence we obtain the following corollary
(see also \cite[Lemma 3.17]{KMR12}).
\Cor
 Let $|\beta|=m$ and $|\gamma|=n$.
 For $w\in\sym_{m,n}$,
the product $\tau_w\seteq\tau_{i_1}\cdots \tau_{i_\ell}\in R( \beta+\gamma)$ does
not depend on the choice of reduced expression
$w=s_{i_1}\cdots s_{i_\ell}$.
\encor

Hence $\tau_w\in R( \beta+\gamma)$ is well-defined for $w\in\sym_{m,n}$.
Since $w[{m,n}]$ belongs to $\sym_{m,n}$, we set
\eq
&&\tau_{m,n}=\tau_{w[{m,n}]}.
\eneq
Note that
$$\deg \tau_{m,n}  e(\beta, \gamma)  =-(\beta,\gamma)
\quad\text{if $|\beta|=m$ and $|\gamma|=n$.}$$

\Prop \label{prop:R}
Let $\beta,\gamma \in \rootl^+$ with $|\beta|=m$ and $|\gamma|=n$,
and
 let $M, N$ be an $R(\beta)$-module and an $R(\gamma)$-module, respectively.
Suppose we have  algebraically independent  indeterminates  $z$ and $z'$.
Then \bnum
\item
$N_{z'}\circ M_z=\soplus_{w\in\sym_{n,m}}\cor[z,z']\tau_w(\one_{z'}\tens\one_{z}\tens N\tens M)$.
\label{item:w}
\item
We have
$$R_{M_z,N_{z'}}(M\tens N) \ \subset \ \soplus_{w\in\sym_{n,m}}\cor[z-z']
\tau_w(\one_{z'}\tens\one_{z}\tens N\tens M).$$
\item  Set $\ell=\inp{\beta,\gamma}$.
Then \eqn
&&R_{M_z,N_{z'}}(u_z\tens v_{z'})- (z'-z)^\ell\tau_{w[n,m]}(v_{z'}\tens u_z)\\
&&\hs{10ex}\in \sum_{s<\ell}(z'-z)^s\tau_{w[n,m]}
(\one_{z'}\tens\one_{z}\tens N\tens M)\\
&&\hs{20ex}
+\sum_{ \substack{w\in\sym_{n,m}, \\ w\not=w[n,m]} }\hs{-2ex}
\cor[z'-z]\tau_w(\one_{z'}\tens\one_{z}\tens N\tens M)
\eneqn
for $u\in M$ and $v\in N$.\label{item:Rw}
\item
Set $$A=\sum_{\mu\in I^\beta, \nu\in I^\gamma}\Bigl(
\prod\limits_{\substack{1\le a\le m, 1\le b\le n\\\mu_a\not=\nu_b}}
 Q_{\mu_a,\nu_b}\bl x_a\etens  e(\gamma), e(\beta)  \etens x_b\br\Bigr)e(\mu)\etens e(\nu)
\in R(\beta)\etens R(\gamma).$$
Then $A$ belongs to the center of $R(\beta)\etens R(\gamma)$, and
$$R_{N,M}R_{M,N}(u\tens v)=A(u\tens v)\qtext{for $u\in M$ and $v\in N$.}$$
\ee
\enprop
\Proof
 By induction on the length of $w$, one can prove
\eq
\vphi_we(\nu)-\tau_w\prod_{(i,j)\in  B }(x_i-x_j)e(\nu)
\in\sum_{w'<w}\tau_{w'}\cor[x_1,\ldots, x_n]\label{eq:phitau}
\eneq
for $w\in\sym_n$ and $\nu\in I^n$, where
$$ B  =\set{(i,j)}{1\le i<j\le n, \, w(i)>w(j), \,  \nu_i = \nu_j }.$$
Now it is easy to see that the statement (iii) follows from
\eqref{eq:phitau}.

The other statements easily follow from Lemma~\ref{lem:ga}. \QED

Hereafter we assume that
\eq
&&\text{the base ring $\cor$ is a field.}
\eneq

For a non-zero $R(\beta)$-module $M$ and a non-zero $R(\gamma)$-module $N$,
\eq&&\parbox{70ex}{%
let $s$ be the order of zeroes of $R_{M_z,N_{z'}}\col  M_z\circ
N_{z'}\To q^{(\beta,\gamma)-2\inp{\beta,\gamma}}N_{z'}\circ M_z$;
i.e., the largest non-negative integer such that the image of
$R_{M_z,N_{z'}}$ is contained in $(z'-z)^s
q^{(\beta,\gamma)-2\inp{\beta,\gamma}}N_{z'}\circ M_z$.}
\label{def:s} \eneq Note that Proposition \ref{prop:R} (iii) shows
that
 such an $s$ exists   and
 $s\le \inp{\beta,\gamma}$.

\Def For a non-zero $R(\beta)$-module $M$ and a non-zero $R(\gamma)$-module
$N$,
we set $$\ds{M,N}\seteq(\beta,\gamma)-2\inp{\beta,\gamma}+2s,$$ and define
$$\rmat{M,N}\col M\circ N\to q^{\ds{M,N}}N\circ M$$
by  $$\rmat{M,N} = \bl (z'-z)^{-s}R_{M_z,N_{z'}}\br\vert_{z=z'=0}.$$
If $M$ or $N$ vanishes, then we set
 \eqn
s=0, \ \ds{M,N}=(\beta,\gamma)-2\inp{\beta,\gamma}, \ \text{and} \ \rmat{M,N}=0.
\eneqn
\edf
 By Proposition \ref{prop:R} (ii), the morphism $\rmat{M,N}$ does not vanish
if $M$ and $N$ are non-zero.

Note that if $R_{M,N}$ does not vanish, then $s=0$ and $\rmat{M,N}=R_{M,N}$.

\Cor\label{cor:int}
If $M$ and $N$ are non-zero, then there exists a non-zero homomorphism
$M\circ N\to N\circ M$. Here we omit the grading. 
\encor

\Rem
We don't know if Corollary \ref{cor:int} holds when
the \KLR\ is not symmetric.
\enrem

Note that the homomorphisms $\rmat{M,N}$ satisfy the Yang-Baxter
equation. Namely, if $L$ is another $R(\gamma)$-module, then the
following diagram is commutative:
$$\xymatrix@R=1.7ex@C=7ex{
&L\circ M\circ N\ar[dl]_{\rmat{L,M}}\ar[dr]^{\rmat{M,N}}\\
M\circ L\circ N\ar[dd]_{\rmat{L,N}}&&L\circ  N\circ M\ar[dd]^{\rmat{L,N}}\\
\\
M\circ N\circ L\ar[dr]_{\rmat{M,N}}&&N\circ L\circ M\ar[dl]^{\rmat{L,M}}\\
&N\circ M\circ L \,. }$$ Here, we omit the grading shift operators.
This immediately follows from the corresponding result for
$R_{M,N}$.

For $\beta_1,\ldots,\beta_t\in\rtl^+$,
a sequence of $R(\beta_k)$-modules $M_k$ ($k=1,\ldots,t$) and $w\in \sym_t$,
we set $d=\sum\limits\ds{M_i,M_{j}}$,
where the summation ranges over the set
$$\set{(i,j)}{1\le i<j\le t, w(i)>w(j)}.$$
We define \eq &&\rmat{M_1,\ldots,M_t}^{\quad w} =\rmat{\{M_s\}_{1\le
s\le t}}^{\quad w}\col M_1\circ \cdots\circ M_t \to q^dM_{w(1)}\circ
\cdots\circ M_{w(t)} \eneq by induction on the length of $w$ as
follows: \eqn &&\hs{-3.5ex}\rmat{\{M_a\}_{1\le a\le t}}^{\ w}
=\begin{cases}
\id_{M_1\circ \cdots\circ M_t}&\text{if $w=e$,}\\[2ex]
\ba{l}\hs{-1ex}\rmat{\{M_{s_k(a)}\}_{1\le a\le t}}^{\ ws_k}\\
\hs{2ex}\cdot\bl M_{1}\circ \cdots \circ M_{k-1}\circ
\rmat{M_{k}, M_{k+1}}\circ  M_{k+2}\circ\cdots\circ M_{t}\br\ea
&\text{if $w(k)>w(k+1)$.}
\end{cases}
\eneqn
Then it does not depend on the choice of $k$ and
$\rmat{M_1,\ldots,M_t}^{\ w}$ is well-defined.

Similarly, we define
\eq
R_{M_1,\ldots, M_t}^w\col
 M_1\circ \cdots \circ M_t
\to q^bM_{w(1)}\circ \cdots \circ M_{w(t)}, \eneq where
$b=\sum\limits_{ \substack{1\le k<k'\le t,\\ w(k)>w(k')} }
(\beta_k,\beta_{k'})-2\inp{\beta_k,\beta_{k'}}$.

We set \eq &&\rmat{M_1,\ldots,
M_t}\seteq\rmat{M_1,\ldots,M_t}^{\hs{2ex} w_t} \quad \text{and}
\quad R_{M_1,\ldots, M_t}\seteq R_{M_1,\ldots, M_t}^{\hs{2ex} w_t},
\eneq where $w_t$ is the longest element of $\sym_t$.

The following lemma is obvious.
\Lemma\label{lem:Rr}
Let $M_k$ $(1\le k\le t)$
be $R(\beta_k)$-modules and   $N_{ k'}$ $(1\le  k'  \le t')$
be $R(\gamma_{k'})$-modules for some $\beta_k, \gamma_{k'} \in \rootl^+$.

Let  $\widetilde{M}\seteq M_1\circ\cdots\circ M_t$
and $\widetilde{N}\seteq N_1\circ\cdots\circ N_{t'}$.
 Suppose we have epimorphisms
 \eqn \widetilde{M} \epito M \ \text{and} & \widetilde{N} \epito N. \eneqn

Then we have
\bnum
\item
$R_{\widetilde{M},\widetilde{N}}=R^{w[{t,t'}]}_{M_1,\ldots,M_t,N_1,\ldots, N_{t'}}$,
\item
there exists a unique morphism
$\vphi\col M\circ N\to N\circ M$ such that, the diagram
$$\xymatrix@C=15ex{
{M_1\circ\cdots\circ M_t\circ N_1\circ\cdots\circ N_{t'}}
\ar[r]^{r^{w[{t,t'}]}_{M_1,\ldots,M_t,N_1,\ldots, N_{t'}}}\ar@{->>}[d]
&
{N_1\circ\cdots\circ N_{t'}\circ M_1\circ\cdots\circ M_t}\ar@{->>}[d]\\
M\circ N\ar[r]^{\vphi}&N\circ M}
$$
is commutative.
\item
Moreover, if $\vphi$ does not vanish, then we have
$$\text{$\vphi=\rmat{M,N}$ \ and \ \
$\ds{M,N}=\sum_{1\le k\le t,\,1\le k'\le t'}\ds{M_k,N_{k'}}$.}$$
\ee
\enlemma

The following proposition will be used later.
Note that the grading plays a crucial role here.

\Prop\label{prop:nonvan}
Let $\beta_k\in \rtl^+$ and
let $M_k$ a {\em non-zero} module in $R(\beta_k)\gmod$
 $(1\le k\le t)$.
For $w\in\sym_t$,
set
$d=\sum\limits_{\substack{1\le k<k'\le t,\\ w(k)>w(k')}} \ds{M_k,M_{k'}}$.
Assume that
\eq&&
\hs{5ex}\Hom_{R(\beta)\gmod}(M_1\circ \cdots  \circ  M_t,\,
q^cM_{ w(1)}\circ \cdots  \circ   M_{ w(t)})=0
\ \text{for  all  $c>d$,}
\label{cond:hom}
\eneq
 where $\beta=\sum\limits_{k=1}^t \beta_k$.
Then the homomorphism
$$\rmat{M_1,\ldots, M_t}^{\ w}\col M_1\circ\cdots \circ M_t
\To q^d M_{w(1)}\circ \cdots\circ M_{w(t)}$$
does not vanish.
\enprop
\Proof
Set $d_{i,j}=\ds{M_i,M_j}$,
$D_{i,j}=(\beta_i,\beta_j)-2\inp{\beta_i,\beta_j}$,
and let $s_{i,j}$ be the degree of  zeroes  of $R_{(M_i)_z,\,(M_j)_{z'}}$.
Then we have
$d_{i,j}=D_{i,j}+2s_{i,j}$ and
$$\rmat{M_i,M_j}=(z'-z)^{-s_{i,j}}R_{(M_i)_z,\,(M_j)_{z'}}\vert_{z=z'=0}.$$
We set $D=\sum_{i,j}D_{i,j}$ and $s=\sum_{i,j}s_{i,j}$.
Here,  the sum is taken over the set
$$ \set{(i,j)}{1\le i<j\le t, \,w(i)>w(j)}.$$

Let $z$ be an indeterminate and
let $a_1,\ldots,a_t$ be generic constants.
 Set $z_k=a_kz$ for $1\le k\le t$.
Consider the morphism
\eq&&\ba{l}
 R_{(M_1)_{z_1},\ldots, (M_t)_{z_t}}^{\ w}\col
(M_1)_{z_1}\circ\cdots\circ (M_t)_{z_t}\\[1ex]
\hs{20ex}\To
q^D(M_{w(1)})_{z_{w(1)}}\circ\cdots\circ
(M_{w(t)})_{z_{w(t)}}.\ea
\eneq
Let $m$ be the degree of  zeroes  of  $R_{(M_1)_{z_1},\ldots, (M_t)_{z_t}}^{\ w}$,
that is, the largest integer $m$ such that
the image of $R_{(M_1)_{z_1},\ldots, (M_t)_{z_t}}^{\ w}$ is contained in
$z^mq^D(M_{w(1)})_{z_{w(1)}}\circ\cdots\circ (M_{w(t)})_{z_{w(t)}}$.
Since $R_{(M_1)_{z_1},\ldots, (M_t)_{z_t}}^{\ w}$ is not identically zero,
such an $m$ exists.

Set $r=\bl (q^{-2}z)^{-m} R_{(M_1)_{z_1},\ldots, (M_t)_{z_t}}^{\
w}\br\vert_{z=0}$. Then $r$ is a non-zero homomorphism from
$M_1\circ \cdots\circ M_t$ to $q^{2m+D}M_{w(1)}\circ \cdots\circ
M_{w(t)}$. By the assumption \eqref{cond:hom}, we obtain $2m+D\le
d=D+2s$. Hence we conclude
$$m\le s.$$

On the other hand, for $i, j$ such that
$1\le i<j\le t$ and $w(i)>w(j)$,
the homomorphism $z^{-s_{i.j}}R_{(M_i)_{z_i}, (M_j)_{z_j}}$
is well-defined and
$r_{M_i,M_j}=\bl z^{-s_{i,j}}R_{(M_i)_{z_i},\, (M_j)_{z_j}}\br\vert_{z=0}$
up to a constant multiple.
Since $z^{-s}R_{(M_1)_{z_1},\ldots, (M_t)_{z_t}}^{\ w}$ is a product of the
$\bl z^{-s_{i,j}}R_{(M_i)_{z_i}, (M_j)_{z_j}}\br$'s, it is  well-defined,
and $\rmat{M_1,\cdots, M_t}^{\ w}$ coincides with
$\bl z^{-s}R_{(M_1)_{z_1},\ldots, (M_t)_{z_t}}^{\ w}\br\vert_{z=0}$
up to a constant multiple.
Hence we have $m\ge s$.
Therefore $m=s$ and
$\rmat{M_1,\cdots, M_t}^{\ w}$ is equal to $r$ up to a constant multiple.
\QED

\subsection{R-matrices for one-dimensional modules}\label{subsec:R1}
Let $\mu=(\mu_1,\ldots, \mu_n)$ be a sequence of elements in $I$.
Let $L(\mu)$ be a free $\cor$-module $\cor u(\mu)$ with the
generator $u(\mu)$ of degree $0$. The following lemma can be easily
verified. \Lemma By defining \eqn &&x_k u(\mu)=0,\quad
\tau_ku(\mu)=0, \quad e(\nu) u(\mu) = \delta_{\nu,\mu} u(\mu),
\eneqn $L(\mu)$ becomes an  $R(\al_{\mu_1}+\cdots+\al_{\mu_n})$
-module if and only if \eq &&\parbox{70ex}{ \bna\item
$(\al_{\mu_k},\al_{\mu_{k+1}})<0$ for $1\le k<n$,
\item if  $\mu_{k}=\mu_{k+2}$ $(1\le k\le n-2)$,
then
$\lan h_{\mu_k}, \al_{\mu_{k+1}}\ran\not=-1$.
\ee
\label{cond:oned}}
\eneq
\enlemma

Now we assume \eqref{cond:sym}; i.e., $R$ is a symmetric \KLR.
The following lemma immediately follows from the definition.
\Lemma \label{lem:expliciteR_easiest cases}
Let $\mu=(\mu_1,\ldots,\mu_m)$
and $\nu=(\nu_1,\ldots,\nu_n)$
be a pair of sequences satisfying \eqref{cond:oned}. If $\mu_i\not=\nu_j$
for any $1\le i\le m$ and $1\le j\le n$,
then
$$R_{L(\mu)_z,L(\nu)_{z'}}\bl u(\mu)_z\tens u(\nu)_{z'}\br
=\tau_{n,m}\bl u(\nu)_{z'}\tens u(\mu)_z\br.$$
\enlemma

We shall use the following lemma later in the type $A$ case.
\Lemma \label{lem:expliciteR}
 Let $i \in I$  and  let  $\mu=(\mu_1,\ldots,\mu_n)$ be a sequence satisfying \eqref{cond:oned}.
Set $\beta=\sum_{k=1}^n\al_{\mu_k}$ and $p=\inp{\beta,\al_i}$.
Then we have
\eqn
&&
\ba{l}
R_{L(\mu)_z, L(i)_{z'}}(u(\mu)_z\tens u(i)_{z'})\\[1ex]
\hs{5ex}=\bl (z'-z)^p\tau_n\cdots\tau_1+\delta_{\mu_n,i}(z'-z)^{p-1}
\tau_{n-1}\cdots \tau_1\br
\bl u(i)_{z'}\tens u(\mu)_z\br,\\[2ex]
R_{L(i)_{z}, L(\mu)_{z'}}(u(i)_{z}\tens u(\mu)_{z'})\\[1ex]
\hs{5ex}=
\bl (z'-z)^p\tau_1\cdots\tau_n+\delta_{\mu_1,i}(z'-z)^{p-1}
\tau_{2}\cdots \tau_n\br
\bl u(\mu)_{z'}\tens u(i)_{z}\br.
\ea
\eneqn
\enlemma
\Proof
We shall prove only the second formula.
We prove it by induction on $n$.
 Since $n=1$ case is obvious, assume that $n>1$.
Set $\mu'=(\mu_2,\ldots,\mu_n)$ and $j=\mu_1$.

\smallskip
First assume that $i\not=j$.
Then we have
\eqn
&&R_{L(i)_{z}, L(j)_{z'}\circ L(\mu')_{z'}}(u(i)_{z}\tens u(j)_{z'}
\tens u(\mu')_{z'})\\
&&\hs{0.3ex}=
\bl L(j)_{z'}\circ R_{L(i)_{z}, L(\mu')_{z'}}\br\cdot
\Bigl(\bl R_{L(i)_{z}, L(j)_{z'}}(u(i)_{z}\tens u(j)_{z'})\br\tens u(\mu')_{z'}
\Bigr)\\
&&\hs{0.3ex}=
\bl L(j)_{z'}\circ R_{L(i)_{z},   L(\mu')_{z'} }\br
\Bigl(\tau_1(u(j)_{z'}\tens u(i)_{z})\tens u(\mu')_{z'}\Bigr)\\
&&\hs{0.3ex}=
\tau_1\bl L(j)_{z'}\circ R_{L(i)_{z},L(\mu')_{z'} }\br
\Bigl(u(j)_{z'}\tens u(i)_{z}\tens u(\mu')_{z'}\Bigr)  \\
&&\hs{0.3ex}=
\tau_1\Bigl( u(j)_{z'}\tens
 R_{L(i)_{z}, L(\mu')_{z'}}(u(i)_{z}\tens u(\mu')_{z'})\Bigr)\\
&&\hs{0.3ex}=
\tau_1\biggl( u(j)_{z'}\tens\Bigl(\bl
(z'-z)^p\tau_1\cdots \tau_{n-1}
+\delta_{i,\mu_2}(z'-z)^{p-1}\tau_2\cdots\tau_{n-1}\br
(u(\mu')_{z'}\tens u(i)_{z})\Bigr)\biggr).
\eneqn
Hence we obtain
\eqn
&&R_{L(i)_{z}, L(\mu)_{z'}}(u(i)_{z}\tens u(\mu)_{z'})\\
&&\hs{4ex}=
\bl
(z'-z)^p\tau_1(\tau_2\cdots \tau_n)
+\delta_{i,\mu_2}(z'-z)^{p-1}\tau_1(\tau_3\cdots\tau_{n})\br
(u(\mu)_{z'}\tens u(i)_z).
\eneqn
Then the last term vanishes since
$\tau_1u(\mu)_{z'}=0$.

Now assume that $i=\mu_1$. Then $\mu_2\not=i$.
Since we have
\eqn
R_{L(i)_z, L(i)_{z'}}(u(i)_z\tens u(i)_{z'})
&=&(\tau_1(x_1-x_2)+1)(u(i)_{z'}\tens u(i)_{z})\\
&=&((z'-z)\tau_1+1)(u(i)_z\tens u(i)_{z'}), \eneqn which gives \eqn
&&R_{L(i)_{z}, L(j)_{z'}\circ L(\mu')_{z'}}(u(i)_{z}\tens u(j)_{z'}
\tens u(\mu')_{z'})\\
&&\hs{3ex}=
\bl L(j)_{z'}\circ R_{L(i)_{z},\circ L(\mu')_{z'}}\br
\Bigl(\bl(z'-z)\tau_1+1\br u(i)_{z'}\tens u(i)_{z}\tens u(\mu')_{z'}\Bigr)\\
&&\hs{5ex}=
((z'-z)\tau_1+1)
(z'-z)^{p-1}\tau_2\cdots \tau_n
(u(i)_{z'}\tens  u(\mu')_{z'}\tens u(i)_{z}).
\eneqn
Hence we obtain the desired result.
\QED

\subsection{Uniqueness of R-matrices}
The following lemma says that the convolution of
two  simple modules
(after adding spectral parameters)
remains simple.

\Lemma Assume that $\cor$ is a field. Let $z$, $z'$ be algebraically
independent indeterminates and set $K=\cor(z,z')$. Let $M$ and $N$
be a simple $R( \beta )$-module and a simple $R( \gamma )$-module,
respectively. Then $K\tens_{\cor[z, z']} (M_z\circ N_{z'})$ is a
simple $K\tens R( \beta+ \gamma )$-module. \enlemma

\Proof
In  this  proof,  we  write
 $M_z$ for the  $K\tens R( \beta )$-module $K\tens_{\cor[z]} M_z$ 
for the sake of simplicity. Recall that $\sym_{m,n}\seteq\set{w\in
\sym_{m+n}}{\text{$w(k)<w(k+1)$ for any $k\not=m$}}$. Then we can
easily see that
$$M_z\circ N_{z'}=\soplus_{w\in \sym_{m,n}}\vphi_w\cdot M_z\tens N_{z'}.$$
 Indeed, it follows from \eqref{eq:phitau} and
the fact that $x_i-x_j\col M_z\tens N_{z'}\to M_z\tens N_{z'}$ is bijective
for $1\le i\le m<j\le m+n$.

Let $L$ be a non-zero $K\tens R( \beta+\gamma )$-submodule of $M_z\circ N_{z'}$.
For $u\in M_z\circ N_{z'}$, we write
$$u=\sum_{w\in \sym_{m,n}}\vphi_w u_w\qtext{with $u_w\in M_z\tens N_{z'}$,}$$
and set $S(u)=\set{w\in \sym_{m,n}}{u_w\not=0}$.

Let us take  a \ non-zero $u\in L$ such that the cardinality of $S(u)$ is
minimal and let $w\in S(u)$. Assume that there exists $w'\not=w$ in
$S(u)$. Since we have $$\set{1\le k\le m+n}{w^{-1}(k)\le m}
\not\subset\set{1\le k\le m+n}{{w'}^{-1}(k)\le m},$$ there exists
$a$ such that $1\le w^{-1}(a)\le m$ and $m<{w'}^{-1} (a)$.

Take a monic polynomial $f(x)$ such that $f(x_k)\vert_N=0$ for $1\le
k\le n$. Then $f(x_k-z')$ acts by zero on $M_z\tens N_{z'}$ for $k>
m$, and the action of $f(x_k-z')$ is invertible on $M_z\tens N_{z'}$
for $k\le m$. Hence $f(x_{{w'}^{-1}(a)}-z')u_{w'}=0$. Since
$$f(x_{a}-z')u=\sum_{y\in S(u)\setminus\{w'\}}\vphi_yf(x_{y^{-1}(a)}-z')u_y,$$
the minimality of $S(u)$ implies that $f(x_{a}-z')u=0$.
In particular, we have
$f(x_{w^{-1}(a)}-z')u_w=0$, which implies that $u_w=0$.
It is a contradiction and we conclude that $S(u)=\{w\}$
and $u=\vphi_w u_w$.

Since $\vphi_{w^{-1}}\vphi_w\col M_z\tens N_{z'}\to M_z\tens N_{z'}$
is a $K[x_1,\ldots, x_{m+n}]$-linear isomorphism, we have $u_w\in
L$. Since $M_z\tens N_{z'}$ is a simple $K\tens R( \beta )\tens R(
\gamma )$-module, we have $M_z\tens N_{z'}\subset L$, and we
conclude that $L=M_z\circ N_{z'}$. \QED

This lemma immediately implies
the following proposition which  says  that an R-matrix for
simple modules is unique up to constant multiple.

\Prop Assume that $\cor$ is a field and let $z$, $z'$ be independent
indeterminates. Let $M$ and $N$ be a  simple  $R( \beta )$-module
and $R( \gamma )$-module, respectively. Then we have \eqn &&\Hom_{R(
\beta +\gamma)}\bl M_z\circ N_{z'}, M_z\circ N_{z'}\br
\simeq\cor[z,z'],\\[1ex]
&&\Hom_{ R( \beta +\gamma)}\bl M_z\circ N_{z'},
N_{z'}\circ M_z\br
\simeq \cor[z,z'](z-z')^{-s} R_{M_z, N_{z'}}.
\eneqn
Here $s$ is the order of zeroes of
$R_{M_z,N_{z'}}$ \ro see \eqref{def:s}\rf.
\enprop
\Proof
The first assertion immediately follows from the preceding lemma.
The second follows from the first since
$R_{M_z, N_{z'}}$ gives an isomorphism
from
$ K\tens_{ \cor[z,z']}(M_z\circ N_{z'})$ to $K\tens_{ \cor[z,z']}(N_{z'}\circ M_z)$,
where $K=\cor[z,z',(z-z')^{-1}]$.
\QED

\section{Quantum affine algebras and their representations}

\subsection{Quantum affine algebras}
 \hfill

In this section, we briefly review the representation theory of quantum affine algebras following \cite{AK, Kas02}.
Hereafter, we take the algebraic closure of $\C(q)$
in $\cup_{m >0}\C((q^{1/m}))$ as a base field $\cor$.

Let $ I$ be an index set and $A=(a_{ij})_{i,j \in I}$ be a generalized
Cartan matrix of affine type; i.e., $A$ is positive semi-definite of corank 1.
 We choose $0\in I$ as the leftmost vertices in the tables
in \cite[{pages 54, 55}]{Kac} except $A^{(2)}_{2n}$-case, in which
case we take the longest simple root as $\al_0$. Set $I_0
=I\setminus\{0\}$. We take a Cartan datum $(A,P,
\Pi,P^{\vee},\Pi^{\vee})$ as follows.

The weight lattice $P$ is given by
\begin{align*}
  P = \Bigl(\soplus_{i\in I}\Z \La_i\Bigr) \oplus \Z \delta
\end{align*}
and the simple roots are given by
$$\al_i=\sum_{j\in I}a_{ji}\La_j+\delta(i=0)\delta.$$
Also, the simple coroots $h_i\in P^\vee=\Hom_\Z (P,\Z)$ are given by
\eqn
&& \lan h_i,\La_j\ran=\delta_{ij},\quad \lan h_i,\delta\ran=0.
\eneqn
\noindent
Let us denote by $\g$ the affine Kac-Moody Lie algebra associated with
the affine Cartan datum $(A,P, \Pi,P^{\vee},\Pi^{\vee})$.
We denote by $\g_0$ the subalgebra of $\g$ generated by
$e_i, f_i, h_i$ for $i \in I_0$.
Then $\g_0$ is a finite-dimensional simple Lie algebra.
Consider the positive integers $c_i$'s and $d_i$'s determined by the conditions
$$\sum_{i\in I}c_i a_{ij}= \sum_{i\in I} a_{ji} d_i =0
\quad\text{for all $j \in I$,}$$
and $\{c_i \}_{i\in I}$, $\{d_i \}_{i\in I}$
be families of relatively prime positive integers
(see \cite[Chapter 4]{Kac}). Then the center of $\g$ is 1-dimensional and is generated by the {\em canonical central element}
$$c= \sum_{i\in I}c_ih_i$$
(\cite[Proposition 1.6]{Kac}).
Also it is known that the imaginary roots of $\g$ are non-zero integral multiples of the {\em null root}
$$\delta= \sum_{i\in I}d_i \alpha_i$$
(\cite[Theorem 5.6]{Kac}).
Note that $d_0 =1$ in all cases and
$c_0=1$ if $\g \neq A^{(2)}_{2n}$ and $c_0=2$ if $\g=A^{(2)}_{2n}$.

Let us denote by $U_q(\g)$ the quantum group associated with the affine Cartan datum $(A,P, \Pi,P^{\vee},\Pi^{\vee})$.
 We denote by $\uqpg$ the subalgebra of $U_q(\mathfrak{g})$ generated by $e_i,f_i,K_i^{\pm1}(i=0,1,,\ldots,n)$.

Set
$$P_\cl=P/\Z\delta$$
and call it the {\em classical weight lattice}. Let $\cl\col P\to
P_\cl$ be the projection. Then $P_\cl=\soplus_{i\in I}\Z\cl(\La_i)$
and we have
$$P_\cl^\vee\seteq\Hom_{\Z} (P_\cl,\Z)=\set{h\in P^\vee}{\lan h,\delta\ran=0}
=\soplus_{i\in I}\Z h_i.$$ Set $\Pi_\cl = \cl(\Pi)$ and
$\Pi^{\vee}_\cl= \{h_0, \ldots, h_n\}$. Then $U_q'(\mathfrak{g})$
can be regarded as the quantum group associated with the quintuple
$(A,P_\cl, \Pi_\cl,P^{\vee}_\cl,\Pi^{\vee}_\cl)$.

Set $P^0_\cl = \set{\la\in P_\cl}{\lan c,\la\ran=0}\subset P_\cl$.
We call the elements of $P^0_\cl$ by the {\em classical integral
weight of level $0$.} Note that we can take $P_\cl^0 $ as the weight
lattice of $\g_0$.

Let $W$ be the {\em Weyl group} of $\g$.
The image of the canonical group homomorphism $W \rightarrow \Aut(P^0_\cl)$
is denoted by $W_\cl$.
Then $W_\cl$ coincides with the Weyl group of $\g_0$.

A $\uqpg$-module $M$ is called an {\em integrable module} if
 \bni
\item $M$ has a weight space decomposition
$$M = \bigoplus_{\lambda \in P_\cl} M_\lambda,$$
where  $M_{\lambda}= \set{ u \in M }{
\text{$K_i u =q_i^{\lan h_i , \lambda \ran} u$ for all $i\in I$}}$,
\item  the actions of
 $e_i$ and $f_i$ on $M$ are locally nilpotent for any $i\in I$.
\ee
In this paper, we mainly consider the category of finite-dimensional integrable $\uqpg$-modules,
denoted by $\CC_\g$. The modules in this category are also called {\em of type $1$} (for example, see \cite{CP94}).

Let $M$ be an integrable $\uqpg$-module.
Then the {\em affinization} $M_\aff$ of $M$
is a $P$-graded $\uqpg$-module
$$M_\aff=\soplus_{\la\in P}(M_\aff)_{\la}\qtext{with $(M_\aff)_{\la}=M_{\cl(\la)}$.}
$$
Let us denote by  $\cl\col M_\aff \to M$ the canonical $\cor$-linear homomorphism.
The actions  $$e_i\col (M_\aff)_\la  \to (M_\aff)_{\la+\al_i}
 \quad \text{and} \quad f_i\col(M_\aff)_\la\to (M_\aff)_{\la-\al_i}$$ are defined
  in such a way  that
they commute with $\cl\col M_\aff\to M$.

  We  denote by $z_M\col M_\aff\to M_\aff$ the $\uqpg$-module
automorphism of weight $\delta$ defined by $(M_\aff)_\la\simeq
M_{\cl(\la)}\simeq (M_\aff)_{\la+\delta}$. Then we have
 $$M\simeq M_\aff/(z_M-1)M_\aff.$$

Let $A$ be a commutative $\cor$-algebra and
let $x$ be an invertible element of $A$.
For an $A \otimes_{\cor} \uqpg$-module $M$, let us denote by $M_x$
the $A \otimes_{\cor} \uqpg$-module $M_\aff/(z_M-x)M_\aff$.
For invertible elements $x, y $ of $A$ and
$A \otimes_{\cor} \uqpg$-modules $M, N$, we have
$$(M_x)_y \simeq M_{xy}\qtext{and}
\quad(M \otimes_A N)_x \simeq M_x \otimes_A N_x.$$
We sometimes write $M_z$ for $M_\aff$
with $z_M=z$.

We embed $P_\cl$ into $P$ by $\iota\col P_\cl\to P$
which is given by $\iota(\cl(\La_i))=\La_i$.
For $u\in M_\la$ ($\la\in P_\cl$),
let us denote by $u_z\in (M_\aff)_{\iota(\la)}$ the element such that
$\cl(u_z)=u$. With this notation,  we have
$$e_i(u_z)=z^{\delta_{i,0}}(e_iu)_z, \quad
f_i(u_z)=z^{-\delta_{i,0}}(f_iu)_z, \quad
K_i(u_z)=(K_iu)_z.
$$
Then we have
$M_\aff\simeq\cor[z,z^{-1}]\tens M$.

\begin{definition}
 Let $M$ be an integrable $U_q(\g)$-module. A weight vector $u \in M_\lambda \ (\lambda \in P)$ is called
an {\em extremal vector} if there exists  a family of vectors $\{u_w\}_{w\in W}$
satisfying the following properties:
\eq
&&\text{$u_w=u$ for $w=e$,} \nonumber \\
&&
\hbox{if $\langle h_i,w\lambda\rangle\ge 0$, then
$e_iu_w=0$ and $f_i^{(\langle h_i,w\lambda\rangle)}u_w=u_{s_iw}$,} \nonumber \\
&&\hbox{if $\langle h_i,w\lambda\rangle\le 0$, then
$f_iu_w=0$ and $e_i^{( -  \langle h_i,w\lambda\rangle)} u_w=u_{s_iw}$.} \nonumber
\eneq
 This definition extends to the $\uqpg$-module case
by replacing $P$ with $P_\cl$
and keeping $W$.
\end{definition}

Hence if such $\{u_w\}_{w\in W}$ exists, then it is unique and
$u_w$ has weight $w\lambda$.
We denote $u_w$ by $S_wu$.

For $\lambda \in P$,
let us denote by $W(\lambda)$
the $U_q(\g)$-module
generated by $u_\lambda$
with the defining relation that
$u_\lambda$ is an extremal vector of weight $\lambda$ (see \cite{Kas94}).
This is in fact a set of infinitely many linear relations on $u_\lambda$.

 Set $\varpi_i=\gcd(c_0,c_i)^{-1}(c_0\Lambda_i-c_i\Lambda_0) \in P^0$
for $i=1,2,\ldots,n$.
Then $\{\cl(\varpi_i)\}_{i=1,2,\ldots,n}$ forms a basis of $P^0_\cl$.
We call $\varpi_i$ a {\em level $0$ fundamental weight}.
 As shown in \cite{Kas02}, for each $i=1,\ldots,n$, there exists a $\uqpg$-module automorphism
$z_i \col W(\varpi_i) \rightarrow W(\varpi_i)$
which sends $u_{\varpi_i}$ to $u_{\varpi_i + \mathsf{d_i} \delta}$,
where $\mathsf{d_i} \in \Z_{>0}$ denotes the generator of the free abelian group $\set{ m \in \Z}{\varpi_i + m \delta \in W \varpi_i }$.

We define the $\uqpg$-module $V(\varpi_i)$ by
$$V(\varpi_i) = W(\varpi_i) / (z_i-1) W(\varpi_i).$$
 It can be characterized as follows (\cite[\S\;1.3]{AK}):
\begin{enumerate}
\item  the weights of $V(\varpi_i)$ are contained in the convex hull of $W_\cl \cl(\varpi_i)$,
\item $\dim V(\varpi_i)_{\cl(\varpi_i)} = 1$,
\item for any $\mu \in W_\cl \cl(\varpi_i) \subset P^0_\cl$, we can associate a non-zero vector $u_\mu$ of weight $\mu$ such that
$$u_{s_i \mu} = \begin{cases}
f_i^{(\langle h_i, \mu \rangle )} u_\mu & \text{if} \ \langle h_i, \mu \rangle \ge 0, \\
e_i^{(-\langle h_i, \mu \rangle) } u_\mu & \text{if} \ \langle h_i, \mu \rangle \leq 0,
\end{cases}$$
\item $V(\varpi_i)$ is generated by $V(\varpi_i)_{\cl(\varpi_i)}$ as a  $\uqpg$-module.
\end{enumerate}
We call $V(\varpi_i)$ the {\em fundamental representation of $\uqpg$ of weight $\varpi_i$}.

We have $V(\varpi_i)_\aff \simeq
\cor[z_i^{{1 / \mathsf{d_i}}}] \otimes_{\cor[z_i]} W(\varpi_i)$,
and hence if $\mathsf{d_i} =1$, then $W(\varpi_i) \simeq V(\varpi_i)_\aff$ \cite[Theorem 5.15]{Kas02}.

\bigskip
An involution of a $\uqpg$-module $M$ is called a \emph{bar
involution} if $\ol{au}=\bar a\bar u$ holds for any $a\in \uqpg$ and
$u\in M$. We say that a finite crystal $B$ with weight in $P_\cl^0$
is a {\em simple crystal} if there exists $\lambda \in P_\cl^0$ such
that $\sharp(B_{\lambda})=1$ and the weight of any extremal vector of
$B$ is contained $W_\cl \lambda$. If a $\uqpg$-module $M$ has a bar
involution, a crystal basis with simple crystal graph, and a global
basis, then we say that $M$ is a {\em good module}
(\cite[\S\;8]{Kas02}). For example, the fundamental representation
$V(\varpi_i)$ is a good $\uqpg$-module. Any good module is a simple
$\uqpg$-module.

\subsection{R-matrices}
 \hfill

We recall the notion of R-matrices of good modules following \cite[\S\;8]{Kas02}.
For a vector $v$ in a $\uqpg$-module $M$,
assume that $\wt(v)$ is of level $0$ and dominant with respect to $I_0$.
Then
$v$ is an extremal weight vector if and only if
$\wt(\uqpg v) \subset \wt(v) + \sum_{i \in I_0} \Z_{\le0} \cl(\alpha_i)$.
In this case, we call $v$ a
{\em dominant extremal weight vector of $M$}.

Let $M_1$ and $M_2$ be good $\uqpg$-modules, and let $u_1$ and $u_2$
be dominant extremal weight vectors in $M_1$ and $M_2$,
respectively. Then there exists a unique $\uqpg$-module homomorphism
\begin{equation}
\Rnorm_{M_1, M_2} \col (M_1)_\aff \otimes (M_2)_\aff \rightarrow
\cor(z_1,z_2)\otimes_{\cor[z_1^{\pm1},z_2^{\pm1}]} \big((M_2)_\aff \otimes (M_1)_\aff \big)\end{equation}
satisfying
\eq \label{eq:r-matrix commute with z}
&&\Rnorm_{M_1, M_2} \circ z_1 = z_1 \circ\Rnorm_{M_1, M_2},\quad
\Rnorm_{M_1, M_2} \circ z_2 =z_2\circ \Rnorm_{M_1, M_2},\\
&&\text{that is, $\Rnorm_{M_1, M_2}$ is $\cor[z_1,z_2]$-linear,}\nn
\eneq
and
\begin{equation}\Rnorm_{M_1, M_2}(u_1 \otimes u_2) = u_2 \otimes u_1
\end{equation}
 (\cite[\S\;8]{Kas02}).
Here we write $z_k$ for the automorphism $z_{M_k}$ of $(M_k)_\aff$
($k=1,2$). We sometimes write $\Rnorm_{M_1, M_2}(z_1,z_2)$ if we
want to emphasize $z_1$ and $z_2$ as variables and consider the
R-matrix as an element of $\cor(z_1,z_2)\tens_\cor\Hom_\cor(M_1\tens
M_2, M_2\tens M_1)$.

Note that $\Rnorm_{M_1, M_2} \bl (M_1)_\aff \otimes (M_2)_\aff\br
\subset\cor(z_2/  z_1 ) \otimes_{\cor[(z_2/z_1)^{\pm1}]}
\big((M_2)_\aff \otimes (M_1)_\aff \big)$. Let $d_{M_1,M_2}(u) \in
\cor[u]$ be a monic polynomial of the smallest degree such that the
image of $d_{M_1,M_2}(z_2/z_1) \Rnorm_{M_1, M_2}$ is contained in
$(M_2)_\aff \otimes (M_1)_\aff$. We call $\Rnorm_{M_1, M_2}$ the
{\em normalized R-matrix} and $d_{M_1,M_2}$ the {\em denominator of
$\Rnorm_{M_1, M_2}$}.

Since $(M_1)_{x} \otimes (M_2)_{y}$ is irreducible
for generic $x,y\in\cor^\times$, we have
\begin{equation}\label{eq:r^2=1}
\Rnorm_{M_2, M_1} \circ \Rnorm_{M_1, M_2} = \id_{(M_1)_\aff \otimes (M_2)_\aff}.
\end{equation}
The normalized R-matrices satisfy the {\em Yang-Baxter equation}
(cf.\ \eqref{dia:RKLR}).

\eq \label{eq:r_YB}
&&\ba{l}
(\Rnorm_{M_2, M_3} \otimes 1 )\circ (1 \otimes \Rnorm_{M_1, M_3})
 \circ (\Rnorm_{M_1, M_2} \otimes 1)\\[2ex]
\hs{20ex}=(1 \otimes \Rnorm_{M_1, M_2}) \circ (\Rnorm_{M_1, M_3} \otimes 1) \circ (1 \otimes \Rnorm_{M_2, M_3}).\ea
\eneq

 Let $M_1,\ldots, M_t$ be good modules
and $w\in \sym_t$. Let $z_k$ be $z_{M_k}$ in $\End((M_k)_\aff)$.
Then we can define
\eq&&\hs{5ex}
\ba{l}
{\Rnorm}_{M_1,\ldots, M_t}^{\hs{2ex}w }
\col (M_1)_\aff\tens\cdots (M_t)_\aff\\[1ex]
\hs{20ex}\To\cor(z_1,\ldots, z_t)\tens_{\cor[z_1^{\pm1},\ldots,
z_t^{\pm1}]} (M_{w(1)})_\aff\tens\cdots (M_{w(t)})_\aff \ea \eneq as
$${\Rnorm}_{M_1,\ldots, M_t}^{ \hs{2ex}w
}=\Rnorm_{M_{w_\ell(i_\ell)},M_{w_\ell(1+i_{\ell})}}\circ
\cdots\circ \Rnorm_{M_{w_1(i_1)},M_{w_1(1+i_{1})}}$$ for a reduced
expression $w=s_{i_1}\cdots s_{i_\ell}$ of $w$, and
$w_k=s_{i_{1}}\cdots s_{i_{k-1}}$. By the Yang-Baxter equation for
normalized R-matrices, this definition does not depend on the choice
of reduced expressions of $w$. We write ${\Rnorm}_{M_1,\ldots, M_t}$
for ${\Rnorm}_{M_1,\ldots, M_t}^{ \hs{2ex}w_t }$, when $w_t$ is the
longest element of $\sym_t$.

The following facts are proved in \cite{AK,Kas02}. .
\Th\label{th:zero} \hfill
\bnum
\item For good modules $M_1,M_2$, the zeroes of $d_{M_1,M_2}(z)$ belong to
$\C[[q^{1/m}]]\;q^{1/m}$ for some $m\in\Z_{>0}$.
\item Let $M_k$ be a good module
with a  dominant extremal vector $u_k$ of weight $\la_k$, and
$a_k\in\cor^\times$ for $k=1,\ldots, t$.

Assume that $a_j/a_i$ is not a zero of $d_{M_i, M_j}(z) $ for any
$1\le i<j\le t$. Then the following statements hold. \bna
\item
 $(M_1)_{a_1}\tens\cdots\tens (M_t)_{a_t}$ is generated by
$u_1\tens\cdots \tens u_t$.
\item The head of
$(M_1)_{a_1}\tens\cdots\tens (M_t)_{a_t}$ is simple.
\item The vector $u_t\tens\cdots\tens u_1$ cogenerates
$(M_t)_{a_t}\tens\cdots\tens (M_1)_{a_1}$; i.e., any non-zero
submodule contains this vector.
\item The socle of $(M_t)_{a_t}\tens\cdots\tens (M_1)_{a_1}$
is simple.
\item
 Let
$$r\col (M_1)_{a_1}\tens\cdots\tens (M_t)_{a_t}
\to (M_t)_{a_t}\tens\cdots\tens (M_1)_{a_1}$$
 be  the specialization of ${\Rnorm}_{M_1,\ldots, M_t}$
at $z_k=a_k$.  Then the image of $r$ is simple and
it coincides with the head of
$(M_1)_{a_1}\tens\cdots\tens (M_t)_{a_t}$
and also with the socle of $(M_t)_{a_t}\tens\cdots\tens (M_1)_{a_1}$.
\ee
\item
Let $M$ be a simple integrable $\uqpg$-module $M$. Then, there exists
a finite sequence $\bl(i_1,a_1),\ldots, (i_t,a_t)\br$
in $I_0\times \cor^\times$
such that
$d_{V(\varpi_{i_k}),V(\varpi_{ i_{k'} })}(a_{k'}/a_k)\not=0$ for $1\le k<k'\le t$ and
$M$ is isomorphic to the head of
$V(\varpi_{i_1})_{a_1}\tens\cdots\tens V(\varpi_{i_t})_{a_t}$.
Moreover, such a sequence $\bl(i_1,a_1),\ldots, (i_t,a_t)\br$
is unique up to permutation.
\ee
\enth

\begin{example} \label{ex:R-matrix}
When $\g=\slnh$, the normalized R-matrices for the fundamental representations are given as follows
for $1\le k,\ell\le N-1$ (see, for example, \cite{DO94}):
 \begin{align}
   \Rnorm_{V(\varpi_k), V(\varpi_{\ell})}=
    \sum_{0 \leq i \leq \min(k,\ell,  N-k, N-\ell)} \prod_{s=1}^{i} \frac{1-(-q)^{|k-\ell|+2s} z }{z- (-q)^{|k-\ell|+2s}} \
    P_{\varpi_{\max(k,\ell)+i}+\varpi_{\min (k,\ell)-i}},
 \end{align}
where $z=z_{V(\varpi_{\ell})}/z_{V(\varpi_k)}$
and $P_\lambda$ denotes the projection from $V(\varpi_k) \otimes V(\varpi_{\ell})$ to the direct summand $V(\lambda)$ as a $U_q(\mathfrak{sl}_N)$-module.
 We understand $\varpi_0=\varpi_N=0$.

Note that
\eq\hs{10ex}
d_{V(\varpi_k), V(\varpi_{\ell})}(z)=\hs{-4ex}
\prod\limits_{s=1}^{\min(k, \ell, N-k, N-\ell)}\hs{-2ex}
\bl z-(-q)^{|k-\ell|+2s}\br
=\hs{-2ex}\prod_%
{\substack{|k-\ell|+2\le p\le k+\ell,\, 2N-k-\ell,\\
p\equiv k-\ell\hs{-1ex}\mod 2}}\hs{-4ex}
(z-(-q)^p).\label{eq:zerR}
\eneq
\end{example}

\begin{remark} \label{rmk:simple poles}
The above example shows that if $\g$ is of type $A_{N-1}^{(1)}$
($N\ge 2$), then every normalized R-matrix between two fundamental
representations of $U_q'(\g)$ has only simple poles. In the
forthcoming paper, we will show that it is not true when $\g$ is of
type $D_N^{(1)}$ ($N \ge 4$).
\end{remark}

\section{Affine Schur-Weyl duality via \KLRs}\label{subsec:action}

\subsection{Action of Khovanov-Lauda-Rouquier algebras  on $\Vhat^{\tensor \beta}$}

Let $\{V_s\}_{s\in \mathcal{S}}$ be a family of good $\uqpg$-modules
and  let $\lambda_s$ be a dominant extremal weight of $V_s$ and $v_s$
a dominant extremal weight vector in $V_s$ of weight $\lambda_s$.

Let $J$ be an index set and
let $X \colon J \rightarrow \mathbb{T}$ and
$S \colon J \rightarrow \Ss$ be maps,
where $\mathbb{T}=\cor^\times$ is the algebraic torus.
For each $i,j \in J $,
we choose $c_{i,j}(u,v)\in\cor[[u,v]]$ such that
\eq&&\ba{l} \label{cond:cij}
c_{i,j}(u,v)c_{j,i}(v,u)=1,\\
c_{i,i}(u,v)=1.
\ea
\eneq
Set
\begin{equation}
P_{ij}(u,v) =(u-v)^{d_{ij}}c_{i,j}(u,v),\label{def:Pij}
\end{equation}
where $d_{ij}$ denotes the order of the zero of
$d_{V_{S(i)},V_{S(j)}}(z_2 / z_1)$ at $z_2 / z_1 = {X(j) / X(i)}$.

Let $R^{J}$ be the symmetric \KLR\ associated with
\begin{equation}\label{eq:Q_omega}
Q_{ij}(u,v) = \delta(i\not=j)P_{ij}(u,v)P_{ji}(v,u)=\delta(i\not=j) (u-v)^{d_{ij}}(v-u)^{d_{ji}}\end{equation}
for $i,j \in J$.

\Rem \label{rmk:cij} If we consider the non-graded $R^{J}$-module category $R^J\smod$,
the choice of $c_{i,j}$ is not important, and we can choose, for example,
$c_{i,j}(u,v)=1$. Indeed, adding $c_{i,j}(u,v)$
is equivalent to  the change of  generators
$\tau_ke(\nu)\mapsto c_{\nu_k,\nu_{k+1}}(x_x,x_{k+1})\tau_ke(\nu)$.
However, as we shall see later in Theorem~\ref{thm:cij},
we need to choose $c_{i,j}(u,v)$ carefully when we  are concerned with
localizations of the graded $R^{J}$-module category $R^J\gmod$.
\enrem

\begin{remark}
The underlying Cartan matrix $A_J=(a^J_{ij})_{i,j\in J}$ is given by
\begin{equation}\label{eq:quiver}
a^J_{ij}=\begin{cases}
2&\text{if $i=j$,}\\
-d_{ij}-d_{ji}&\text{if $i\not=j$.}\end{cases}
\end{equation}
Consider the quiver $\Gamma_J \seteq (J, \Omega)$ with the set of
vertices $J$ and the set of oriented edges $\Omega$ such that
$$\sharp\set{h \in \Omega}{s(h) = i, t(h)=j } = d_{ij}
\qtext{ for all $i,j\in J$,}$$
where $s(h)$ and $t(h)$ denote the source and the target of an oriented edge $h \in \Omega$.
Theorem~\ref{th:zero} implies that $X(j)/X(i)\in\C[[q^{1/m}]]q^{1/m}$ for some $m>0$
if $d_{ij}>0$.
Hence we obtain
$$\text{if $d_{ij} > 0$, then $d_{ji} =0$.}$$
Thus the quiver $\Gamma_J$ has neither loops nor cycles.
The underlying unoriented graph of $\Gamma_J$ gives a symmetric Cartan datum and the polynomials in \eqref{eq:Q_omega}
are essentially same as
the ones used in \cite{VV09} associated with
 the symmetric Cartan datum of $\Gamma_J$.
\end{remark}

\bigskip
Let
\begin{equation*}
\Oh_{\mathbb{T}^n, X(\nu)} = \cor [[X_1 - X(\nu_1), \ldots, X_n-X(\nu_n)]]
\end{equation*}
 be the completion of
the local ring $\sho_{\mathbb{T}^n, X(\nu)}$
 of $\mathbb{T}^n$ at  $X(\nu)\seteq(X(\nu_1),\ldots,X(\nu_n))$
 and $\bK[\nu]$  the fraction field of $\Oh_{\mathbb{T}^n, X(\nu)}$.
 For $\beta \in \rootl^+$ with $|\beta|=n$,
set
\begin{align*}
&\bP_{ \beta } \seteq \soplus_{\nu \in J^{ \beta }} \cor [x_1, \ldots, x_n] e(\nu)\subset R^J( \beta ), \\
&\bPh_{ \beta }\seteq\soplus_{\nu\in J^{ \beta }}\Oh_{\mathbb{T}^n, X(\nu)}e(\nu), \\
&\bK[ \beta ] \seteq \soplus_{\nu \in J^{ \beta }} \bK[\nu]e(\nu).\end{align*}

Then we have
\begin{equation*}
\bP_{ \beta }\hookrightarrow\bPh_{ \beta }\hookrightarrow\bK[ \beta ]
\end{equation*}
as $\cor$-algebras, where the first arrow is given by
\begin{equation*}
x_k e(\nu) \mapsto X(\nu_k)^{-1}\bl X_k -X(\nu_k)\br e(\nu).
\end{equation*}
Note that
\begin{equation*}
\cor [X_1^{\pm1}, \ldots, X_n^{\pm1}] \subset \Oh_{\mathbb{T}^n, X(\nu)}
\qtext{and}\quad \Oh_{\mathbb{T}^n, X(\nu)}\simeq\cor[[x_1,\ldots,x_n]]
\qtext{for all} \ \nu \in J^{ \beta }.
\end{equation*}
Let
\begin{align*}
\cor [\sym_n] \seteq \bigoplus_{w \in \sym_n} \cor r_w
\end{align*}
be the group algebra of $\sym_n$;
i.e., the $\cor$-algebra with the multiplication
$r_{w}r_{w'}=r_{ww'}$ for $w,w'\in \sym_n$.
We denote $r_{s_i}$ by $r_i$ for $i=1,\ldots, n-1$.

The symmetric group $\sym_n$ acts on $\bP_{ \beta }$, $\bPh_{ \beta }$, $\bK[ \beta ]$
from the left and we have
\begin{equation*}
\bP_{ \beta }\otimes {\cor[\sym_n]} \hookrightarrow \bPh_{ \beta }\otimes {\cor[\sym_n]} \hookrightarrow \bK[ \beta ]\otimes {\cor[\sym_n]}
\end{equation*}
as algebras. Here the algebra structure on $\bK[ \beta ]\otimes {\cor[\sym_n]}$ is given by
\begin{align} \label{eq:rel of Khat S_n}
r_w f =w(f) r_w & \quad \text{for $f \in \bK[ \beta ]$, $w \in \sym_n$.}
\end{align}
Then $\bK[ \beta ]$
may be regarded as a right $\bK[ \beta ] \otimes {\cor[\sym_n]}$-module
by
$$a (f\otimes r_w)=w^{-1}(af) \quad (a,f\in \bK[ \beta ], \ w\in
\sym_n).$$

 Define $\tau_a\in\bK[ \beta ]\otimes {\cor[\sym_n]}$ by
\begin{equation*}
e(\nu) \tau_a = \begin{cases}
e(\nu)  P_{\nu_a, \nu_{a+1}}(x_{a}, x_{a+1})r_a & \text{if $\nu_a \neq \nu_{a+1}$,} \\
e(\nu) (r_a -1)(x_a -x_{a+1})^{-1} & \text{if $\nu_a = \nu_{a+1}$.}
\end{cases}
\end{equation*}
Then the subalgebra of $\bK[ \beta ]\otimes {\cor[\sym_n]}$
generated by \eqn &&e(\nu) \ (\nu \in J^{ \beta }),\quad x_a \ (1
\leq a \leq n), \quad \tau_a  \ (1 \leq a \leq n-1) \eneqn is
isomorphic to the \KLR\ $R^J({ \beta })$ at ${ \beta }$ associated
with $Q_{ij}(u,v) = \delta(i\not=j)P_{ij}(u,v) P_{ji}(v,u)$
(\cite[Proposition 3.12]{R08}, \cite[Theorem 2.5]{KL09}).

\bigskip
For each $\nu =(\nu_1,\ldots, \nu_n) \in J^{ \beta }$, we set
$$
V_\nu =( V_{S(\nu_1)})_\aff \otimes \cdots
\otimes (V_{S(\nu_n)})_\aff$$
which is a $\bl\cor[X_1^{\pm1},\ldots, X_n^{\pm1}]\otimes\uqpg \br$-module,
where $X_k=z_{V_{S(\nu_k)}}$.
 We   define
\begin{align}
\ba{ll}
\hV_\nu&\seteq \Oh_{\mathbb{T}^n, X(\nu)}\otimes _{\cor[X_1^{\pm1},\ldots,X_n^{\pm1}]}
V_\nu,\\[1ex]
\bV[ \beta ]&\seteq \soplus\nolimits_{\nu \in J^{ \beta }}
\hV_\nu e(\nu),\\[1ex]
\bVK[ \beta ]&\seteq \bK[ \beta ]\tens_{\bPh_{ \beta }}\bV[{ \beta
}] \simeq\soplus\nolimits_{\nu\in J^\beta}\bK[\nu]\tens
\nolimits_{\Oh_{\mathbb{T}^n, X(\nu)}} \hV_\nu.\ea
\label{eq:Vhat}
\end{align}
For each $\nu \in J^{ \beta }$ and $a=1,\ldots, n-1$, there exists a $\uqpg$-module homomorphism
\begin{equation*}
R^{\nu}_{a, a+1} \colon \cor(X_1,\ldots,X_n) \otimes_{\cor[X_1^{\pm1}, \ldots X_n^{\pm1}]}
V_\nu \rightarrow
\cor(X_1,\ldots,X_n) \otimes_{\cor[X_1^{\pm1}, \ldots X_n^{\pm1}]} V_{s_a(\nu)}
\end{equation*}
which is given by
\begin{equation*}
 v_1 \otimes \cdots \otimes v_a \otimes v_{a+1} \otimes \cdots \otimes v_n
\mapsto v_1 \otimes \cdots \otimes \Rnorm_{V_{S(\nu_a)},V_{S(\nu_{a+1})}}(v_a \otimes v_{a+1}) \otimes \cdots \otimes v_n
\end{equation*}
for $v_k \in (V_{S(\nu_k)})_\aff$ $(1 \leq k \leq n)$.
It follows that
\begin{align}
&R^{\nu}_{a, a+1} \circ X_k = X_{s_a(k)} \circ R^{\nu}_{a, a+1} && \text{from \eqref{eq:r-matrix commute with z},} \nonumber\\
&R^{s_a(\nu)}_{a, a+1} \circ R^{\nu}_{a, a+1} = 1_{V_{\nu}}&& \text{from \eqref{eq:r^2=1},} \nonumber \\
&R^{s_{a+1}s_a(\nu)}_{a, a+1} \circ R^{s_a (\nu)}_{a+1, a+2} \circ R^{\nu}_{a, a+1} =
R^{s_a s_{a+1}(\nu)}_{a+1, a+2} \circ R^{s_{a+1}(\nu)}_{a, a+1} \circ R^{\nu}_{a+1, a+2}
&& \text{from \eqref{eq:r_YB}.} \ \nonumber
\end{align}
Set $d_{\nu_a,\nu_{a+1}}(u)=d_{V_{S(\nu_a)}, V_{S(\nu_{a+1})}} (u)$. Then,
we have
\begin{equation}
 R^{\nu}_{a, a+1}\circ d_{\nu_a,\nu_{a+1}}(X_{a+1}/X_a)
\colon V_{\nu} \rightarrow V_{s_a(\nu)}.\label{eq:regR}
\end{equation}

The algebra $\bK[ \beta ] \otimes {\cor[\sym_n]}$ acts on $\bVK[{ \beta }]$ from the right, where
\begin{align*}
e(\nu) r_a \col \; &\widehat{\mathbb K}_\nu \otimes_{\cor[X_1^{\pm1}, \ldots X_n^{\pm1}]} V_\nu \\
&\rightarrow \widehat{\mathbb K}_{s_a(\nu)} \otimes_{\cor(X_1,\ldots,X_n)} \big( \cor(X_1,\ldots,X_n) \otimes_{\cor[X_1^{\pm1}, \ldots X_n^{\pm1}]} V_{s_a(\nu)}\big)
\end{align*}
is given by
\begin{equation*}
(f \otimes v) e(\nu) r_a = s_a(f) e(s_a(\nu)) \otimes R^{\nu}_{a, a+1}(v)
\end{equation*}
for $f \in \widehat{\mathbb K}_{\nu}$, $v \in \cor(X_1,\ldots,X_n) \otimes_{\cor[X_1^{\pm1}, \ldots, X_n^{\pm1}]} V_\nu$.
The subalgebra $\widehat{\mathbb{K}}_{ \beta }$ acts by the multiplication.
The relation
\eqref{eq:rel of Khat S_n} follows from the properties of normalized R-matrices and hence we have a well-defined right action of the algebra $\widehat{\mathbb{K}}_{ \beta } \otimes {\cor[\sym_n]}$ on
$\bVK[{ \beta }]$.
Since the normalized R-matrices are $\uqpg$-module homomorphisms,
$\bVK[{ \beta }]$ has a structure of
 $(\uqpg,\widehat{\mathbb{K}}_{ \beta } \otimes {\cor[\sym_n]})$-bimodule.

\begin{theorem}
The subspace $\bV[{ \beta }]$ of \/  $\bVK[{ \beta }]$ is stable under
the  right action of the subalgebra $R^J({ \beta })$ of\/ $\widehat{\mathbb{K}}_{ \beta } \otimes {\cor[\sym_n]}$.
In particular, $\bV[{ \beta }]$ has a structure of
$(\uqpg, R^J({ \beta }))$-bimodule.
\begin{proof}
It is obvious that $\bV[{ \beta }]$
is stable under the actions of $e(\nu)$ $(\nu \in J^{ \beta })$ and
$x_a$ $(1 \leq a \leq n)$.
Thus it is enough to show that $\bV[{ \beta }]$ is stable under $e(\nu) \tau_a$ $(\nu \in J^{ \beta }, \ 1 \leq a < n)$.

Assume $\nu_a \neq \nu_{a+1}$. Then we have
\eqn
&&e(\nu) P_{\nu_a,\nu_{a+1}}(x_{a},x_{a+1})\\
&&\hs{1ex}=e(\nu) d_{\nu_a,\nu_{a+1}}(X_{a+1}/X_a)
\dfrac{ (X(\nu_{a})^{-1}X_a-X(\nu_{a+1})^{-1}X_{a+1})^{d_{\nu_a, \nu_{a+1}}}
c_{\nu_a,\nu_{a+1}}(x_{a},x_{a+1})} {d_{\nu_a,\nu_{a+1}}(X_{a+1}/X_a)} \nonumber.
\eneqn
Since $d_{\nu_a, \nu_{a+1}}$ is the multiplicity of the zero of the polynomial
$d_{\nu_a,\nu_{a+1}}(X_{a+1}/X_a)$ at $X_{a+1}/X_a=X(\nu_{a+1}) / X(\nu_{a})$,
we have
$$\dfrac{ (X(\nu_{a})^{-1}X_a-X(\nu_{a+1})^{-1}X_{a+1})^{d_{\nu_a, \nu_{a+1}}}
c_{\nu_a,\nu_{a+1}}(x_{a},x_{a+1})} {d_{\nu_a,\nu_{a+1}}(X_{a+1}/X_a)} \in
 \Oh_{\mathbb{T}^n, X(\nu)}.$$
It follows that
$$\hV_\nu e(\nu) P_{\nu_a,\nu_{a+1}}(x_a,x_{a+1})
\subset\hV_\nu e(\nu)d_{\nu_a,\nu_{a+1}}(X_{a+1}/ X_{a}).$$
Hence
\eqn
&&\hV_\nu e(\nu)\tau_a
=\hV_\nu e(\nu)P_{\nu_a,\nu_{a+1} }(x_{a}, x_{a+1})r_a
\subset \hV_\nu e(\nu) d_{\nu_a,\nu_{a+1}}(X_{a+1} / X_{a})r_a,
\eneqn
and it is contained in $\hV_{s_a(\nu)} e(s_a(\nu))$ by \eqref{eq:regR}.

\smallskip

Assume $\nu_a=\nu_{a+1}$.
Then $\Rnorm_{V_{S(\nu_a)},V_{S(\nu_a)}}$ does not have a pole at
$X_a=X_{a+1}$ by Theorem~\ref{th:zero}.
Since $(V_{S(\nu_a)})_x\otimes (V_{S(\nu_{a})})_x$ is irreducible for any $x\in\cor^\times$, we obtain
$\Rnorm_{V_{S(\nu_a)},V_{S(\nu_a)}}\vert_{X_a=X_{a+1}}=\id$.
Therefore, we have
\begin{align*}
\hV_\nu e(\nu) \tau_a
&=\hV_\nu e(\nu) (r_a -1)(x_a -x_{a+1})^{-1} \\
&=\hV_\nu e(\nu) X(\nu_a) (r_a -1) (X_a -X_{a+1})^{-1}
\subset \hV_\nu e(\nu),
\end{align*}
as desired.
\end{proof}
\end{theorem}

\subsection{The Functor $\F$}

Since $\bV[{ \beta }]$ is a $(\uqpg, R^J({ \beta }))$-bimodule, we can construct the following functor:
\begin{align}
\F_{ \beta } \colon \Modg(R^J(\beta)) &\rightarrow \Mod(\uqpg)\label{eq:the functor}
\end{align}
  sending an $R^J(\beta)$-module $M$  to the $\uqpg$-module
\eqn \F_{ \beta } (M) \seteq \bV[{ \beta }] \otimes_{R^J({ \beta })}
M. \eneqn
Set \eqn && \F_n \seteq \bigoplus_{\beta \in \rootl^+, \,
|\beta|=n } \F_\beta : \Mod(R^J(n))
\to \Mod(\uqpg),  \\
&& \F\seteq \soplus_{n\in\Z_{\ge0}}\F_n
\col\soplus_{n\in\Z_{\ge0}} \Mod(R^J(n)) \rightarrow \Mod(\uqpg).
\eneqn

 Recall that $\CC_\g$ denotes the category of
finite-dimensional integrable $\uqpg$-modules.
\begin{theorem} \label{thm:conv to tensor}
The functor $\F$ induces a tensor functor
$$\F\col\soplus\nolimits_{\beta \in \rootl^+} R^J( \beta) \gmod
\to\CC_\g.$$ Namely, $\F$ sends finite-dimensional graded
$R^J(\beta)$-modules to $\uqpg$  -modules in $\CC_\g$, and  there
exist canonical $\uqpg$-module isomorphisms \eqn \F(R^J(0))\simeq
\cor,  &  \F(M_1 \circ M_2) \simeq \F(M_1) \otimes \F(M_2) \eneqn
for $M_1 \in R^J( \beta_1 ) \gmod$ and $M_2 \in R^J( \beta_2 )\gmod$
 such that the diagrams in \eqref{dia:tensf} are  commutative .
\end{theorem}

\begin{proof}
(i) \ First let us show that
$\F(M)$ belongs to $\CC_\g$ for any $M\in  R^J( \beta )\gmod$.
Since $\bV[\beta]$ has a weight decomposition, $\F(M)=
\bV[\beta]\tens_{\R^J(\beta)}M$ has also a weight decomposition.
Since $\F(M)$ is a quotient of
$\bV[\beta]\tens_{\mathbb{P}_\beta}M$ which is finite-dimensional,
$\F(M)$ is also finite-dimensional.

\smallskip
\noi
(ii)\quad We shall construct a canonical isomorphism
$\F(M_1 \circ M_2) \simeq \F(M_1) \otimes \F(M_2)$.
 Set $\beta=\beta_1 +\beta_2$.  For each $\nu=(\nu_1,\ldots,\nu_n) \in J^{\beta}$ such that
$\nu'=(\nu_1, \ldots, \nu_{n_1}) \in J^{\beta_1}$, $\nu''=(\nu_{1+n_1}, \ldots, \nu_{n}) \in J^{\beta_2}$,
we have an algebra homomorphism
$\Oh_{\mathbb{T}^{n_1}, X(\nu')} \otimes \Oh_{\mathbb T^{n_2}, X(\nu'')}
\to  \Oh_{\mathbb T^n, X(\nu)}$.
Moreover, for any finite-dimensional $\Oh_{\mathbb{T}^{n_1}, X(\nu')}$-module
$L_1$ and any finite-dimensional $\Oh_{\mathbb{T}^{n_2}, X(\nu'')}$-module $L_2$,
the induced morphism
$$L_1\otimes L_2\to \Oh_{\mathbb T^n, X(\nu)}
\tens_{\Oh_{\mathbb{T}^{n_1}, X(\nu')} \otimes \Oh_{\mathbb T^{n_2}, X(\nu'')}}(L_1\otimes L_2)$$
is an isomorphism.
Hence for any finite-dimensional $\mathbb{P}_{ \beta_1 }$-module
$L_1$ and any finite-dimensional $\mathbb{P}_{ \beta_2 }$-module $L_2$,
the induced morphism
$$
(\bV[ \beta_1 ] \otimes \bV[ \beta_2 ])
\tens_{\mathbb{P}_{ \beta_1 }\otimes\,\mathbb{P}_{ \beta_2 }}(L_1\otimes L_2)
\to
\bV[{ \beta }]\tens_{\mathbb{P}_{{ \beta_1 }}\otimes\,\mathbb{P}_{{ \beta_2 }}}(L_1\otimes L_2)
$$
is an isomorphism.

The module
$\bV[{ \beta }] \otimes_{R^J({ \beta })}(M_1 \circ M_2)\simeq \bV[{ \beta }]
\otimes_{R^J({ \beta_1 }) \otimes R^J({ \beta_2 })} (M_1 \otimes M_2)$
is the quotient of
$\bV[{ \beta }] \otimes_{\mathbb{P}_{{ \beta_1 }}\otimes\mathbb{P}_{{ \beta_2 }}} (M_1 \otimes M_2)$
by the submodule generated by
$va\otimes u-v\otimes au$,  where $a\in R^J({ \beta_1 }) \otimes R^J({ \beta_2 })$,
$v\in \bV[{ \beta }]$, $u\in M_1 \otimes M_2$.
A similar result holds  for $\bl\bV[{ \beta_1 }]\otimes
\bV[{ \beta_2 }] \br\otimes_{R^J({ \beta_1 }) \otimes R^J({ \beta_2 })} (M_1 \otimes M_2)$.
Thus we obtain the desired result
$$\bl\bV[ \beta_1 ]\otimes \bV[ \beta_2 ]\br\otimes_{R^J( \beta_1 ) \otimes R^J( \beta_2 )} (M_1 \otimes M_2)
\simeq \bV[ \beta ]\otimes_{R^J( \beta_1 ) \otimes R^J( \beta_2 )} (M_1 \otimes M_2).$$

The commutativity of \eqref{dia:tensf} is immediate.
\end{proof}

The following proposition is obvious by the construction.

\begin{prop} \label{prop:image of tau} \hfill

\bnum
\item For any $i \in J$,
we have
\begin{align}
  \mathcal F(L(i)_z) \simeq \cor[[z]]\tens_{\cor[z_{V_{S(i)}}^{\pm}]}(V_{S(i)})_\aff,
\end{align}
where $\cor[z_{V_{S(i)}}^{\pm 1}]\to\cor[[z]]$ is  given  by
$z_{V_{S(i)}}\mapsto X(i)(1+z)$.
\item For $i,j\in J$,
let  $\phi=R_{L(i)_z,L(j)_{z'}}\col L(i)_z \circ L(j)_{z'}
\rightarrow L(j)_{z'} \circ L(i)_z$; i.e.,
the $R^J(\alpha_i+\alpha_j)$-module homomorphism given by
\begin{align}
  \phi\bl u(i)_z\otimes u(j)_{z'}\br = \vphi_1\bl u(j)_{z'} \otimes u(i)_z\br,
\end{align}
where $\vphi_1$ is the intertwiner defined in {\rm \S\,\ref{def:int}.}
Then we have
\eqn
&&\mathcal F(\phi) =(X_i/X(i)-X_j/X(j))^{d_{i,j}}c_{i,j}(X_i/X(i)-1,X_j/X(j)-1)
\Rnorm_{V_{S(i)}, V_{S(j)}}\eneqn
as a morphism
\eqn
&&\Oh_{\mathbb T^2,( X(i),X(j)) }\tens_{\cor[X_i^{\pm 1},X_j^{\pm 1}]}
\bl(V_{S(i)})_\aff \otimes (V_{S(j)})_\aff\br \\
&&\hs{10ex}\longrightarrow\Oh_{\mathbb T^2,( X(j),X(i)) }
\tens_{\cor[X_j^{\pm 1},X_i^{\pm 1}]} \bl(V_{S(j)})_\aff \otimes
(V_{S(i)})_\aff\br,
\eneqn where ${X_i}=z_{V_{S(i)}}$ and
${X_j}=z_{V_{S(j)}}$. \ee
\end{prop}

\subsection{Exactness of  the functor $\F$} The following
propositions are  key ingredients in proving our main theorem.
\begin{prop}[{\cite[Corollary 2.9]{Kato12}, \cite[Theorem 4.6]{McNa12}}]%
\label{pro:finite global dimension}
If the quiver associated with  $R^J$   is of   finite  type $A,D,E$, then $R^J( \beta )$ has
a finite global dimension  for every $\beta \in \rootl^+$.
\end{prop}

\begin{prop} \label{pro:projectivity}
Let $A\to B$ be a homomorphism of algebras.
We assume the following conditions:
\bna
\item $B$ is a finitely generated projective $A$-module,
\item $\Hom_A(B,A)$ is a projective $B$-module,
\item the global dimension of $B$ is finite.
\ee
Then we have:
\bni
\item any $B$-module projective over $A$ is projective over $B$,
\item any $B$-module flat over $A$ is flat over $B$.
\ee
\end{prop}

\begin{proof}
Since the proof is similar, we give only the proof of (ii).

Let us denote by $\fd M$ the flat dimension of an $A$-module $M$.
Then we have $$\fd(M)\le\fd[B](M)$$ for any $B$-module $M$,
because $\Tor_k^A(N,M)\simeq \Tor_k^B(N\tens_A B,M)$ for any
$A^\opp$-module $N$ and $k\in\Z$ by (a).

By (b),  $\Hom_A(B,A)\tens_AL$ is a flat $B$-module
if $L$ is a flat $A$-module. Indeed,
the functor $X\tens_B\Hom_A(B,A)$ is exact in $X\in\Mod(B^\opp)$
and hence the functor $X\tens_B\Hom_A(B,A)\tens_AL$
is also exact in $X$.

On the other hand, for any $A$-module $L$, the canonical $B$-module homomorphism
$$\Hom_A(B,A)\tens_AL\to \Hom_A(B,L),
\qquad f\tens s\longmapsto (B\ni b\mapsto f(b)s)
$$
is an isomorphism by (a).
Hence we conclude that
$\Hom_A(B,L)$ is a flat $B$-module
for any flat $A$-module $L$.
It immediately implies that
$$\fd[B]\bl\Hom_A(B,L)\br\le\fd (L)\quad\text{for any $A$-module $L$.}$$

Now, let $M$ be a $B$-module.
Then there exists a canonical $B$-module homomorphism
$$\vphi_M\colon M\to \Hom_A(B,M)$$
given by $\vphi_M(x)(b)=bx$.
It is evidently injective.

In order to prove the proposition,
it is enough to show the following statement for any $d\ge0$:
\eq \text{for any $B$-module $M$,
$\fd (M)\le d$ implies $\fd[B](M)\le d$.} \nonumber
\eneq
We shall show it by the descending induction on $d$.
If $d\gg0$, it is a consequence of (c).
Let $M$ be a $B$-module with $\fd(M)\le d$.
We have an exact sequence
$$0\to M\To[\vphi_M] \Hom_A(B,M)\to N\to 0.$$
Then
$$\fd\bl\Hom_A(B,M)\br\le\fd[B]\bl\Hom_A(B,M)\br\le \fd (M)\le d.$$
Hence we have $\fd[A]N\le d+1$,
which implies that $\fd[B](N)\le d+1$
by the induction hypothesis.
Finally, we conclude that
$\fd[B](M)\le d$.
Thus the induction proceeds.
\end{proof}

\begin{theorem} \label{thm:exact}
If the quiver associated with  $R^J$  is of finite type $A, D, E$,
then the functor  $\F_\beta$ is exact  for every $\beta \in
\rootl^+$.
\end{theorem}
\begin{proof}
Let us apply Proposition \ref{pro:projectivity} with $A=\mathbb{P}_{ \beta }^\opp$
and $B= R^J({ \beta })^\opp$.
The conditions (a) and (b) are well-known, and (c) is nothing but Proposition \ref{pro:finite global dimension}. Therefore,
$\bV[{ \beta }]$ is a flat $R^J({ \beta })^\opp$-module
since it is a flat $\mathbb{P}_{ \beta }^\opp$-module.
\end{proof}

\section{Quantum affine algebra $\UA$
and the category $\mathcal C_{J}$}

In this section, we investigate the tensor category structure
of  $\CC_{\slnh}$  ($N\ge2$)  via its
vector representation using the method introduced in the previous section.
 In the case of $\g=\slnh$, we have
$I=\{0,1,\ldots,N-1\}$ and
$$(\al_i,\al_j)=
2\delta_{i,j}-\delta(i\equiv j+1\mod N)-\delta(i\equiv j-1\mod N).$$
The base field $\cor$ is $\C(q)$.

\subsection{R-matrix for $V(\varpi_1)$ of $\UA$}

Let $V\seteq V(\varpi_1)$ be the fundamental representation of $\UA$
of fundamental weight $\varpi_1$.
Then
$V=\soplus_{k=1}^{N}\cor u_k$
with the action
\eqn
e_i(u_k)&=&\delta({k\equiv i+1\hs{-1ex}\mod N}) u_{i}, \\
f_i(u_k)&=&\delta({k\equiv i\hs{-1ex}\mod N}) u_{i+1},\\
K_iu_k&=&q^{\delta(k\equiv i\hs{-1ex}\mod N)-\delta(k\equiv i+1\hs{-1ex}\mod N)}u_k.
\eneqn
Here $u_0=u_N$.

The normalized R-matrix
$$R=\Rnorm_{V(\varpi_1), V(\varpi_1)} \col
V_{z} \otimes V_{z'}
\rightarrow V_{z'} \otimes V_{z}$$
is explicitly given by
\eq&&
\ba{l}
 R((u_i)_z\tens (u_j)_{z'})\\
\hs{5ex}=\bc \dfrac{(1-q^2)z'{}^{\delta(i>j)}z^{\delta(i<j)}}{z'-q^2z}
(u_i)_{z'}\tens (u_j)_z+
 \dfrac{q(z'-z)}{z'-q^2z}(u_j)_{z'}\tens (u_i)_z&\text{if $i\not=j$,}\\
(u_i)_{z'}\tens (u_i)_z&\text{if $i=j$.}
\ec\ea\label{eq: RV}
\eneq
 It shows that $d_{V,V}(z'/z)=z'/z-q^2$.

\smallskip

Let $\mathcal{S}=\{V\}$,
$J=\Z$ and let $X\col J\to \cor^\times$ be the
map given by $X(j)=q^{2j}$.
Then we have
\eq
d_{ij}=\delta(j=i+1)\qtext{for $i,j\in J$.}
\eneq
Then, for $i,j\in J$, we have
$$(\al_i,\al_j)=\begin{cases}-1&\text{if $i-j=\pm1$,}\\
2&\text{if $i=j$,}\\
0&\text{otherwise,}
\end{cases}$$
and
$$Q_{ij}(u,v)=\begin{cases}
\pm(u-v)&\text{if $j=i\pm1$,}\\
0&\text{if $i=j$,}\\
1&\text{otherwise.}
\end{cases}$$
Therefore the corresponding \KLR\ $R=R^J$ is of type $\mathrm{A}_{\infty}$.

\medskip

We take
$$P_J=\soplus_{a\in \Z}\Z\oep_a$$
as the weight lattice with $(\oep_a,\oep_b)=\delta_{a,b}$.
The root lattice $\rtl_J=\soplus_{i\in J}\Z\al_i$
is embedded into $P_J$ by
$\al_i=\oep_i-\oep_{i+1}$. We write as usual $\rtl_J^+$
for $\soplus_{i\in J}\Z_{\ge0}\al_i$.

Note that,  for $\beta\in \rtl^+_J$ and $\nu\in J^\beta$,  we have
$$(\tau_{a+1}\tau_a\tau_{a+1}-\tau_a\tau_{a+1}\tau_{a})e(\nu)
=\begin{cases}\pm e(\nu)&\text{if $\nu_a=\nu_{a+2}=\nu_{a+1}\mp1$,}\\
0&\text{otherwise.}
\end{cases}$$
 Also,  we have
$$P_{ij}(u,v)=(u-v)^{\delta(j=i+1)}c_{i,j}(u,v).$$
 We will choose
$c_{i,j}(u,v)$  satisfying \eqref{cond:cij} later in  Theorem \ref{thm:cij}.
Recall  that
 the functor $$\F \col \soplus_{\beta\in  \rtl^+_J} \Modg(R(\beta))\to
\Mod( \UA) $$
 defined in \eqref{eq:the functor} is exact (Theorem \ref{thm:exact}).

\bigskip

\subsection{KLR-modules in $A$ case}
Let us recall the representation theory of \KLRs\ in $A$-case.

 A pair of integers  $(a,b)$ such that $ a \leq b$
is called a {\em segment}.  The \emph{length} of $(a,b)$  is $b-a+1$.
A \emph{multisegment} is a  finite sequence of  segments.

 For a segment $(a,b)$ of length $\ell$, we define a graded $1$-dimensional
 $R( \oep_a-\oep_{b+1} )$-module
 $L(a,b)=\cor u{(a,b)}$ in $R( \oep_a-\oep_{b+1} )\gmod$ which is generated by a
vector $u{(a,b)}$ of degree $0$ with the action of $R( \oep_a-\oep_{b+1} )$ given by
 \begin{align}
x_m u{(a,b)} =0 , && \tau_k u{(a,b)} =0 , &&
e(\nu) u{(a,b)}= \begin{cases}
u{(a,b)} & \text{if $\nu=(a,a+1, \ldots, b)$,} \\
0 & \text{otherwise.}
\end{cases}
\end{align}
Note that it was denoted by $L(a,a+1,\ldots,b)$ in
\S\,\ref{subsec:R1}.
We understand that $L(a,a-1)$ is the 1-dimensional module over $R(0) = \cor$  and
the length of $(a,a-1)$ is $0$.
When $a=b$, we use the notation $L(a)$ instead of $L(a,a)$.

Recall that $w[\ell,\ell']$ denotes the element in
the symmetric group $\sym_{\ell+\ell'}$ given by
$$w[\ell,\ell'](k)=\begin{cases}\ell'+k&\text{for $1\le k\le \ell$,}\\
k-\ell&\text{for $\ell<k\le\ell+\ell'$,}
\end{cases}$$
and we write  $\tau_{\ell,\ell'}$ for
$\tau_{w[\ell,\ell']}\in R( \beta )$  with $|\beta|=\ell+\ell'$.
For example, we have $\tau_{1,\ell}=\tau_\ell\cdots\tau_1$.

 In the sequel, for $n\in\Z_{\ge0}$ and $\beta\in\rtl^+_J$
with $|\beta|=n$, we sometimes write $e(n)$ for $e(\beta)$.
 Since the proofs of the following lemmas are straightforward, we omit them.
\Lemma
Let $z$ and $z'$ be  algebraically  independent indeterminates.
 For $a<b$, set  $\ell = b-a+1$.  Then  we have
\eq
&&  \tau_{1,\ell}  (e(1)\etens\tau_{\ell-1,1})(u(a,b)_{z'}\tens u(a)_z)=
u(a,b)_{z'}\tens u(a)_z
\eneq
in $L(a,b)_{z'}\circ L(a)_z$.
\enlemma

\Lemma\label{lem:rmat.segments}
 For $a \leq b$,  we have
\eqn&&R_{L(a,b)_z,L(a,b)_{z'}}(u(a,b)_{z}\tens u(a,b)_{z'})\\
&&\hs{15ex}=
\bl(z'-z)^{b-a+1} \tau_{b-a+1,b-a+1}
+(z'-z)^{b-a}\br(u(a,b)_{z'}\tens u(a,b)_z).\eneqn
\enlemma

Some parts of the following proposition appeared in \cite{Vaz02}
in terms of modules over affine Hecke algebras.

We omit the details of its proof.
\begin{prop} \label{prop:exact sequences}
For $a \leq b$ and $a' \leq b'$,
 set $\ell=b-a+1$, $\ell'=b'-a'+1$ and $p=\sharp([a,b]\cap[a',b'])
=\max(0,\min(b',b)-\max(a,a')+1)$, $\beta=\oep_{a}-\oep_{b+1}$
and $\beta'=\oep_{a'}-\oep_{b'+1}$.
Let $s$ be the order of zeros of
$R_{L(a,b)_z,L(a',b')_{z'}}$ \ro see \eqref{def:s}\rf.
  \bnum
\item If $a'=a$ and $b'=b$, then $s=b-a$ and we have
$\rmat{L(a,b),L(a,b)}=\id_{L(a,b)\circ  L(a,b)}$.
  \item
\bna
\item
If $a \leq a' \leq b \leq b'$, then $s=b-a'$ and
 there exists a nonzero homomorphism
  \begin{align*}
   f\seteq\rmat{L(a,b),L(a',b')} : L(a,b) \circ L(a',b') \to q^{\delta_{a,a'}+\delta_{b,b'}-2}L(a',b') \circ L(a,b).
  \end{align*}
\item  Unless $a \leq a' \leq b \leq b'$,  we have  $s=p$,
$$R_{L(a,b)_z,L(a',b')_{z'}}(u(a,b)_{z}\tens u(a',b')_{z'})
=(z'-z)^{p}\tau_{\ell',\ell}(u(a',b')_{z'}\tens u(a,b)_z)$$ and
there exists a nonzero homomorphism
 \begin{align*}
   g\seteq \rmat{L(a,b),L(a',b')}: L(a,b) \circ L(a',b')\to q^{(\beta,\beta')}
 L(a',b') \circ L(a,b)
  \end{align*}
 given  by
$g\bl u(a,b)\tens u(a',b')\br=\tau_{\ell',\ell}\bl u(a',b')\tens u(a,b)\br$.
\ee

\item If $a\le a'\le b'\le b$, then
 $L(a,b) \circ L(a',b')$ is irreducible and
$$L(a,b) \circ L(a',b') \simeq q^{\delta_{a,a'}-\delta_{b,b'}}L(a',b') \circ L(a,b).$$
\item If $b'<a-1$,
then  $L(a,b) \circ L(a',b')$ is irreducible and
$$g\col L(a,b) \circ L(a',b')\isoto L(a',b') \circ L(a,b).$$
  \item If $a' < a\le b' < b$, then
we have the following exact sequence
\begin{align*}
 & 0 \To qL(a',b) \circ L(a, b') \To[\vphi]
  L(a, b) \circ L(a',b')\\
&\hs{20ex}\To[{\hs{1.5ex} g\hs{1.5ex}}] L(a',b') \circ L(a, b)\To
 q^{-1}L(a', b) \circ L(a,b')\To 0.\end{align*}
Moreover, the image of $g$ coincides with the
head of $ L(a, b)\circ L(a',b')$ and the socle of $ L(a',b') \circ L(a, b)$.
\item If $a=b'+1$, then we have an exact sequence
\begin{align*}
  0 \to q L(a',b) \To[\psi]
 L(a, b) \circ L(a',b')
\To[{g}] q^{-1}L(a',b') \circ L(a, b)\to
q^{-1} L(a', b)\rightarrow 0.
\end{align*}
Moreover, the image of ${g}$ coincides with the
head of $ L(a, b)\circ L(a',b')$ and the socle of
$ q^{-1}L(a',b') \circ L(a, b)$.
\item
$\ds{L(a,b),L(a',b')}
=(\beta,\beta')-2\delta(a\le a'\le b\le b')$.
\ee
\end{prop}
\Proof
The assertion (i) is noting but Lemma \ref{lem:rmat.segments}.

\smallskip
\noi
 Let us show (ii) (a).
Let $a\le b$ and $a\le c\le b+1$. Then it is easy to see that
there exists a unique morphism
$$\xi_{a,c,b}\col L(a,b)\To q^{\delta_{a,c}+\delta_{c,b+1}-1}L(c,b)\circ L(a,c-1)$$
such that the diagram
\eq
\xymatrix{
L(a,c-1)\circ L(c,b)\ar@{->>}[r]\ar[dr]_(.4){R_{L(a,c-1),L(c,b)}}
&L(a,b)\ar[d]^{\xi_{a,c,b}}\\
&q^{\delta_{a,c}+\delta_{c-1,b}-1}L(c,b)\circ L(a,c-1)}
\label{mor:xi}
\eneq
commutes.
We have explicitly
$$\xi_{a,c,b}(u(a,b))=\tau_{b-c+1,c-a}\bl u(c,b)\tens u(a,c-1)\br.$$

 Let $f$ be the composition
\eqn
&&\hs{-2ex}
L(a,b)\circ L(a',b')\\
&&\To[{\;\xi_{a,a',b}\circ \xi_{a',b+1,b'}\;}]
q^{\delta_{a,a'}-1}L(a',b)\circ L(a,a'-1)\circ
q^{\delta_{b,b'}-1} L(b+1,b')\circ L(a',b)\\
&&\To[{R_{L(a,a'-1), L(b+1,b')}}]
 q^{\delta_{a,a'}+\delta_{b,b'}-2}L(a',b)\circ L(b+1,b')
\circ L(a,a'-1)\circ L(a',b)\\
&&\To q^{\delta_{a,a'}+\delta_{b,b'}-2} L(a',b')\circ L(a,b).
\eneqn
 We can check easily that this composition
does not vanish.
Hence it coincides with
$\rmat{L(a,b),\;L(a',b')}$ by using (i) and Lemma~\ref{lem:Rr}.

\vs{2ex}

\noi
  Let us show (ii) (b).

\noi
By Proposition~\ref{prop:R}~\eqref{item:Rw},
we can write
\eqn
&&R_{L(a,b)_z, L(a',b')_{z'}}(u(a,b)_z \tensor u(a',b')_{z'})\\
&&\hs{8ex}=(z'-z)^{p}\tau_{\ell',\ell}(u(a',b')_{z'}\tens u(a,b)_z)
+a(z'-z)\tau_{\ell',\ell}(u(a',b')_{z'}\tens u(a,b)_z)\\
&&\hs{16ex}+\sum_{w\in A}g_w\tau_w (u(a',b')_{z'}\tens u(a,b)_z),
\eneqn
where $A=\set{w\in\sym_{\ell',\ell}}{w\not=w[\ell',\ell]}$,
$a(z'-z)$ is a polynomial of degree $<p$,
and $g_w\in \cor[z',z]$.

The degree of $(z'-z)^{p}\tau_{\ell',\ell}(u(a',b')_{z'}\tens u(a,b)_z)$
is equal to the degree of
$a(z'-z)\tau_{\ell',\ell}(u(a',b')_{z'}\tens u(a,b)_z)$. Hence $a(z'-z)$
should vanish.
Moreover, we may assume that $w\in A$ satisfies
$w\nu=w[\ell',\ell]\nu$ where $\nu=(a',\ldots, b',a,\ldots,b)$.
 We can easily see that
there is no such  $w$  except the case $a \le a' \le b \le b'$.
Hence we obtain (ii) (b).

\vskip 1em

The assertions (iii)-(vi) appeared in \cite[Lemma 4]{Vaz02} except the descriptions of homomorphisms.
We will describe the homomorphisms explicitly.

The  left arrow $\psi$  in (vi) is given by \eqref{mor:xi}.

The left arrow $\vphi$ in (v) is given by
$$
qL(a',b)\circ L(a,b')
\To[{\xi}]
 L(a,b)\circ L(a',a-1)\circ L(a,b')
\To L(a,b)\circ L(a',b').
$$
The right arrow in  (v) is given by
\eqn
L(a',b')\circ L(a,b)
\To[\xi]L(a',b')\circ q^{-1}L(b'+1,b)\circ L(a,b')\to
q^{-1}L(a',b)\circ L(a,b').
\eneqn
\vskip 1em

The assertion (vii) is a consequence of (ii).
\QED

\Rem
If we had chosen $Q_{i,i+1}(u,v)=v-u$, then
$r_{L(a,b),L(a,b)}$ would be $(-1)^{b-a}\id$.
\enrem

We give a total order on the set of segments as follows:
\begin{align} \label{eq:right order}
(a_1,b_1) > (a_2,b_2) && \text{if} \ a_1 > a_2 \ \text{or} \ a_1=a_2 \ \text{and} \ b_1 > b_2.
\end{align}

The following proposition was proved in \cite{KP11}.
The corresponding statement for affine Hecke algebras was proved in \cite[Theorem 2.2]{Vaz02} (see also \cite{BZ77, Z80}).
\begin{prop} \cite[Theorem 4.8, Theorem 5.1]{KP11} \label{prop:multisegments}
\bnum
\item
Let $M$ be a finite-dimensional simple  graded  $R(\ell)$-module.
Then there exist a unique pair of a multisegment
$\big ((a_1,b_1), \ldots , (a_t,b_t) \big)$ and an integer $c$ such that
\bna
\item $(a_k, b_k) \ge (a_{k+1}, b_{k+1})$ for $1 \leq k \leq t-1$,
\item $\sum_{k=1}^t \ell_k=\ell$, where $\ell_k \seteq b_k -a_k +1$,
\item
$ M \simeq q^c\hd \big(L(a_1,b_1)\circ \cdots \circ L(a_t,b_t) \big)$,
 where $\hd$ denotes the head.
\ee
\item Conversely, if  a multisegment   $\big ((a_1,b_1), \ldots , (a_t,b_t) \big)$  satisfies
$\rm (a)$ and $\rm (b)$, then
$\hd \big(L(a_1,b_1)\circ \cdots \circ L(a_t,b_t) \big)$
is a simple   graded  $R(\ell)$-module.
\ee
\end{prop}

If a multisegment $\big ((a_1,b_1), \ldots , (a_t,b_t) \big)$
satisfies the condition (a) above, then
we say that it is an {\em ordered multisegment}.
We call the  ordered  multisegment $\big ((a_k,b_k)\big)_{1\le k \le t}$
in Proposition \ref{prop:multisegments}~(i)
the {\em multisegment associated with $M$}.

The following lemma is more or less proved in \cite{Vaz02}
if we ignore the grading.
\Lemma\label{lem:rmat.mult}
Let $\big ((a_1,b_1), \ldots , (a_t,b_t) \big)$ be a multisegment
satisfying the conditions {\rm (a), (b)} in
{\rm Proposition~\ref{prop:multisegments}}.
Set $\beta_k=\oep_{a_k}-\oep_{b_k+1}$ and
$$d\seteq\sum_{1\le i<j\le t}\ds{L(a_i,b_i), L(a_j,b_j)}
=\sum_{1\le i<j\le t,\,\beta_i\not=\beta_j}(\beta_i,\beta_j).$$
 Also  set
$L=L(a_1,b_1)\circ \cdots \circ L(a_t,b_t)$ and $L'=L(a_t,b_t)\circ
\cdots \circ L(a_1,b_1)$. Then the following statements hold.
\bnum
\item $\hd(L)$
is isomorphic to $q^d\soc(L')$, where $\soc$ denotes the socle.
\item
$\bl\hd(L)\br^*\simeq q^{-2\sum_{1\le i<j\le t}\delta_{\beta_i,\beta_j}}\hd(L)$.
\item
For any $s\in\Z$ and any non-zero homomorphism $\phi\col L\to q^sL'$,
we have
$$\hd(L)\simeq \phi(L)\simeq\soc(q^sL').$$

\item We have
$$\Hom_{R(\beta)\gmod}(L,q^sL')\simeq
\begin{cases}\cor&\text{if $s=d$,} \\
0&\text{otherwise.}
\end{cases}$$
\ee
\enlemma
\begin{proof}
Let us rename the multisegment $\big ((a_1,b_1), \ldots , (a_t,b_t) \big)$ by
\begin{align*}
 &\big( (c_1,d_{1,1}),(c_1,d_{1,2}), \ldots , (c_1,d_{1,s_1}), (c_2,d_{2,1}),
(c_2,d_{2,2}), \ldots, (c_2,d_{2,s_2}), \\
 &\ldots,
(c_p,d_{p,1}), (c_p,d_{p,2}), \ldots, (c_p,d_{p, s_t}) \big)
\end{align*}
satisfying
\begin{align*}
  c_k > c_{k+1} \quad (1 \leq k < p) \quad \text{and} \quad d_{k,j}\ge d_{k,j+1}
\quad (1\le k\le p,\;1\leq j < s_k)
\end{align*}
and set
$$L^k \seteq L(c_k,d_{k,1}) \circ L(c_k,d_{k,2}) \circ \cdots \circ L(c_k,d_{k,s_k}) $$
for each $1 \leq k \leq p$.
Then $L \simeq L^1 \circ \cdots \circ L^p$.

The module $L^k$
is an irreducible $R(\gamma_{k})$-module
 by \cite[Lemma 5]{Vaz02}, where
$\gamma_{k,j} = \oep_{c_k}-\oep_{d_{k,j}+1}$ for $1 \leq j \leq s_k$
and $\gamma_k=\sum_{j=1}^{s_k}\gamma_{k,j}$. Define
$$\nu^k \seteq (\underbrace{c_k,\ldots,c_k}_{m_{k,c_k}}, \underbrace{c_k+1,\ldots,c_k+1}_{m_{k,1+c_k}}, \ldots, \underbrace{d_{k,1},\ldots,d_{k,1}}_{m_{k,d_{k,1}}})
\in I^{\gamma_k},$$
where $m_{k,j}$ denotes the number of occurrences of $j$ in
$$(c_k, c_k +1 \ldots, d_{k,1}; c_k, c_k +1, \ldots ,d_{k,2}; \ldots; c_k, c_k +1, \ldots d_{k, s_k}).$$
Then the shuffle lemma says that
$$\dim e(\nu^k) L^k = m_{k,c_k} !\; m_{k,1+c_k} ! \cdots m_{k,d_{k,1}} !. $$
Hence $e(\nu^k) L^k$ is a simple
$R(\beta_{k,c_k})\etens\cdots \etens R(\beta_{k,d_{k,1}})$-module,
and  $L^k$ is generated by $e(\nu^k) L^k$ as an $R(\gamma_{k})$-module.
Here $\beta_{k,j}=m_{k,j}\al_{j}$. Note that $\gamma_k=\sum_j\beta_{k,j}$.

On the other hand, by the shuffle lemma, we have
$$\dim e(\nu) L = \prod_{k=1}^p (m_{k,c_k} ! \cdots
m_{k,d_{k,1}} !),$$ where $\nu\seteq \nu^1 * \nu^2 * \cdots * \nu^p$
is the concatenation of the $\nu_k$'s. It follows that
$$e(\nu) L = e(\nu^1) L^1 \boxtimes \cdots \boxtimes e(\nu^p)L^p$$
is an irreducible $\bl R(\beta_{1,c_1}) \etens \cdots \etens
R(\beta_{1,d_{1,1}})\br \etens \cdots \etens \bl R(\beta_{p,c_p})
\etens \cdots \etens R(\beta_{p,d_{p,1}})\br$-module. Hence $e(\nu)
L$ is isomorphic to the tensor product of Kato modules up to a
grading shift and
\eq&& \ba{l} \dim_q e(\nu) L = q^A
\prod_{k=1}^p ([m_{k,c_k}] ! \cdots [m_{k,d_{{k,1}}}] !),
\\[2ex]
\Bigl(q^{-A}e(\nu) L\Bigr)^*
\simeq q^{-A}e(\nu) L
\ea
\label{eq:qdim}
\eneq
for some integer $A$.

Let $K$ be a submodule of $L$ such that $e(\nu) K \neq 0$. Then we have
$$e(\nu) K = e(\nu) L,$$
because $e(\nu) K$
is an $(R(\beta_{1,c_1}) \otimes \cdots \otimes R(\beta_{1,d_{1,1}})) \otimes \cdots \otimes (R(\beta_{p,c_p}) \otimes \cdots \otimes R(\beta_{p,d_{p,1}}))$-module.
Since $e(\nu^1) L^1 \boxtimes \cdots \boxtimes e(\nu^p)L^p$ generates
$L^1 \boxtimes \cdots \boxtimes L^p$ and $L^1 \boxtimes \cdots \boxtimes L^p$ generates $L$,
 $e(\nu) L$ generates $L$.
Thus $e(\nu) K$ generates $L$. It follows that for any proper
submodule $K$ of $L$, we have $e(\nu) K=0$. Hence $L$ has a unique
maximal submodule and therefore the head of $L$ is irreducible.
Moreover, we have
$$e(\nu) L
\isoto e(\nu) \hd(L).$$
By \eqref{eq:qdim}, we have
$$(\hd(L))^*\simeq q^{-2A}\hd(L).$$

\smallskip
Note that $L'\simeq q^{-B}L^*$  by \eqref{eq:dualconv},  where
$B=\sum_{1\le k<k'\le t}(\beta_k,\beta_{k'})$.
Hence $L'$
has a simple socle
and $e(\nu)\soc(L')\isoto
e(\nu) L'$.
Moreover, we have
$$\soc(L')\simeq q^{-B}\soc(L^*)  \simeq  q^{-B}(\hd(L))^*\simeq q^{-B-2A}\hd(L).$$

If $\phi\col L \rightarrow q^sL'$ is a non-zero homomorphism, then we have
$$e(\nu) \phi(L) \neq 0$$
because $e(\nu) L$ generates $L$ and hence
$e(\nu) \phi(L)$ generates $\phi(L)$.
Note that
 $e(\nu) q^sL'$ generates
the socle of $q^sL'$. Hence
we conclude that $\phi(L) = \soc (q^sL')\simeq q^{s-2A-B}\hd (L)$.
It follows that $s=B+2A$ and $\phi$ is equal to the composition
$$L\epito\hd(L)\isoto q^{B+2A}\soc(L')\monoto q^{B+2A}L'$$
up to a constant multiple.
Because
$\dim \Hom_{R(\beta) \gmod}(M,N) = \delta(M \simeq N)$ for any simple modules $M$ and $N$ in
$R(\beta) \gmod$ (see \cite[Corollary 3.19]{KL09}),
we have
\eqn
\Hom_{R(\beta) \gmod}(L, q^s L') \simeq \begin{cases}
 \cor & \text{if} \ s=B+2A, \\
 0 & \text{otherwise}.
\end{cases}
\eneqn

Now it remains to show that
$$A=-\sharp\set{(i,j)}{1\le i<j\le t, \beta_i=\beta_j}.$$

\bigskip
Set
$$  \dim_{q} e(\nu^k) L^k = q^{A_k}[m_{k,c_k}] ! [m_{k,1+c_k}] ! \cdots [m_{k,d_{k,1}}] !. $$
We have
$\prod_{s=1}^{m}\frac{(1-q^{-2s})}{(1-q^{-2})}=q^{-m(m-1)/2}[m]!$.
We can see easily that
$A_k=A'_k+A_k''$ with
\eqn
A_k'&=&-\sum_{{c_k}\le s\le d_{k,1}}\dfrac{m_{k,s}(m_{k,s}-1)}{2},\\
A''_k&=&\sharp\set{(u,v,j)}{1\le u<v\le s_k, c_k\le j\le d_{k,v},\;j<d_{k,u}}.
\eneqn
 Note that $A''_k$ is the largest degree of the Laurent polynomial
$\dim_q e(\nu^k) L^k$.

First we shall calculate $A'_k$.
We have $m_{k,s}=\sharp\set{u}{1\le u\le s_k,\,c_k\le s\le d_{k,u}}$.
Hence we have
\eqn
m_{k,s}(m_{k,s}-1)&=&\sharp\set{(u,v)}{1\le u\not=v\le s_k,\;
c_k\le s\le d_{k,u}, c_k \le s\le d_{k,v}}\\
&=&2\sharp\set{(u,v)}{1\le u<v\le s_k,\;
c_k\le s\le d_{k,u}, c_k \le s\le d_{k,v}}\\
&=&2\sharp\set{(u,v)}{1\le u<v\le s_k,\;
c_k\le s\le d_{k,v}},
\eneqn
because $d_{k,v} \le d_{k,u}$ for $1\le u<v \le s_k$.
It follows that
$$A'_k=-\sum_{1\le u<v\le s_k }(d_{k,v}- c_k +1).$$

Next we shall calculate $A''_k$.
We have
\eqn
&&A''_k=
\sharp\set{(u,v,j)}{1\le u<v\le s_k, c_k\le j\le d_{k,v}}\\
&&\hs{10ex}-\sharp\set{(u,v,j)}{1\le u<v\le s_k, j=d_{k,u}=d_{k,v}}.
\eneqn
The first term is equal to
\eqn
&&\sum_{1\le u<v\le c_k}(d_{k,v}-c_k+1)=-A'_k.
\eneqn
The second term is
$$\sharp\set{(u,v)}{1\le u<v\le s_k,\;d_{k,u}=d_{k,v}}=
\sharp\set{(i,j)}{1\le i<j\le t, \beta_i=\beta_j, a_i=c_k}.$$ Thus we
obtain $$A_k=-\sharp\set{(i,j)}{1\le i<j\le t, \beta_i=\beta_j,
a_i=c_k}$$ and hence
$$A=\sum_kA_k=-\sharp\set{(i,j)}{1\le i<j\le t, \beta_i=\beta_j}$$
as desired.
\end{proof}

\Prop \label{prop:irreducible to irreducible}
Let $M$ be a finite-dimensional  graded  simple $R(\ell)$-module and let
$$\big ((a_1,b_1), \ldots , (a_t,b_t) \big)$$
be the  ordered  multisegment associated with $M$.
Set
\eqn
&&\beta_k=\oep_{a_k}-\oep_{1+b_k}, \quad
d=\sum_{1\le i<j\le t,\,\beta_i\not=\beta_j}(\beta_i,\beta_j), \quad
L_k=L(a_k,b_k), \ \text{and} \\
&&r\seteq\rmat{L_1,\ldots, L_t}\col
L_1\circ\cdots\circ L_t\to q^d L_t\circ\cdots\circ L_1.
\eneqn
Then $M\simeq \Img (r)$.
\enprop
\Proof
By Propositions~\ref{prop:nonvan} and Lemma~\ref{lem:rmat.mult},
the morphism $r$ does not vanish and hence the result follows from
Lemma~\ref{lem:rmat.mult} (iii).
\QED

\Cor\label{cor:irred} Let $\{(a_k,b_k)\}_{1\le k\le t}$ be a sequence of segments.
If $L(a_j,b_j)\circ L(a_k,b_k)$ is simple for any $1\le j<k\le t$,
then $L(a_1,b_1)\circ\cdots \circ L(a_t, b_t)$ is simple.
\encor
\Proof
Under the assumption, $r_{L(a_j,b_j),L(a_k,b_k)}$ is an isomorphism
for any $1\le j<k\le t$.
 Hence $r\seteq r_{L(a_1,b_1),\ldots, L(a_t,b_t)}$ is an isomorphism so that $\Img (r)=q^{d}L(a_t,b_t) \circ \cdots \circ L(a_1,b_1)$.
By the above proposition, $q^d L(a_t,b_t) \circ \cdots \circ
L(a_1,b_1)$ is simple and so is $L(a_1,b_1) \circ \cdots \circ
L(a_t,b_t)$. \QED

\subsection{Properties of the functor $\F$}

 The {\em trivial representation} is the 1-dimensional $\UA$-module
on which $e_i$, $f_i$ act by $0$.
It is a unit object of the tensor category
$\UA\smod$.
For $k > N$ or $k < 0$, $V(\varpi_{k})$ is understood to be zero, and the modules
$V(\varpi_{0})$ and $V(\varpi_{N})$ are understood to be the trivial representation.

\begin{prop} \label{prop:image of fund. repns.}
Let $(a,b)$ be a segment with length $\ell\seteq b-a+1$.
 Then we have
\begin{align*}
\F(L(a,b)) \simeq V(\varpi_\ell)_{(-q)^{a+b}}\simeq\begin{cases}
0 & \text{if \/ $\ell >N$,} \\
V(\varpi_\ell)_{(-q)^{a+b}} & \text{if \/ $ 0\le \ell \le N $.}
\end{cases}
\end{align*}
\end{prop}
\begin{proof}
We will show our assertion by induction on $\ell$. In the course of
the proof, we omit the grading. We write
\eqn &&
\rmat{(a,b),(a',b')}\col L(a,b)\circ L(a',b')\To L(a',b')\circ
L(a,b) \eneqn for $\rmat{L(a,b),L(a',b')}$.

When $\ell =1$, we have $\F(L(a)) \simeq V_{(-q)^{2a}}$
by Proposition \ref{prop:image of tau} (i).

\noindent Assume that $\ell \ge 2$.
Consider the following exact sequence in $R(\ell) \smod$
\begin{align*}
  0 \rightarrow L(a,b) \rightarrow L(b) \circ L(a,b-1)
  \To[{\rmat{(b),(a,b-1)}}] L(a,b-1) \circ L(b)
 \rightarrow L(a,b) \rightarrow 0
\end{align*}
given in Proposition \ref{prop:exact sequences} (vi). Applying the
exact functor $\F$ and using the induction hypothesis, we obtain an
exact sequence \eq &&\ba{l}
0\to \F(L(a,b))\to V_{(-q)^{2b}}\tens V(\varpi_{\ell-1})_{(-q)^{a+b-1}}\\[2ex]
\hs{10ex}\To[{\;\F(\rmat{(b),(a,b-1)})\;}]
V(\varpi_{\ell-1})_{(-q)^{a+b-1}}\tens V_{(-q)^{2b}
}\to\F(L(a,b))\to0.\ea
\eneq

Now assume that $\ell\le N$.
It is known that there exists an exact sequence (\cite[Lemma B.1]{AK})
\eq
&&\ba{l}
0\to V(\varpi_\ell)_{(-q)^{a+b}}
\to V_{{(-q)}^{2b}}\tens V(\varpi_{\ell-1})_{(-q)^{a+b-1}}\\[2ex]
\hs{10ex}\To[\;h\;]V(\varpi_{\ell-1})_{(-q)^{a+b-1}}\tens V_{(-q)^{2b}
}\to  V(\varpi_\ell)_{(-q)^{a+b}}\to0\ea
\eneq
such that $h$ is non-zero.
If $\F(\rmat{(b),(a,b-1)})$ vanishes, then
$V_{(-q)^{2b}}\tens V(\varpi_{\ell-1})_{(-q)^{a+b-1}}$ and
$V(\varpi_{\ell-1})_{(-q)^{a+b-1}}\tens V_{(-q)^{2b}}$ are isomorphic,
which is a contradiction.
Hence $\F(\rmat{(b),(a,b-1)})$ does not vanish.

\smallskip\noi
We know
$$\Hom_{\UA} ( V_{(-q)^{2b}}\tens V(\varpi_{\ell-1})_{(-q)^{a+b-1}},
V(\varpi_{\ell-1})_{(-q)^{a+b-1}}\tens V_{(-q)^{2b}})\simeq\cor.$$
Since $\F(\rmat{(b),(a,b-1)})$ does not vanish, it is equal to $h$
up to a constant multiple and hence $\F(L(a,b))$ is isomorphic to $
V(\varpi_\ell)_{(-q)^{a+b}}$. Thus we have proved the proposition
when $\ell\le N$.

Now assume that $\ell=N+1$. Then $\F(L(a,b-1))\simeq\F(L(a-1,b))\simeq \cor$.
Applying $\F$ to  the  epimorphism
$L(a,b-1)\circ L(b)\epito L(a,b)$,
$\F(L(a,b))$ is a quotient of
$V_{(-q)^{2b}}$. Similarly, applying
$\F$ to   the   epimorphism
$L(a)\circ L(a+1,b)\epito L(a,b)$,
$\F(L(a,b))$ is a quotient of
$V_{(-q)^{2a}}$. Since $V_{(-q)^{2b}}$ and $V_{(-q)^{2a}}$ are
simple modules and they are not isomorphic to each other, we conclude that
$\F(L(a,b))$ vanishes.

For $\ell>N+1$,
$\F(L(a,b))$ vanishes since
it is a quotient of
$$\F(L(a,a+N))\tens\F(L(a+N+1,b))\simeq0.$$
\QED

\Lemma\label{lem:intertwiners} Assume that two segments $(a,b)$ and
$(a',b')$ satisfy $(a,b)\ge (a',b')$. Set $\ell=b-a+1$,
$\ell'=b'-a'+1$, $c=(-q)^{a+b}$ and $c'=(-q)^{a'+b'}$. Then the
following statements hold.
\bnum
\item
$c'/{c}$ is not a zero of the denominator
$d_{V(\varpi_\ell),V(\varpi_{\ell'})}(z'/z)$ of
$\Rnorm_{V(\varpi_\ell),V(\varpi_{\ell'})}(z,z')$,
\item
the homomorphism
  $$\F(\rmat{L(a,b),L(a',b')}) \col
  V(\varpi_{\ell})_{c} \otimes V(\varpi_{\ell'})_{c'} \rightarrow
V(\varpi_{\ell'})_{c'} \otimes V(\varpi_{\ell})_{c} $$
   is a non-zero constant multiple of the normalized R-matrix
  $\Rnorm_{V(\varpi_{\ell}), V(\varpi_{\ell'})} (c,c')$.
\ee
\enlemma
\begin{proof}
(i) follows from \eqref{eq:zerR}
because $(b'-a'+1)-(b-a+1)\ge (a'+b')-(a+b)$.

\smallskip
\noi
By (i) and  Theorem \ref{th:zero},
the module
$V(\varpi_{\ell})_{c} \otimes V(\varpi_{\ell'})_{c'}$
is generated by the dominant extremal vector
$v_\ell\tens v_{\ell'}$.
Since
$$\dim \Big( V(\varpi_{\ell'})_{c'}
\otimes V(\varpi_{\ell})_c \Big)_{\cl (\varpi_{\ell} + \varpi_{\ell'})}=1,$$
any non-zero homomorphism from
 $ V(\varpi_{\ell})_c \otimes V(\varpi_{\ell'})_{c'}$ to
  $V(\varpi_{\ell'})_{c'} \otimes V(\varpi_{\ell})_c$
is unique up to a constant multiple. Hence it is enough to show that
$\F(\rmat{L(a,b),L(a',b')})$ does not vanish.

We may therefore assume that $r\seteq\rmat{L(a,b),L(a',b')}$
is not an isomorphism.
Then we have $a'<a\le b'<b$   or $a=b'+1$.
Applying
$\F$ to the exact sequences (v) or (vi)
in Proposition~\ref{prop:exact sequences},
we obtain an exact sequence:
\eqn
&&\hs{-3ex}  0 \rightarrow V(\varpi_{\ell_1})_{(-q)^{a'+b}} \otimes V(\varpi_{\ell_2})_{(-q)^{a+b'}}
  \rightarrow V(\varpi_{\ell})_{c}
\otimes V(\varpi_{\ell'})_{c'} \\
&&\hs{3ex} \To[{\;\F(r)\;}]
  V(\varpi_{\ell'})_{c'} \otimes V(\varpi_{\ell})_{c} \rightarrow
  V(\varpi_{\ell_1})_{(-q)^{a'+b}} \otimes V(\varpi_{\ell_2})_{(-q)^{a+b'}} \rightarrow 0,
\eneqn
where  $\ell_1 = b-a' +1$ and $\ell_2 = b'-a +1$.

Since $\Big( V(\varpi_{\ell_1})_{(-q)^{a'+b}} \otimes
V(\varpi_{\ell_2})_{(-q)^{a+b'}} \Big)_{\cl (\varpi_{\ell} + \varpi_{\ell'})} = 0$,
we deduce that $\F(r)$ is a non-zero homomorphism
and hence it is a non-zero constant multiple of the normalized R-matrix.
\end{proof}

The following theorem will play a crucial role in the rest of this section.
\begin{theorem} \label{thm:irreducible to irreducible}
Let $M$ be a finite-dimensional irreducible  graded  $R(\ell)$-module and
$$\big ((a_1,b_1), \ldots , (a_t,b_t) \big)$$
be the multisegment associated with $M$. Set $\ell_k=b_k-a_k+1$.
\bnum
\item If $\ell_k > N$ for some $1 \leq k \leq t$, then $\F(M) \simeq 0$.
\item If $\ell_k \leq N$ for all $1 \leq k \leq t$, then $\F(M)$ is irreducible.
\ee
\end{theorem}
\begin{proof}
If $\ell_k > N$ for some $k$, the assertion follows from Proposition
\ref{prop:image of fund. repns.}. Assume that $\ell_k \leq N$ for
all $1 \leq k \leq t$. We know that $M$ is isomorphic to the image
of $r\seteq\rmat{L(a_1,b_1),\ldots, L(a_t,b_t)}$. Set
$V_k=V(\varpi_k)$ and $c_k=(-q)^{a_k+b_k}$. Then
$\F(L(a_k,b_k))\simeq (V_k)_{c_k}$ and $c_{k'}/{c_k}$ is not a zero
of the denominator $d_{V_k,V_{k'}}(z'/z)$ of
$\Rnorm_{V_k,V_{k'}}(z,z')$ for $k<k'$ by
Lemma~\ref{lem:intertwiners}. Hence Theorem~\ref{th:zero} says that
the image of the R-matrix
$$R\col (V_1)_{c_1}\tens\cdots \tens (V_t)_{c_t}\to
(V_t)_{c_t}\tens\cdots \tens (V_1)_{c_1}$$ is  irreducible. On the
other hand, $\F(r)$ is equal to $R$ up to a constant multiple by
Lemma~\ref{lem:intertwiners}. Hence $\F(M)$ is irreducible.
\end{proof}

 \subsection{Quotient of  the category $R\gmod$}
Set $\As_{\ell} = R({\ell}) \gmod$ and set $\As = \soplus_{{\ell} \ge 0} \As_{\ell}$.
Similarly, we  define $\Ab_\ell$ and
$\Ab$ by $\Ab_\ell=\Modg(R(\ell))$ and $\Ab=\soplus_{\ell\in\Z_{\ge0}}\Ab_\ell$.
Then we have a functor $\mathcal F =\bigoplus_{\ell \ge 0} \F_\ell\col \Ab \rightarrow
\Mod(\UA)$,
where $\mathcal F_{\ell}$ is the functor from $\Ab_{\ell}$ to $\Mod(\UA)$ given in \eqref{eq:the functor}.

Let $\Ss$ be the smallest Serre subcategory
of $\As$ (see Appendix~\ref{app:Serre}) such that
\eq&&
\parbox{70ex}{
\begin{enumerate}
\item $\Ss$ contains $L(a, a+N)$ for any $a\in\Z$,
\item $X \circ Y, \ Y \circ X \in \Ss$
for all $X \in \As$ and $Y \in \Ss$.
\end{enumerate}}
\eneq
Note that $\Ss$ contains $L(a,b)$ if $b\ge a+N$.

Let us denote by $\As/ \Ss$ the quotient category of $\As$ relative to $\Ss$
and denote by $\mathcal Q\col \As \rightarrow \As/ \Ss$ the canonical functor.
Since $\F$ sends $\Ss$ to $0$,
the functor $\F \col \As \rightarrow
\UA \smod$ factors through $\mathcal Q$
by Theorem \ref{thm:quotient category} \eqref{q6}:
\eqn&&
\xymatrix@C=6ex@R=4ex{\As\ar[r]^{\mathcal{Q}}\ar[dr]_-{\F}
&\As/ \Ss\ar[d]^-{\F'}\\
&{\UA\smod}
}
\eneqn
for a functor $\F'\col \As/\Ss\to\UA \smod$.

\smallskip

Note that $\As$ and $\As/\Ss$ are tensor categories with the
convolution as tensor product. The module $R(0)\simeq \cor$ is a
unit object. Note also that $Q\seteq q R(0)$ is an invertible
central object of $\As/\Ss$ and $X\mapsto  Q\circ X\simeq X\circ Q$
coincides with the grading shift functor. The functors $\mathcal Q$,
$\F$ and $\F'$ are tensor functors.

\smallskip

Similarly, we define $\Sb$ as the smallest Serre subcategory of $\Ab$ such that
\eq&&
\parbox{70ex}{
\begin{enumerate}
\item $\Sb$ contains $L(a, a+N)$,
\item $X \circ Y, \ Y \circ X \in \Sb$
for all $X \in \Ab$, $Y \in \Sb$,
\item $\Sb$ is stable under (not necessarily finite) direct sums.
\end{enumerate}}
\eneq
Then we can easily see that $\Sb\cap \As=\Ss$ and hence we have
\eqn
&&\text{The functor $\As/\Ss\to\Ab/\Sb$ is fully faithful.}
\eneqn

\begin{prop}\hfill

\bna
\item If an object $X$ is simple in $\As / \Ss$, then
there exists a  simple object  $M$ in $\As$ satisfying

\bni
\item $\mathcal Q(M) \simeq X$,
\item $b_k-a_k+1 \leq N$ for $1 \leq k \leq r$, where
$\big ( (a_1,b_1), \ldots , (a_r,b_r) \big)$ is the multisegment associated with $M$.
\ee
\item Let $\big ( (a_1,b_1), \ldots , (a_r,b_r) \big)$  be the multisegment associated with a simple  object  $M$ in $\As$.
If $b_k-a_k+1 \leq N$ for $1 \leq k \leq r$, then
$\mathcal Q(M)$ is simple in $\As / \Ss$.
\ee
\begin{proof}
(a) If $X$ is simple in $\As / \Ss$, then there exists an irreducible module
$M \in \As$ such that $\mathcal Q(M) \simeq X$
by Proposition \ref{prop:simples in quotient cat.} (b).
Let $\big ( (a_1,b_1), \ldots , (a_r,b_r) \big)$ be the multisegment associated with $M$. Then
$ M \simeq \hd \big(L(a_1,b_1)\circ \cdots \circ L(a_r,b_r) \big)$
in $R(\ell) \smod$, by Proposition \ref{prop:multisegments}.
If $b_k-a_k+1 > N$ for some $1 \leq k \leq r$, then $\mathcal Q(M)=0$
by the definition of $\Ss$.
Since $X \simeq \mathcal Q(M)$ is simple, it is a contradiction.

\smallskip\noi
(b)\quad Since $M$ is irreducible, $\mathcal Q(M)$ is zero or simple in $\As / \Ss$.
If $ M \simeq \hd \big(L(a_1,b_1)\circ \cdots \circ L(a_r,b_r) \big)$ and $b_k-a_k +1 \leq N$ for all $1 \leq k \leq r$, then $\F(M) \not\simeq0$
by Theorem \ref{thm:irreducible to irreducible}.
It follows that $\mathcal Q(M) \not\simeq0$.
\end{proof}
\end{prop}

We obtain the following corollary as an immediate consequence.
\begin{corollary}
The functor $\F'\col\As/\Ss\to \UA\smod$ sends simple objects in
$\As/\Ss$ to simple objects in $\UA\smod$.
\end{corollary}

\subsection{The category $\T'_J$}

 Since   all the images of $L(a,a+N-1)$  under $\F'$  are isomorphic
to the trivial representation of $\UA$,
 we can localize  $\As/\Ss$  one step further by
using Appendix~ \ref{app:Local_tensor_categories}.

Set
\begin{equation}
L_a\seteq L(a, a+N-1)  \quad \text{and} \quad
  u_a\seteq u(a,a+N-1)\in L_a
\quad\text{for} \  a \in \Z.
\end{equation}
Then $\F(L_a)$ is isomorphic to the trivial representation of
$\UA$.

The following proposition will play a central role in the rest of this section.

\begin{prop}  Let $a,j\in\Z$ and set
\eqn
p&=&(\oep_a-\oep_{a+N},\al_j)-2\delta(a\le j\le a+N-1)\\
&=&-\delta_{j,a}-\delta_{j,a-1}-\delta_{j,a+N-1}-
\delta_{j,a+N}-2\delta(a<j<a+N-1).\eneqn
\bnum
    \item  The image of the morphism $R_{L_a, L(j)_z} \col L_a \circ L(j)_z \to
q^{p} L(j)_z \circ L_a$ is contained in
\eqn q^{p}z^{\delta({a\le j<a+N-1})}L(j)_z \circ L_a.\eneqn
\item  The image of the morphism $R_{L(j)_z, L_a } \col L(j)_z \circ L_a \to
q^{p}L_a \circ L(j)_z$ is contained in
\eqn q^{p}z^{\delta(a<j\le a+N-1)} L_a \circ L(j)_z.\eneqn
\item If $j \neq a-1,a+N $, then the morphisms
$$z^{-\delta(a\le j<a+N-1)}R_{L_a, L(j)_z}\col L_a \circ L(j)_z
\to  q^{\delta_{j,a}-\delta_{j,a+N-1}}L(j)_z \circ L_a$$
and
$$(-1)^{\delta( a < j \le a+N-1 )}z^{-\delta(a< j\le a+N-1)}R_{L(j)_z,L_a }\col
 q^{\delta_{j,a}-\delta_{j,a+N-1}}L(j)_z \circ L_a\to L_a \circ L(j)_z $$
are isomorphisms and the inverses to each other.
\item If $j=a-1$, then
we have  a commutative diagram with an exact row:
$$
\scalebox{.79}{\xymatrix@C=10ex{
0\ar[r]&{ L_a \circ L(a-1)_z\rule[-1.2ex]{0ex}{1.5ex}}
\ar[r]^-{R_{L_a, L(a-1)_z}}\ar@{^{(}->}[dr]&
q^{-1}L(a-1)_z\circ L_a\ar[r] \ar@{>->}[d]^{z^{-1}R_{L(a-1)_z, L_a}}
& q^{-1} L(a-1,a+N-1)\ar[r]&0\\
&&L_a\circ z^{-1}L(a-1)_z\,.}}
$$
\item  If $j=a+N$,
we have  a commutative diagram with an exact row:

\eqn&&\hs{-4ex}
\scalebox{.85}{\xymatrix@C=10ex{
0\ar[r]&{qL(a+N)_z \circ L_a\;\rule[-1.2ex]{0ex}{1.5ex}}\ar[r]^{qR_{L(a+N)_z,L_a}}\ar@{^{(}->}[dr]
& L_a \circ L(a+N)_z \ar[r]
\ar@{>->}[d] ^{-z^{-1}R_{L_a, L(a+N)_z}}&L(a,a+N)\ar[r]&0\\
&&qz^{-1}L(a+N)_z \circ L_a\,.} } \eneqn \ee \enprop
\Proof (i) and
(ii) immediately follow from Lemma~\ref{lem:expliciteR}.

\smallskip
\noi
(iii)\quad For any $j\in \Z$,
 set
$h_j(z)=\prod_{a\le k\le a+N-1, k\not=j}Q_{j,k}(z,0)$.
Then we can easily see  that
$$h_j(z)=\begin{cases}-z^2&\text{if $a<j<a+N-1$,}\\
(-1)^{\delta(j=a+N-1,a+N)}z&\text{if  $j=a-1,a,a+N-1,a+N$,}\\
1&\text{otherwise}
\end{cases}
$$
which can be rewritten as
$$h_j(z)=(-1)^{\delta(a < j \le a+N-1)}z^{\delta(a\le j<a+N-1)+\delta(a<j\le a+N-1)}$$
 for $j \ne a-1, a+N$.
Then Lemma~\ref{lem:ga}~  \eqref{ga6}  implies
\eq
&&
\ba{l}R_{L_a, L(j)_z}\circ R_{L(j)_z,L_a}=h_j(z)\id_{L(j)_z \circ  L_a},
\\[2ex]
 R_{L(j)_z,L_a}\circ R_{L_a, L(j)_z}=h_j(z)\id_{L_a \circ L(j)_z}.
\ea
\label{eq:RR=1}
\eneq
Hence we obtain (iii).

\medskip\noi
(iv)\quad By Lemma~\ref{lem:expliciteR}, we have
$$R_{L_a, L(a-1)_z }(u_a\tens u(a-1)_z)
=\tau_N\cdots \tau_1\bl u(a-1)_z\tens u_a\br.$$

\hs{-1ex} On the other hand, we have
$$(\tau_2\cdots\tau_N)\Bigl(\tau_N\cdots \tau_1\bl u(a-1)_z\tens u_a\br\Bigr)
=\tau_1\bl u(a-1)_z\tens u_a\br,$$ which implies
$$\Img(R_{L_a, L(a-1)_z })=R(  \oep_{a-1} - \oep_{a+N}  )\tau_1\bl u(a-1)_z\tens u_a\br.$$
 Note that we  have $\tau_1^2\bl u(a-1)_z\tens u_a\br=x_1\bl u(a-1)_z\tens u_a\br$.
 Therefore  we obtain
\eqn
&&\Ker\bl L(a-1)_z\circ L_a\to L(a-1,a+N-1)\br\\
&&\hs{2ex}=R(  \oep_{a-1} - \oep_{a+N} )x_1\bl u(a-1)_z\tens u_a\br
+R(  \oep_{a-1} - \oep_{a+N} )\tau_1\bl u(a-1)_z\tens u_a\br \\
&&\hs{2ex}=\Img(R_{L_a, L(a-1)_z }).
\eneqn
The other parts are derived from
$R_{L(a-1)_z, L_a}\circ R_{L_a,L(a-1)_z}=z \id$
which is obtained by \eqref{eq:RR=1}.

\smallskip\noi
(v) is proved similarly to (vi).
\QED

\medskip
Define an abelian group homomorphism
\begin{align*}
c_a \col \rootl_J \rightarrow \Z
\end{align*}
as
\begin{align}
    c_a(\alpha_j) \seteq (\oep_a+\oep_{a+N},\al_j)=
  \begin{cases}
0 & \text{if} \ j \neq a, a-1,a+N-1, a+N, \\
-1 & \text{if} \ j = a-1, \\
1 & \text{if} \ j = a, \\
-1 & \text{if} \ j = a+N-1, \\
1 & \text{if} \ j = a+N.
  \end{cases}
\end{align}
Then, by the preceding proposition,
we obtain a homomorphism
\eq
&&\ba{l}
z^{-\delta(a\le j < a+N-1)-\delta_{j,a+N}}
R_{L_a,L(j)_z}\col L_a\circ L(j)_z\to q^{c_a(\al_j)}
z^{-1}L(j)_z\circ L_a.
\ea\label{eq:int:shift}
\eneq
For $a,j\in\Z$, set
\eq \label{eq:faj}
&&f_{a,j}(z)=(-1)^{\delta_{j,a+N}}z^{-\delta(a\le j<a+N-1)-\delta_{j,a+N}}.
\eneq

Then we  obtain   the following corollary.
\Cor \label{cor:oldR} For any $a,j\in J$, the $R(N+1)$-module homomorphism
$$f_{a,j}(z)R_{L_a,L(j)_z}\col
L_a\circ L(j)_z\To
q^{c_a(\al_j)}z^{-1}L(j)_z\circ L_a$$
induces an isomorphism
$$f_{a,j}(z)R_{L_a,L(j)_z}\col
L_a\circ L(j)_z\To
q^{c_a(\al_j)} L(j)_z\circ L_a$$
in $\Ab/\Sb$.
\encor

\Rem
 To make Corollary~\ref{cor:oldR} hold,
it is enough to take
$$f_{a,j}(z)=cz^{-\delta(a\le j<a+N-1)-\delta_{j,a+N}}$$ for an arbitrary $c\in\cor^{\times}$.
We take $(-1)^{\delta_{j,a+N}}$ as $c$  so  that
(ii) and (iii) in Theorem~\ref{thm:L_a commutes with X} below hold.
\enrem

For $\beta\in \rtl_J^+$, we define commutative algebras
\eqn
&&  \bP_{\beta} \seteq\soplus_{\nu\in J^\beta}\cor[x_1,\ldots, x_\ell]e(\nu)\subset R(\beta), \\
&& K(\beta)\seteq\soplus_{\nu\in J^\beta}\cor[x_1^{\pm1},\ldots, x_\ell^{\pm1}]e(\nu), \text{\ and}\\
&& \KR(\beta)\seteq K(\beta)\tens_{ \bP_{\beta} }R(\beta)\supset  R(\beta),
\eneqn
where $\ell=|\beta|$.
 Set
\eqn
&&f_{a,\beta}=\sum_{\nu\in I^\beta}\prod_{k=1}^{\ell}f_{a,\nu_k}(x_k)e(\nu)
\in K(\beta)\subset \KR(\beta).
\eneqn
It belongs to the center of $\KR(\beta)$.
Hence we may consider $f_{a,\beta}$
as an $R(\beta)$-module endomorphism of
$\KR(\beta)$.

Since
$R(\beta)\simeq
\soplus_{\nu\in J^\beta}(L(\nu_1))_{z_1}\circ \cdots\circ (L(\nu_\ell))_{z_\ell}$,
we can apply the preceding proposition to study $L_a\circ R(\beta)$ and
$R(\beta)\circ L_a$.
Then Corollary~\ref{cor:oldR} yields the following proposition.

\Prop\label{prop:risom}
 For any $\beta\in\rtl_J^+$,
the $R(\oep_a-\oep_{a+N}+\beta)$-module homomorphism
$f_{a,\beta}R_{L_a, R_K(\beta)}\col L_a\circ \KR(\beta)\to
q^{c_a(\beta)}\KR(\beta)\circ L_a$
induces an isomorphism
$$L_a\circ R(\beta)\isoto q^{c_a(\beta)}R(\beta)\circ L_a$$
in $\Ab/\Sb$.
\enprop

\Proof By  Corollary~\ref{cor:oldR},   the assertion holds for
$|\beta|=1$. The general case immediately follows from this since
$\Ab/\Sb$ is a tensor category with the convolution $\circ$ as
tensor product by Proposition \ref{prop:quot tensor}. \QED

Note that the morphism
 $f_{a,\beta}R_{L_a,R(\beta)}\col L_a\circ R(\beta)\isoto q^{c_a(\beta)}R(\beta)\circ L_a$
in  $\Ab/\Sb$
commutes with the right action of $R(\beta)$.

\Lemma
Let $S$ be the automorphism of $P_J=\soplus_{a\in\Z}\Z\oep_a$
given by $S(\oep_a)=\oep_{a+N}$.
We define the bilinear form $B$ on $P_J$ by
\eq
&&B(x,y)=-\sum_{k>0}(S^kx,y)\qtext{for $x,y\in P_J$.}
\label{def:S}
\eneq
Then we have
$$c_a(x)=B(x,\oep_a-\oep_{a+N})-B(\oep_a-\oep_{a+N},x)$$
for any $x\in\rtl_J$.
\enlemma

\Proof
Set
 $\beta_a=\oep_a-\oep_{a+N}$.
Then we have
$$B(\beta_a,x)=-(\oep_{a+N},x).$$
On the other hand, $\sum_{k\in\Z}(x,S^k\beta_a)=0$ implies that
$$B(x,\beta_a)=-\sum_{k>0}(S^kx,\beta_a)
=-\sum_{k>0}(x,S^{-k}\beta_a)=\sum_{k\ge0}(x,S^k\beta_a)=(x,\oep_a).$$
\QED

For $\al\in\rtl^+_J$, set $\As_\al=R(\al)\gmod$ and
$\Ss_\al=\Ss\cap\As_\al$. Then $\Ss_\al$ is a Serre subcategory of
$\As_\al$ and we have
$\As/\Ss=\soplus_{\al\in\rtl^+_J}(\As/\Ss)_\al$, where
$(\As/\Ss)_\al=\As_\al/\Ss_\al$. Then Proposition~\ref{prop:risom}
yields an isomorphism
$$ q^{B(\beta_a,\al)} L_a\circ X\isoto  q^{B(\al,\beta_a)} X\circ L_a$$
in $\As/\Ss$  functorial in $X\in (\As/\Ss)_{\al}$.
\Def
We define the new tensor product
$\nconv\col\Ab/\Sb\times \Ab/\Sb\to\Ab/\Sb$ by
\eqn
X\nconv Y=q^{B(\al,\beta)} X\circ Y\simeq Q^{\tens B(\al,\beta)}\circ X\circ Y,
\eneqn
 where  $X\in(\Ab/\Sb)_{\al},$ $Y\in(\Ab/\Sb)_{\beta}$ and $Q=q\,\one$.
\edf Then $\Ab/\Sb$ as well as $\As/\Ss$ is endowed with a new
structure of tensor category by $\nconv$ as shown in
Appendix~\ref{app:twist}. With this tensor category structure,
Proposition~\ref{prop:risom} can be rephrased as follows.
\Lemma\label{lem:com} For any $a\in J$ and $\beta\in\rtl^+_J$, the
$\bl R(\oep_a-\oep_{a+N-1}+\beta),R(\beta)\br$- bimodule
homomorphism
$$f_{a,\beta}R_{L_a,\KR(\beta)}\col
L_a\nconv \KR(\beta)\To
\KR(\beta)\nconv L_a$$
induces an isomorphism
$$f_{a,\beta}R_{L_a,R(\beta)}\col
L_a\nconv R(\beta)\To
R(\beta)\nconv L_a$$
in $\Ab/\Sb$  which commutes
with  the right actions of $R(\beta)$.
\enlemma

\begin{theorem} \label{thm:L_a commutes with X}
The family $\{L_a\}_{a\in J}$
is a commuting family of central objects in $\As/\Ss$ \ro see
{\rm \S \ref{app:Local_tensor_categories}}\rf.
Namely,  the following statements  hold.
\bnum
\item
$L_a$ is a central object   in  $\As/\Ss$ \ro see
{\rm\S\,\ref{sec:central}}\rf; i.e., \bna
\item
$f_{a,j}(z)R_{L_a,L(j)_z}$ induces an isomorphism in $\As/\Ss$
$$R_a(X)\col L_a\nconv X\isoto X\nconv L_a$$
 functorial in  $X\in\As/\Ss$,
\item the diagram
$$\xymatrix@C=10ex
{L_a\nconv X\nconv Y\ar[r]^{R_a(X)\nconv Y}\ar[dr]_-{R_a(X\nconv Y)\hs{2ex} }&
X\nconv L_a\nconv Y\ar[d]^{X\nconv R_a(Y)}\\
& X\nconv Y \nconv L_a }
$$ commutes in $\As/\Ss$ for any $X,Y\in\As/\Ss$.
\ee
\item The  isomorphism
$R_a(L_a)\col L_a\nconv L_a\isoto L_a\nconv L_a$
coincides with $\id_{L_a\nconv L_a}$ in $\As/\Ss$.
\item
For $a,b\in\Z$,
 the  isomorphisms
\eqn
 R_a(L_b)\col L_a\nconv L_b\isoto L_b\nconv L_a \ \text{and} \ R_b(L_a)\col L_b\nconv L_a\isoto
L_a\nconv L_b
\eneqn

 in $\As/\Ss$ are  the inverses  to each other.
\ee
\enth

\medskip
\noi
\Proof (i)\,(a) follows from Lemma~\ref{lem:com} and the fact that
$L_a\nconv X\simeq\bl L_a\nconv R(\beta)\br\tens_{R(\beta)} X$
and
$X\nconv L_a\simeq\bl R(\beta)\nconv L_a\br\tens_{R(\beta)} X$.

\medskip

\noi
(i)\,(b)
\ For $\beta,\gamma\in\rootl_J^+$  ($|\beta|=\ell, |\gamma| =\ell'$),  we have
\eqn
&&\bl R(\beta) \nconv R_a(R(\gamma))\br \circ \bl R_a(R(\beta))\nconv R(\gamma)\br
\bl u_a\tensor e(\beta) \tensor e(\gamma)\br \\
&&\hs{5ex}= \bl R(\beta) \nconv R_a(R(\gamma))\br
\varphi_{\ell,N} f_{a,\beta}\bl e(\beta) \tensor u_a \tens e(\gamma)\br \\
&&\hs{5ex}=\varphi_{\ell,N} f_{a,\beta}
\bl e(\beta)\tens
\vphi_{\ell',N} f_{a,\gamma}(e(\gamma)\tensor u_a)\br\\
&&\hs{5ex}=\varphi_{\ell+\ell',N}
(f_{a,\beta}  \etens f_{a,\gamma})(e(\beta) \tensor e(\gamma) \tensor u_a)
\eneqn
and
\begin{align*}
&R_a(R(\beta)\nconv R(\gamma))(u_a \tensor e(\beta) \tensor e(\gamma))
= \varphi_{\ell+\ell',N} f_{a,\beta+\gamma} (e(\beta) \tensor e(\gamma) \tensor u_a).
\end{align*}
Since
$$f_{a,\beta+\gamma} e(\beta) \etens e(\gamma)=
(f_{a,\beta} \etens f_{a,\gamma})  e(\beta) \etens e(\gamma),$$
we obtain
$$R_a(R(\beta)\nconv R(\gamma))=\bl R(\beta) \nconv R_a(R(\gamma))\br
\cdot \bl R_a(R(\beta))\nconv R(\gamma)\br.$$
We obtain (i)\,(b) by applying
$\tens_{R(\beta)}X$ and $\tens_{R(\gamma)}Y$
for $X\in\As_\beta$ and $Y\in \As_\gamma$.

\bigskip
If $X \in (\AA /S)_\beta$, we have $R_a(X_z)|_{z=0} = R_a(X)$.
If $|\beta|=\ell$ and $x_1, \ldots, x_\ell$ act by $0$ on $X$, then we have
\begin{align*}
  R_a(X_z)(u_a \tensor x_z) = &
   f_{a,\beta}(z) R_{L_a,X_z}(u_a \tensor x_z)
\end{align*}
for $x \in X$, where $f_{a,\beta}(z) \seteq f_{a,\beta}|_{x_1=\cdots=x_\ell=z}$.

\vs{2ex}\noi
(ii)\quad
When  $X=L_a$,  we have $f_{a,\beta}(z)= z^{-(N-1)}$ and hence
$$R_a(L_a)=R_a((L_a)_z)|_{z=0}=r_{L_a,L_a}=\id_{L_a}$$
by Proposition \ref{prop:exact sequences} (i).

\smallskip\noi
(iii)\quad By (ii), we may assume that $a<b$.
It follows from Lemma \ref{lem:ga}~  \eqref{ga6}  that
$$R_{(L_b)_{z'}, (L_a)_z}\circ R_{(L_a)_z,(L_b)_{z'}}
=\prod_{ \substack{{a} \le i\le a+N-1, \\ {b} \le j\le b+N-1,\,i\not=j} }
Q_{ij}(z,z').$$

Set $\beta_a=\oep_{a}-\oep_{a+N}$
Then
$$
f_{a,\beta_b}(z'-z)\ R_{(L_a)_z,(L_b)_{z'}}
\col(L_a)_z\nconv (L_b)_{z'}\to
(L_b)_{z'}\nconv (L_a)_{z}$$
and
$$
f_{b,\beta_a}(z-z')\ R_{(L_b)_{z'},(L_a)_{z}} \col
(L_b)_{z'}\nconv (L_a)_{z}
\to (L_a)_z\nconv (L_b)_{z'}$$
specialize to
 $R_a(L_b) \colon L_a \nconv L_b \isoto L_b \nconv L_a $ and $R_b(L_a) \colon L_b \nconv L_a \isoto L_a \nconv L_b$.
 Note that $f_{a,\beta_b}(z)=\prod_{j=b}^{b+N-1} f_{a,j}(z)$.
Hence,  to prove our claim,  it is enough to show that
\eq
f_{a,\beta_b}(z'-z)f_{b,\beta_a}(z-z')
\prod_{\substack{{a} \le i\le a+N-1, \\ {b} \le j\le b+N-1,\,i\not=j}}
Q_{ij}(z,z')=1.\label{eq:ffQ}
\eneq

Set
\eqn
A(a,b)&=&\set{i}{a\le i+1\le a+N-1,\; b\le i\le b+N-1}\\
&=&\set{i}{a-1\le i\le a+N-2,\; b\le i\le b+N-1}.
\eneqn
Then we have
$$\prod_{\substack{{a} \le i\le a+N-1, \\ {b} \le j\le b+N-1,\,i\not=j}}
Q_{ij}(z,z')
=(z'-z)^{\sharp A(a,b)}(z-z')^{\sharp A(b,a)}.
$$
 Similarly,  set
$$B(a,b)=\set{i}{a\le i\le a+N,\;i\not=a+N-1,\;b\le i\le b+N-1}.$$
Then we have
$$f_{a,\beta_b}(z'-z)=
(-1)^{\delta(b\le a+N\le b+N-1)}(z'-z)^{-\sharp B(a,b)}.$$
 Therefore,  we obtain
\eqn\sharp A(a,b)-\sharp B(a,b)&=&\delta(b\le a-1\le b+N-1)-\delta(b\le a+N\le b+N-1)\\
&=&\delta(1\le a-b\le N)-\delta(1\le b-a\le N),
\eneqn
 which proves  \eqref{eq:ffQ}.
\end{proof}

By the preceding theorem, $\{(L_a, R_a)\}_{a \in J}$
forms a commuting family
of central objects in $(\AA / \Ss , \nconv)$
( \S\;\ref{app:Local_tensor_categories}).
Following Appendix \ref{app:Local_tensor_categories}, we localize $(\AA / \Ss , \nconv)$ by this commuting family.
Let us denote by  $\T'_J$
the resulting category $(\AA / \Ss)[L_a^{\nconv -1}\mid a\in J]$ .
Let $\Upsilon \colon \AA / \Ss \to \T'_J$ be the  projection functor.
We denote by $\T_J$ the tensor category
$(\As / \Ss)[L_a\simeq\one\mid a\in J]$ and by $\Xi \col
\T'_J \to \T_J$ the canonical functor (see \S\;\ref{app:graded}
and the remark below).
Thus we have a chain of  functors 
$$\As \To[\ {\mathcal Q}\ ]\As/\Ss\To[\ \Upsilon\ ] (\AA / \Ss)[L_a^{\nconv -1}\mid a\in J]
\To[\ \Xi\ ]
(\As / \Ss)[L_a\simeq\one\mid a\in J].$$

\Rem
Note that $\As$, $\As/\Ss$ and $\T'_J$ are $\rtl_J$-graded
(namely, $\T'_J$ has a decomposition $\T'_J=\soplus_{\al\in\rtl_J}(\T'_J)_\al$, etc.).
The category $\T_J$ is $\rtl_{J,N}$-graded   with
$\rtl_{J,N}\seteq\rtl_{J}/\sum_{a\in\Z}\Z\beta_a$,
 where $\beta_a=\eps_a-\eps_{a+N}$.
Note that $\rtl_{J,N}\simeq\soplus_{k=1}^{N-1}\Z\al_k$.

\enrem

\subsection{Rigidity of  the tensor categories  $\T'_J$ and $\T_J$}
In this subsection,
we will show that the tensor category $\T'_J$ is rigid;
i.e., every object in $\T'_J$ has a left dual and a right dual.
 The rigidity of  $\T_J$ follows from this fact.

Let $\ell$ be a  non-negative  integer and $a\in J$. We set
$\beta_a = \oep_a -\oep_{a+N}$ and $\gamma_a=\beta_a -\alpha_a=
\oep_{a+1}-\oep_{a+N}$.  Set
\begin{align*}
  K_\ell(a) \seteq e(\alpha_a,\gamma_a+\ell\beta_a) L_a^{\circ (\ell+1)}
\in R(\gamma_a+ \ell \beta_a) \gmod.
\end{align*}

By the shuffle lemma, we know   $$L_a^{\circ (\ell+1)} = e(\alpha_a, \gamma_a + \ell \beta_a) L_a^{\circ (\ell+1)}.$$
Hence $K_\ell(a)$ is isomorphic to $L_a^{\circ (\ell+1)}$ as a vector space.
For example, we have
$$K_0(a) \simeq L(a+1, a+N-1) \in R(\gamma_a) \gmod.$$
 Let
$$L_\ell(a) \seteq L(a)_z / z^{\ell+1} L(a)_z \in R(\alpha_a) \gmod.$$
We denote by $u_\ell(a)\in L_\ell(a)$ the image of $u(a)_z\in L(a)_z$.

\Prop \label{prop:rigid}
For $\ell\ge 0$,
let us denote by $z$
the  $R(\gamma_a+\ell \beta_a)$-module endomorphism of $K_\ell (a)$
given by the action of $x_1$ on $L_a^{\circ(\ell+1)}$.
Then we have
\bni
\item $z^{\ell+1}=0$,
\item $\Ker z = \Img z^{\ell} = R(\alpha_a, \gamma_a+\ell \beta_a)
  u_a^{\tensor (\ell+1)} \simeq K_0(a) \circ L_a^{\circ \ell}$,
\item $\Ker z^{\ell} = \Img z = R(\alpha_a, \gamma_a+ \ell \beta_a)
R(\ell \beta_a,\beta_a) u_a^{\tensor(\ell+1)} \simeq K_{\ell-1}(a) \circ L_a$.
\ee
\enprop
\Proof
Set
$u = u_a^{\tensor(\ell+1)} \in L_a^{\circ (\ell+1)}.$
Then
$$ R(\alpha_a,\gamma_a+\ell \beta_a) u  \simeq K_0(a) \circ L_a^{\circ \ell},$$
which is an irreducible $R(\gamma_a +\ell \beta_a)$-module
by Corollary~\ref{cor:irred}.
It is obvious that  $R(\alpha_a, \gamma_a + \ell \beta_a) u \subset \Ker z$.

In order to show the converse inclusion, let us prove
\begin{equation}
  \label{eq:nonzero elem in img z ell}
z^{\ell} \tau_1 \cdots \tau_{\ell N} \ u =
 (-1)^\ell (\tau_2 \cdots \tau_N) \ (\tau_{N+2} \cdots \tau_{2N}) \cdots
 (\tau_{(\ell-1) N+2} \cdots \tau_{\ell N}) \ u.
\end{equation}
If $\ell=0$, it is trivial.
Set $a(u,v) = \dfrac{u^\ell-v^\ell}{u-v}$.
Then we have
\begin{align*}
  x_1^\ell \tau_1 \tau_2 \cdots \tau_{\ell N}  \ u
&=\bl \tau_1 x_2^\ell -a(x_1,x_2)\br \tau_2 \cdots \tau_{\ell N} \ u \\
 &= \tau_1 \tau_2 \cdots \tau_N x_{N+1}^\ell \tau_{N+1} \cdots \tau_{\ell N} \ u
 - \tau_2 \cdots \tau_N a(0, x_{N+1}) \tau_{N+1} \cdots \tau_{\ell N} \ u.
\end{align*}
By induction on $\ell$, we obtain
$$x_{N+1}^\ell \tau_{N+1} \cdots \tau_{\ell N} \ u =
x_{N+1}(-1)^{\ell-1}(\tau_{N+2} \cdots \tau_{2N}) \cdots
 (\tau_{(\ell-1) N+2} \cdots \tau_{\ell N}) \ u=0,$$

\begin{align*}
  a(0,x_{N+1}) \tau_{N+1} \cdots \tau_{\ell N} \ u
  &= x_{N+1}^{\ell-1} \tau_{N+1} \cdots \tau_{\ell N} \ u \\
  &=(-1)^{\ell-1} (\tau_{N+2} \cdots \tau_{2N}) \cdots(\tau_{(\ell-1)N+2} \cdots \tau_{\ell N}) \ u,
\end{align*}
 from which we obtain  \eqref{eq:nonzero elem in img z ell}.

\smallskip
Since the right-hand side of \eqref{eq:nonzero elem in img z ell} is
a non-zero element of the simple $R(\gamma_a + \ell \beta_a)$-module
$R(\alpha_a, \gamma_a + \ell \beta_a) u$, we conclude that
$$R(\alpha_a, \gamma_a + \ell \beta_a) u \subset \Ker z \cap \Img z^{\ell}.$$

Consider the following sequence of homomorphisms
$$
 \bl \Ker z^{\ell+1} / \Ker z^{\ell} \br \stackrel{z}{\hookrightarrow}
\bl  \Ker z^{\ell} / \Ker z^{\ell-1} \br\stackrel{z}{\hookrightarrow}
\cdots
 \stackrel{z}{\hookrightarrow} \bl \Ker z^2 / \Ker z \br
 \stackrel{z}{\hookrightarrow} \Ker z.
$$
Since
$\Ker z^{\ell+1} / \Ker z^{\ell} \isoto[z^{\ell}] \Ker z \cap \Img z^\ell$,
we have
$$\dim \bl \Ker z^{k} / \Ker z^{k-1} \br \ge \dim \bl \Ker z^{\ell} / \Ker z^{\ell-1} \br \ge \dim K_0(a) \circ L_a^{\circ \ell}$$
for $1 \le k \le \ell+1$.
Because
\begin{align*}
  \dim K_0(a) \circ L_a^{\circ \ell} = \dfrac{((\ell+1)N-1) !}{(N-1)! (N!)^{\ell}}, \ \text{and} \quad
  &\dim K_\ell(a) = \dfrac{((\ell+1)N )!}{(N!)^{\ell+1}},
\end{align*}
we have
$$\dim \Ker z^{\ell+1} = \sum_{k=1}^{\ell+1} \dim \bl  \Ker z^{k} / \Ker z^{k-1} \br \ge
(\ell+1) \dim K_0(a) \circ L_a^{\circ \ell} = \dim K_\ell (a). $$
It follows that
\begin{align*}
  &K_\ell(a) = \Ker z^{\ell+1},  \\
  &\Ker z^{k} / \Ker z^{k-1} \isoto[z^{k-1}] \Ker z \ \text{for} \ 1 \le \ k \le \ell+1,  \\
  &\Ker z = \Img z^{\ell} \simeq K_0(a) \circ L_a^{\circ \ell}.
  \end{align*}
Hence we get the assertions (i) and (ii).

For (iii), observe that
\begin{align*}
\dim K_{\ell-1}(a) \circ L_a = \ell \ \dim K_0(a) \circ L_a^{\circ \ell} , \\
K_{\ell-1}(a) \circ L_a \simeq R(\alpha_a, \gamma_a + \ell \beta_a) R(\ell \beta_a , \beta_a) u.
\end{align*}
 On the other hand,
$z|_{K_{\ell-1}(a) \circ L_a}$ is induced by
$\bl z|_{K_{\ell-1}(a)} \br  \circ L_a $.
It follows that
$$R(\alpha_a, \gamma_a + \ell \beta_a) R(\ell \beta_a , \beta_a) u \subset \Ker z^{\ell}.$$
Comparing the dimensions, we have
$$R(\alpha_a, \gamma_a + \ell \beta_a) R(\ell \beta_a , \beta_a) u = \Ker z^{\ell}$$
as desired.
\QED
\Cor
There  exist  a surjective homomorphism
\eq \label{eq:rigid surj}
L_\ell(a) \circ K_\ell(a) \twoheadrightarrow L_a^{\circ (\ell+1)},
\eneq
and an injective homomorphism
\eq \label{eq:rigid inj}
 L_a^{\circ (\ell+1)}  \monoto
q^{-\ell-1} K_\ell(a) \circ L_\ell(a)
\eneq in $R\bl (\ell+1)\beta_a
\br \gmod$. \encor \Proof From Proposition \ref{prop:rigid} (i), we
have $x_1^{\ell+1}=0$ on $L_a^{\circ (\ell+1)}$. Hence we obtain
\eqref{eq:rigid surj}. Taking duals, we have by \eqref{eq:dualconv}
\eqn
 \bl L_a^{\circ (\ell+1)} \br^*  \monoto q^{(\gamma_a+\ell \beta_a, \alpha_a)} K_\ell(a)^* \circ L_\ell(a)^*.
\eneqn
Since
\eqn
\bl L_a^{\circ (\ell+1)}\br^* \simeq q^{\ell(\ell+1)} L_a^{\circ (\ell+1)}, \
 K_{\ell}(a)^* \simeq  q^{\ell(\ell+1)}K_{\ell}(a), \ \text{and} \ L_\ell(a)^*\simeq q^{-2\ell}L_\ell(a),
\eneqn
we obtain
\eqref{eq:rigid inj}.
\QED

Note that $B(\al_a,\gamma_a)=B(\al_a,\beta_a)=B(\beta_a,\beta_a)=1$.
Hence we have
\eqn &&L_a^{\nconv(\ell+1)}\simeq q^{\ell(\ell+1)/2}L_a^{\circ (\ell+1)}, \\
   &&L_\ell(a)\nconv K_\ell(a)\simeq q^{\ell+1}L_\ell(a)\circ K_\ell(a).
\eneqn
 From  \eqref{eq:rigid surj}, we have
\eqn
L_\ell(a) \nconv K_\ell(a)
 \epito   q^{\ell+1-\frac{\ell(\ell+1)}{2}}L_a^{\nconv(\ell+1)}.
\eneqn
Set
\eq \tK_\ell(a) \seteq q^{\frac{(\ell+1)(\ell-2)}{2}} K_\ell(a) \nconv L_a^{\nconv -(\ell+1)} \in \T'_J.\eneq
Then we obtain a morphism in $\T'_J$
\eqn
\eps_\ell \col L_\ell(a) \nconv \tK_\ell(a) \longrightarrow \one.
\eneqn
Similarly, from \eqref{eq:rigid inj}, we obtain a morphism
\eqn
\eta_\ell \col \one \longrightarrow \tK_\ell(a) \nconv  L_\ell(a).
\eneqn

\Th
\hfill
\bni
\item The object $\tK_\ell(a)$ is a right dual to $L_\ell(a)$
in the category $\T'_J$ and $(\eps_\ell,\eta_\ell)$ is a
quasi-adjunction  \ro see {\rm \S\;\ref{sec:A}}\rf. 
\item The category $\T'_J$  and  $\T_J$
are rigid tensor categories; i.e., every object   has a right dual
object and a left dual object. \ee \enth \Proof (i)\quad  We shall
prove it by the induction on $\ell$. By interpreting
$\tK_0(a)=L_0(a)=0$, the $\ell=0$ case  is  obvious. Assume that
$\ell>0$. By the definition, we have an exact sequence in $\AA$ \eqn
1 \longrightarrow q^{2\ell} L(a) \longrightarrow L_\ell(a)
\longrightarrow L_{\ell-1}(a) \longrightarrow 0. \eneqn

On the other hand, by Proposition \ref{prop:rigid} (ii) and (iii) we
have  an exact sequence \eq 0 \longrightarrow K_{\ell-1}(a) \circ L_a
\longrightarrow K_\ell(a) \stackrel{z^{\ell}}{\longrightarrow}
q^{-2\ell} K_0(a) \circ L_a^{\circ\ell} \longrightarrow
0.\label{eq:Kexact} \eneq Since $B(\gamma_a,\beta_a)=0$, we have
\eqn &&K_\ell(a)\nconv L_a\simeq q^\ell K_\ell(a)\circ L_a, \\
&&K_0(a) \nconv L_a^{\nconv\ell}\simeq
K_0(a) \circ L_a^{\nconv\ell}\simeq q^{(\ell-1)\ell/2}K_0(a) \circ L_a^{\circ\ell}.
\eneqn
Hence \eqref{eq:Kexact}  can be understood as
\eqn
&&0 \To q^{1-\ell}K_{\ell-1}(a) \nconv L_a \To K_\ell(a) \To[\ {z^\ell}\ ]
q^{-2\ell-(\ell-1)\ell/2} K_0(a) \nconv L_a^{\nconv\ell} \To 0.
\eneqn
Applying the functor $\nconv \bl q^{\frac{(\ell+1)(\ell-2)}{2}} L_a^{\nconv -(\ell+1)}\br $,
we have an exact sequence in $\T'_J$
\eqn
0 \longrightarrow \tK_{\ell-1}(a) \longrightarrow \tK_\ell(a) \longrightarrow q^{-2\ell}\tK_0(a)\longrightarrow 0.
\eneqn

We can easily see that the following diagrams are commutative:
\vskip 0.5em
$\scalebox{.93}{\xymatrix{
L_\ell(a) \nconv \tK_{\ell-1}(a)  \ar[r]  \ar[d]
& L_{\ell-1}(a) \nconv \tK_{\ell-1}(a) \ar[d]_{\eps_{\ell-1}} \\
L_\ell(a) \nconv \tK_\ell (a) \ar[r]^{\eps_{\ell}}
& \one \\
q^{2\ell} L(a) \nconv \tK_{\ell}(a)  \ar[r]  \ar[u]
& q^{2\ell} L(a) \nconv q^{-2\ell} \tK_0(a) \ar[u]^{\eps_0},}}
$\hs{2ex}
$
\scalebox{.93}{\xymatrix{
 \tK_{\ell-1}(a) \nconv L_{\ell-1}(a)  \ar[r]
& \tK_{\ell}(a) \nconv L_{\ell-1}(a)   \\
\one \ar[r]^{\eta_\ell} \ar[d]_{\eta_0} \ar[u]^{\eta_{\ell-1}}
&  \tK_\ell (a) \nconv  L_\ell(a) \ar[u] \ar[d]\\
q^{-2\ell} \tK_0(a) \nconv q^{2\ell} L(a) \ar[r]
& q^{-2\ell} \tK_{0}(a) \nconv L_\ell(a).
}}
$

\vskip 1em
Then the assertion follows by the induction on $\ell$ and Lemma \ref{lem:quasi adjunctions}.

\medskip

(ii) By (i), $L_\ell(a)$ has a right dual for every $a \in \Z$ and $\ell \in \Z_{\ge 0}$.
Hence an object of the form
\eqn
q^s L_{\ell_1}(a_1) \nconv \cdots \nconv L_{\ell_r} (a_r)\nconv S
\eneqn
has a right dual, where $S$ is
a tensor product of copies of  $L_a^{\nconv-1}$ ($a\in\Z)$.

Because every object $X$ in $\T'_J$ has a resolution
\eqn
P'\rightarrow P \rightarrow X \rightarrow 0,
 \eneqn
where $P'$ and $P$ are direct sums of  objects with the above form,
we conclude that $X$ also has a right dual.

Similarly, every object has a left dual.
Note that the left dual of $L(a)$ in $\T_J'$
is isomorphic to $q^{-1} L(a-N+1,a-1) \nconv L_{a-N+1}^{\nconv  -1}$.
\QED

Now we will show that the functor
$\F'\col \AA / \Ss \to \UA \smod$ factors through $\T_J$.
We need the following lemma.

\Lemma
For $b\in J$, set $V_k=V_{q^{2(b-k)}}$  $(1\le k\le N)$,
$W=V_{N}\tens V_{N-1}\tens\cdots \tens V_{1}$,
and choose an epimorphism
$\vphi\col W\to \cor$ in $\UA\smod$.
 Let
$$\Rnorm_{W,V_z}\col W\tens V_z\to V_z\tens W$$
be the R-matrix obtained by the composition of normalized R-matrices
$$V_{N}\tens\cdots \tens V_{1}\tens V_z
\To[\Rnorm_{V_1,V_z}]V_{N}\tens\cdots \tens V_{2}
\tens V_z\tens V_1\To\cdots
\To[\Rnorm_{V_N,V_z}]V_z\tens V_{N}\tens\cdots \tens V_{1},$$
and  let  $g(z)=\dfrac{q^{N-1}(z-q^{2(b-N)})}{z-q^{2(b-1)}}$.

Then we have a commutative diagram
\eq
\xymatrix{
W\tens V_z\ar[r]^{\Rnorm_{W,V_z}}\ar[d]_{\vphi\tens V_z}&V_z\tens W\ar[d]^{V_z\tens \vphi}\\
\cor\tens V_z\ar[r]^{g(z)}& V_z\tens \cor.
}\label{diag:RV}
\eneq
\enlemma
\Proof
Set $W'=V_{1}\tens V_{2}\tens\cdots \tens V_{N}$.
Then there exists a $\UA $-module homomorphism
$r\col W\to W'$ such that $\Img(r)\simeq \cor$.
 Thus   we have a commutative diagram
$$\xymatrix@C=10ex{
W\tens V_z\ar[r]^{\Rnorm_{W,V_z}}\ar[d]_{r\tens V_z}&V_z\tens W\ar[d]^{V_z\tens r}\\
W'\tens V_z\ar[r]^{\Rnorm_{W',V_z}}&V_z\tens W'.
}$$
Since $\Hom_{ \uqpg}(V_z,V_z)=\cor(z)$ and
$\Img(r)\simeq \cor$, there exists $g(z)\in \cor(z)$
such that the diagram \eqref{diag:RV}  is commutative.

Set $c_k=q^{2(b-k)}$ and
choose $\vphi$ such that
$\vphi((u_N)_{c_N}\tens\cdots\tens(u_1)_{c_1})=1$.
 Denote   by
$\Rnorm_k$  the normalized $R$-matrix
$$\Rnorm_k\col V_k\tens V_{k-1}\tens\cdots\tens V_{1}\tens V_z
\to  V_z\tens  V_k\tens V_{k-1}\tens\cdots\tens V_{1}$$
given inductively by $\Rnorm_{V_k, V_z }\cdot(V_k\tens  \Rnorm_{k-1} )$.

It is enough to show that
\eq
&& (V_z \tensor \vphi) \circ\Rnorm_N((u_N)_{c_N}\tens\cdots\tens
(u_1)_{c_1}\tens (u_1)_z)
=g(z)(u_1)_z.\label{eq:phig}
\eneq
We shall show
\eq
&&
\ba{l}\Rnorm_k\bl(u_k)_{c_k}\tens\cdots\tens
(u_1)_{c_1}\tens (u_1)_z\br\\
\hs{3ex}\in  \dfrac{q^{k-1}(z-c_k)}{z-c_1}(u_1)_z\tens (u_k)_{c_k}\tens\cdots\tens
(u_1)_{c_1}+\sum_{1<j\le N}
(u_j)_{ z }\tens  V_{k-1}\tens\cdots \tens V_{1}
\ea
\label{eq:Rk}
\eneq
by induction on $k$.
It is trivial if $k=1$. Assume that $k>1$.
Then by \eqref{eq: RV}, we have
\eqn
\Rnorm\bl(u_k)_{c_k}\tens (u_1)_z\br
&=&\dfrac{q(z-c_k)}{z-q^2c_k}(u_1)_z\tens (u_k)_{c_k}+
\dfrac{z(1-q^2)}{z-q^2c_k}(u_k)_z\tens (u_1)_{c_k}\\
&=&\dfrac{q(z-c_k)}{z-c_{k-1}}(u_1)_z\tens (u_k)_{c_k}+
\dfrac{z(1-q^2)}{z-q^2c_k}(u_k)_z\tens (u_1)_{c_k}
\eneqn
and
$$\Rnorm\bl(u_k)_{c_k}\tens (u_j)_z\br
\in \sum_{1<s\le N} (u_s)_{z}\tens V_{k}  \ \text{for}  \ j>1.$$
Hence we obtain \eqref{eq:Rk}.

 Applying $ V_z \tensor \vphi$, we obtain
$$ (V_z \tensor \vphi) \circ\Rnorm_N\bl(u_N)_{c_N}\tens\cdots\tens
(u_1)_{c_1}\tens (u_1)_z\br
\in g(z)(u_1)_z+\sum_{1<j\le N}\cor (u_j)_z,$$
 which yields  \eqref{eq:phig}.
\QED

\medskip
Now we will choose $\{c_{i,j}(u,v)\}_{i,j \in J}$ as  promised
in Remark \ref{rmk:cij}. For $r \in \Z$, set \eqn A_r(z) =
\begin{cases}
  1 & \text{if $r=0$,} \\
  q^{-r}(z+1-X(r)) & \text{otherwise,}
\end{cases}
\eneqn
where $X(r)=q^{2r}$.
 Set
\eqn
B_r(u,v)=
\begin{cases}
  1 & \quad \text{if $r \le 0$,} \\
q^{-r}(1+v)-q^r(1+u) & \quad \text{if} \ r \ge 1.
\end{cases}
\eneqn

Then, for $r \ge 1$, we have
\eq \label{eq:Br Ar}
B_r(0,z)=A_r(z) \quad \text{and} \quad  B_r(z,0)=-A_{-r}(z).
\eneq

\Th \label{thm:cij}

For $k \in \Z_{\ge 0}$ and  $i,j\in \Z$, set
\eqn
&&c_{k,0}(u,v) =
\prod_{\substack{0\le s\le k,\\s\equiv k\bmod N} }
\dfrac{B_{s}(u,v)B_{s-N}(u,v)}
{B_{s-1}(u,v)B_{s-N+1}(u,v)}, \\
&&
c_{i,j}(u,v) =\bc
 c_{ i-j,0}(u,v) & \text{for $j \le i$,}\\
c_{ j-i,0}(v,u)^{-1} & \text{for $j > i$.}\ec
\eneqn

Then the diagram \eqref{eq:Psi}  is commutative  for  the  functor $\F'\col\As/\Ss\to \UA\smod$
 and  the  commuting family of central  objects  $\{(L_a,R_a)\}_{a\in J}$. That is,
 the diagram
\eq &&\hs{2ex}
\ba{l}\xymatrix@C=8.5ex{
\F'(L_a\nconv M)\ar[d]^-{\F'(R_a(M))}\ar[r]^-\sim
&\F'(L_a)\tens\F'(M)\ar[r]^-{ g_a \tens \F'(M)}
&\cor\tens \F'(M)\ar[dr]\\
\F'(M\nconv L_a)\ar[r]^-\sim&\F'(M)\tens\F'(L_a)\ar[r]^-{\F'(M)\tens
g_a} &\F'(M)\tens  \cor  \ar[r]&\F'(M) }\ea\label{eq:1com} \eneq is
commutative  for any isomorphism $ g_a \col\F'(L_a)\isoto\cor$.
\enth

\Proof
First, one can easily check that $\{c_{i,j}(u,v)\}$ satisfies the condition \eqref{cond:cij}.
It is enough to show the commutativity of the diagram \eqref{eq:1com}
for $M=L(j)_z$. In this case, we have $\F'(M)\simeq V_\aff$.

Set $L=L(a)\circ L(a+1)\circ\cdots\circ L(a+N-1)$
and $W=\F'(L)$.
Then we have
$R_a(L(j)_z)=f_{a,j}(z)R_{L, L(j)_z}$. On the other hand,
Proposition~\ref{prop:image of tau} implies
$$\F'(R_{L, L(j)_z})=\prod_{a\le k\le a+N-1}P_{k,j}(0,z)
\Rnorm_{W,\F'(L (j)_z)}.$$
 The above lemma  implies that
$$\F'\bl R_a(L(j)_z)\br
=f_{a,j}(z)q^{N-1}\dfrac{Z-X(a)}{Z-X(a+N-1)}
\prod_{a\le k\le a+N-1}P_{k,j}(0,z)\id_{V_\aff},$$
where $ Z=z_{V}$ and $Z=X(j)(z+1)$.
Hence it is enough to show that
\eq&&f_{a,j}(z)q^{N-1}\dfrac{Z-X(a)}{Z-X(a+N-1)}
\prod_{a\le k\le a+N-1}P_{k,j}(0,z)=1.
 \label{eq:fajPkj}\eneq
Note that $P_{k,j}(0,z)=c_{k,j}(0,z)(-z)^{\delta(j=k+1)}$. Because
\eqn f_{a,j}(z)= (-1)^{\delta_{j,a+N}} z^{-\delta(a \le j <
a+N-1)-\delta_{j,a+N}}, \eneqn
it amounts to showing that
\eqn \label{eq:ck0z}&&\hs{-5ex}\ba{rcl}
\displaystyle\prod_{k=a}^{a+N-1} c_{k,j}(0,z)&=&
(-1)^{\delta(a+1\le j\le a+N -1)}q^{1-N}
\dfrac{\bl z+1-X(a-j+N-1)\br^{\delta(j\not=a+N-1)}}%
{\bl z+1-X(a-j)\br^{\delta(j\not=a)}}\\
&=&
(-1)^{\delta(a < j \le a+N-1)} \dfrac{A_{a-j+N-1}(z)}{A_{a-j}(z)}\ea
\eneqn
for all $a,j \in \Z$.
Since
$ c_{i+1,j+1}(u,v)=c_{i,j}(u,v) $
for all $i,j \in \Z$,  we have only  to show that
\eqn
\displaystyle\prod_{k=a}^{a+N-1} c_{k,0}(0,z)=(-1)^{\delta(a < 0  \le a+N-1)} \dfrac{A_{a+N-1}(z)}{A_{a}(z)}
\eneqn
for all $a \in \Z$.

It is straightforward to show that 
\eqn
\displaystyle\prod_{k=a}^{a+N-1} c_{k,0}(u,v)=
\bc
\dfrac{B_{a+N-1}(u,v)}{B_a(u,v)} & \quad \text{if $a \ge 0$,}\\
\dfrac{B_{-(a+N-1)}(v,u)}{B_{-a}(v,u)} & \quad \text{if $a<1-N$,}\\
\dfrac{B_{a+N-1}(u,v)}{B_{-a}(v,u)} & \quad \text{if $1-N\le a<0$.}
\ec
\eneqn
Then by \eqref{eq:Br Ar}, we  obtain  the desired result.
\QED

Hence,  Proposition~\ref{prop:P=1} implies that
the functor $\F' \colon \As/\Ss\to\UA\smod$ factors through
$\T_J$.
Consequently, we obtain a functor $\tF\col \T_J\to \UA\smod$
such that the following diagram quasi-commutes:
\begin{equation}
 \xymatrix@C=10ex{
\As\ar[r]^-{\mathcal Q}\ar[drr]_-{\F}&\As/\Ss\ar[r]^{\Upsilon }\ar[rd]^(.55){\F'}
&\T'_J\ar[d]\ar[r]^{ \Xi  }&\T_J\ar[dl]^-{\tF}\\
&&\UA\smod\,.}
\end{equation}

Moreover, by  Proposition~ \ref{prop:abelian finite2}, we obtain
\begin{prop}
  The functor $\widetilde{\mathcal F}$ is exact.
\end{prop}
\bigskip

\subsection{The Category $\mathcal C_J$} \label{subsec:CJ}
 \hfill

Recall that $\CC_\g$ denotes the category of finite-dimensional
integrable $\UA$-modules. Let $\mathcal C_J$ be the full subcategory
of $\CC_\g$ consisting of $\UA$-modules $M$ such that every
composition factor of $M$ appears as a composition factor of a
tensor product of modules of the form $V(\varpi_1)_{q^{2s}}$ ($s\in
J$).
By the definition, $\mathcal C_J$ is abelian and is stable under
taking submodules, quotients, extensions and tensor products.
Moreover, $\shc_J$ contains $V(\varpi_i)_{(-q)^{i+2s -1}}$ for $1\le
i\le N-1$ and $s\in \Z$. Hence $\tF$ can be considered as an exact
functor
$$\tF\col \T_J\to \mathcal C_J.$$
Note that the category $\mathcal C_J$
 coincides with  the category $\mathcal C_\Z$ in \cite{HL10}.

\Lemma \label{lem:simples in A/S}
Let $\bs=\bl (a_1,b_1), \ldots, (a_r,b_r) \br$
be an ordered multisegment such that $b_k-a_k+1 \le N$ for any $1 \le k \le r$.
Let $t$ be an integer such that $1 \le t \le r$ and $b_t- a_t +1 =N$.
Let us set $\bs'=\bl (a_k,b_k) \br_{k \neq t}$, and let $M(\bs)$ and $M(\bs')$ be the simple graded $R$-modules associated with $\bs$ and $\bs'$, respectively.

Then
$M(\bs') \circ L(a_t,b_t)$ is isomorphic to $M(\bs)$
in $\AA /\Ss$ up to a grading shift.
\enlemma
\Proof
In this proof, we omit the grading shift.  Set
\begin{align*}
  &L = L(a_1,b_1) \circ \cdots \circ L(a_{t-1},b_{t-1}), \\
  &L' = L(a_{t-1},b_{t-1}) \circ \cdots \circ L(a_1,b_1), \\
  &K = L(a_{t+1},b_{t+1}) \circ \cdots \circ L(a_r,b_r), \\
  &K' = L(a_r,b_r) \circ \cdots \circ L(a_{t+1},b_{t+1}).
\end{align*}
Then $M(\bs')$ is isomorphic to the image of $L \circ K \stackrel{f'}{\to} K' \circ L'$   and
$M(\bs)$ is isomorphic to the image of $L \circ L(a_t,b_t) \circ K \stackrel{f}{\to} K' \circ L(a_t,b_t) \circ L'$.
The homomorphism $f$ is decomposed into
\eqn
L \circ L(a_t,b_t) \circ K    &\To[{\hs{3ex} \phi\hs{3ex} }] &  L \circ K \circ L(a_t,b_t) \\
&\To[{\ f' \circ L(a_t,b_t)\ }] & K' \circ L' \circ L(a_t,b_t) \\
&\To[{\hs{2ex} \psi\hs{2ex} }] & K' \circ L(a_t,b_t) \circ L'.
\eneqn
Since $L(a_t,b_t) \circ K \to K \circ L(a_t,b_t)$ and $L' \circ L(a_t,b_t) \to L(a_t,b_t) \circ L'$
are isomorphisms in $\AA /\Ss$ by Proposition \ref{prop:exact sequences}, $\phi$ and $\psi$ are also
isomorphisms in $\AA /\Ss$.
Hence $M(\bs') \circ L(a_t,b_t)$ is isomorphic to $M(\bs)$ in $\AA /\Ss$.
\QED

\Cor \label{cor:C}
If $X$ be a simple object in $\AA /\Ss$, then $X \circ L_a$ is a simple object in $\AA /\Ss$ for any $a \in \Z$.
\encor

\Prop \hfill
\bni
\item The canonical functor
$\Omega  =  \Xi \circ \Upsilon   \colon \AA/\Ss \to \T_J$ sends simple objects to simple objects.

\item $\dim_\cor \Hom_{\T_J}(X,Y) < \infty$ for any $X,Y \in \T_J$.

\item
Let us denote by  $\Irr(\T_J)$ the set of the isomorphism classes of simple objects in $\T_J$.
 Define an equivalence relation $\sim$ on $\Irr(\T_J)$ by $X \sim Y$ if and only if
$X \simeq q^c Y$ in $\T_J$ for some integer $c$.
Let $\Irr(\T_J)_{q=1}$ be a set of representatives of elements in $\Irr(\T_J) / \sim $.

Then the set $\Irr(\T_J)_{q=1}$  is isomorphic to the set of ordered multisegments
$$\bs=\bl (a_1,b_1),\ldots (a_r,b_r) \br$$ satisfying
\begin{equation} \label{eq:segments<N}
  b_k -a_k+1 <N \ \text{for any} \
    1 \le k \le r.
\end{equation}
\item The functor $\tF\col \T_J \to \mathcal C_J$ induces a bijection between $\Irr(\T_J)_{q=1}$ and $\Irr(\mathcal C_J)$, the set of isomorphism classes of irreducible objects in $\mathcal C_J$.
\ee
\enprop
\Proof
(i)  follows from Corollary \ref{cor:C} and  Proposition \ref{prop:abelian finite2}.

\medskip
\noi
(ii) follows from Lemma \ref{lem:finite Hom}.

\medskip

\noi
(iii)\quad
By (i),  every element in $\Irr(T_J)_{q=1}$ is of the form
 $[\Omega( M(\bs))]$, for some ordered multisegment $\bs=\bl (a_1,b_1),\ldots (a_r,b_r) \br$.
 By Lemma \ref{lem:simples in A/S}, we can assume that $\bs$ satisfies \eqref{eq:segments<N}.
Hence the assignment
$\bs \mapsto [\Omega(M(\bs))] \in \Irr(T_J)_{q=1}$ is surjective.

For two multisegments $\bs_1$ and $\bs_2$ satisfying \eqref{eq:segments<N},  if $\Omega(M(\bs_1)) \simeq \Omega(M(\bs_2))$ in $\T_J$
up to  grading shift, then
$\F(M(\bs_1))\simeq \F(M(\bs_2))$ in $\UA\smod$
implies
$\bs_1=\bs_2$ by Theorem~\ref{th:zero} (iii).
Thus we obtain (iii).

\medskip
(iv)\quad By (iii),  for any simple object $X$ in $\T_J$, we have $X
\simeq (\Omega \cdot\mathcal Q) (q^c M(\bs))$ for some ordered
multisegment $\bs$ satisfying \eqref{eq:segments<N} and  some
integer $c$. Then $\tF(X) = \F(M(\bs))$ is irreducible by Theorem
\ref{thm:irreducible to irreducible}.

It is known that every irreducible module in $\mathcal C_J$ can be obtained
as the head of a tensor product of the form
$$V(\varpi_{i_1})_{(-q)^{c_1}}  \tensor \cdots \tensor V(\varpi_{i_r})_{(-q)^{c_r}}
$$
for some $\{c_k \in \Z\}_{1 \le k \le r}$
such that $c_k\equiv i_k-1\mod 2$ and that $(-q)^{c_k-c_j}$ is not the zero of
$d_{V(\varpi_{j}), V(\varpi_{k})}$ for $1\le j<k\le r$.
Moreover, such a sequence $\bl (i_1,c_1),\ldots(i_r, c_r) \br$
is unique up to a permutation (Theorem~\ref{th:zero} (iii)).

Set $a_k=\dfrac{c_k-i_k+1}{2}$ and $b_k=\dfrac{c_k+i_k-1}{2}$.
 By applying a permutation,  we may assume that the
multisegment $\bl(a_1,b_1),\ldots (a_r,b_r)\br$ is ordered. Note
that $(a_j,b_j)\ge(a_k,b_k)$ implies that $(-q)^{c_k-c_j}$ is not
the zero of $d_{V(\varpi_{j}), V(\varpi_{k})}$. Then we have $$\F(
L(a_1,b_1) \circ \cdots \circ L(a_r,b_r)) \simeq
V(\varpi_{i_1})_{(-q)^{c_1}}  \tensor \cdots \tensor
V(\varpi_{i_r})_{(-q)^{c_r}}$$ and $\F(M(\bs))$ is isomorphic to the
head of $V(\varpi_{i_1})_{(-q)^{c_1}}  \tensor \cdots \tensor
V(\varpi_{i_r})_{(-q)^{c_r}}$. Hence the assignment
$\Irr(\T_J)_{q=1} \ni X \mapsto [\tF(X)] \in \Irr(\mathcal C_J)$ is
bijective. \QED

Finally, we have established   one of the main theorems  of this paper.

\Th
The exact functor $\tF\col\T_J\to \mathcal C_J$ induces  a ring  isomorphism
\begin{equation*}
\phi_{\widetilde{\mathcal F}} \col K(\T_J)/(q-1)K(\T_J) \isoto K(\mathcal C_J).
\end{equation*}
\enth
Therefore $\T_J$ may be regarded as  a
$\Z$-graded lifting of the rigid
tensor category ${\mathcal C}_{J}$.

\vskip 0.5em

Recall that the ring $K(\mathcal C_J)$ admits interesting $t$-deformations (\cite{VV02, Nak04, Her04}).
In \cite{HL11}, Hernandez and Leclerc gave a presentation of $\mathcal K_t$, one of those $t$-deformations.
For each simple object $W$ of $\mathcal C_J$, there is a distinguished element
$\chi_{q,t}(W)$ in $\mathcal K_t$,
 and the $\chi_{q,t}(W)$'s form a basis of $\mathcal K_t$
as a $\C(t^{{1/2}})$-vector space. 
Set
$y_{i,0} \seteq \chi_{q,t}(V(\varpi_1)_{q^{2i}})$ and
set $y_{i,m+1}:=\chi_{q,t}(^*W)$, if $y_{i,m}=\chi_{q,t}(W)$ ($m \in \Z$).
Then $\mathcal K_t$ coincides with the $\C(t^{{1/2}})$-algebra  generated by
$\set{y_{i,m}}{i=1,\ldots,N-1, \ m \in \Z}$ with the defining relations
\begin{align}
&\label{eq:R1}\text{for every} \ m \in \Z,  \\
&\quad\nonumber \begin{cases}
y_{i,m} y_{j,m} =y_{j,m}y_{i,m} &\text{if $j \neq i+1, i-1$,} \\
y_{i,m}^2 y_{j,m}  -(t +t^{-1}) y_{i,m} y_{j,m} y_{i,m} + y_{j,m}^2 y_{i,m}=0
&\text{if $j = i+1$ or $i-1$,}
\end{cases} \\
&\label{eq:R2}\text{for every $m \in\Z$, and every 
$i,j \in \{1,2,\ldots, N-1\}$}, \\
&\quad \nonumber y_{i,m} y_{j,m+1} = t^{-2\delta_{i,j}+\delta(i= j+1)+\delta(i= j-1)} y_{j,m+1}y_{i,m}+\delta{(i=j)}(1-t^{-2}),\\
&\label{eq:R3}\text{for every} \ p > m+1, \text{and every} \ i,j \in \{1,2,\ldots, N-1\}, \\
&\quad  \nonumber y_{i,m} y_{j,p} = t^{(-1)^{p-m} (2\delta_{i,j}-\delta(i= j+1)-\delta(i= j-1))} y_{j,p} y_{i,m}.
\end{align}

Note that 
 our generator  $y_{i,m}$ is the element
$x^{\overleftarrow Q}_{i,m-1}$ given in \cite[\S7.1]{HL11}, associated with the quiver $\overleftarrow Q \ : \ 1 \leftarrow 2 \leftarrow \cdots \leftarrow N-1$.

 Set 
$K_{\C(q^{1/2})}(\mathcal T_J)=\C(q^{1/2}) \tens_{\Z[q^{\pm 1}]}K(\mathcal T_J)$. 

\Th \label{th:def}
There is  a  $\C$-algebra isomorphism 
$\psi\colon \mathcal K_t \rightarrow K_{\C(q^{1/2})}(\mathcal T_J)$ sending
\begin{align*}
  t^{1/2} \mapsto q^{-{1/2}}, \quad  y_{i,2k} \mapsto  [L(i+kN)], \quad  y_{i,2k+1} \mapsto  [q^{-1}L(i+kN+1, i+(k+1)N-1)].
\end{align*}
\enth
\Proof
In the proof, we denote by  $Z_{i,2k}$ and $Z_{i, 2k+1}$ 
the objects $L(i+kN)$ and $q^{-1}L(i+kN+1, i+(k+1)N-1)$), respectively.
Note that $Z_{i,m+1}$ is the right dual of $Z_{i,m}$. 
We denote by $z_{i,m}\in K(\mathcal T_J)$ the element $[Z_{i,m}]$.
It is straightforward to check that $\set{z_{i,m}}{i=1\ldots,N-1, \ m \in \Z}$ satisfy the relations
\eqref{eq:R1}--\eqref{eq:R3}.
For example, we have
\begin{align*}
  &z_{i,2k} \nconv z_{i,2k+1}
  =q^{B(\eps_i-\eps_{i+1}, \eps_{i+1}-\eps_{i+N})} z_{i,2k} \circ z_{i,2k+1}
  =q( z_{i,2k} \circ z_{i,2k+1}) \\
  =&-q^2 L_{i+kN} + q^2 z_{i,2k+1}\circ  z_{i,2k} +L_{i+kN} \\
  =&(1-q^2) \one + q^{2-B(\eps_{i+1}-\eps_{i+N},\eps_{i}-\eps_{i+1})}  z_{i, 2k+1} \nconv z_{i,2k}
  =(1-q^2) \one + q^2 z_{i,2k+1} \nconv z_{i,2k},
\end{align*}
which yields the relation \eqref{eq:R2} when $i=j$ and $m=2k$.
Here, the third equality follows 
from Proposition \ref{prop:exact sequences} (vi).
The other relations can be checked similarly.

Note that we have
\eq \label{eq:generate1}
\hs{3ex}\chi_{q,t}(V(\varpi_i)_{(-q)^p}) &=& \frac{t^{-1/2}}{1-t^{-2}}
\Big( \chi_{q,t}(V(\varpi_{i-1})_{(-q)^{p-1}})  \chi_{q,t}(V(\varpi_1)_{(-q)^{p+i-1}}) \\
&& \qquad -t^{-1} \chi_{q,t}(V(\varpi_1)_{(-q)^{p+i-1}}) \chi_{q,t}(V(\varpi_{i-1})_{(-q)^{p-1}}) \Big), \nonumber
\eneq
for all $(i,p) \in \{2,\ldots,N-1\} \times \Z$ such that 
$i \equiv p+1  \mod 2$.
Indeed, if $i=2$, then it is nothing but a consequence of \cite[Proposition 5.6]{HL11}.
The other cases can be also shown by similar arguments. 
 On the other hand, 
 Proposition \ref{prop:exact sequences} (vi) implies that 
 \eq \label{eq:generate2} 
[L(a,b)] = \frac{1}{1-q^2}
\Big( [L(a,b-1)] \nconv [L(b)] 
-q [L(b)] \nconv  [L(a,b-1)] \Big),
\eneq 
for every $a, b \in \Z$ 
 with $0 < b-a+1 <N$. For such a pair $a,b \in \Z$, 
assume that there exists $k \in \Z$ such that $1+kN \le a \le b \le N-1 + kN$.
Then by comparing \eqref{eq:generate1} with \eqref{eq:generate2}  and
arguing by induction on $b-a+1$, 
 we deduce that
\eq \label{eq:fund_seg}
\psi\Big(\chi_{q,t}(V(\varpi_{b-a+1})_{(-q)^{a+b}} \Big) = q^{\frac{b-a}{2}}[L(a,b)]
\eneq 
 if $1+kN \le a \le b \le N-1 + kN$ for some $k\in\Z$.
If there is no such a $k$, then we have 
\eq
&&\text{there exists $s\in\Z$ such that $a\le sN\le b$.}\label{eq:ab}
\eneq
Recall that there is a $\C$-algebra anti-automorphism $d'$ on $\mathcal K_t$ defined by
$d'(t^{1/2})= t^{-1/2}$,
$d'\Big(\chi_{q,t}(V(\varpi_i)_{(-q)^p}) \Big)= \chi_{q,t}(V(\varpi_{N-i})_{(-q)^{p+N}})$ (\cite[\S5.7, \S7.1]{HL11}).
On the other hand, from the duality, we obtain a ring anti-automorphism $d$ on $K(\mathcal T_J)$ given by
$d(q)=q^{-1}$, $d([X])=[^*X]$ 
$(X \in \mathcal T_J)$. 
Then we have 
$d \circ \psi = \psi \circ d'$, by the definition of $\psi$.
For $a,b \in \Z$ with $0 < b-a+1 <N$ and \eqref{eq:ab}, we have
\eqn
 &&\psi\Big(\chi_{q,t}\bl V(\varpi_{b-a+1})_{(-q)^{a+b}}\br\Big) 
=\psi \Big( d'\big(\chi_{q,t}\bl V(\varpi_{N-b+a-1})_{(-q)^{a+b-N}}\br \big) \Big) \\
&=&d\Big(\psi\big(\chi_{q,t}\bl V(\varpi_{N-b+a-1})_{(-q)^{a+b-N}}\br \big) \Big) 
=d\Big( q^{\frac{N-b+a-2}{2}}L(b+1-N,a-1)\Big) \\
&=&q^{\frac{-N+b-a+2}{2}} q^{-1} L(a,b)
=q^{\frac{-N+b-a}{2}} L(a,b),
\eneqn
where the third equality comes from \eqref{eq:fund_seg}.
Hence we have $$\psi(\chi_{q,t}(V(\varpi_{b-a+1})_{(-q)^{a+b}}))\equiv [L(a,b)]$$
for every $a,b \in \Z$ with $0 < b-a+1 < N$,  where $x \equiv y$ means $x= q^{m/2} y$  for some $m \in \Z$.
 Let $\mathcal K_{\C[t^{\pm 1/2}]}$ denote
the $\C[t^{\pm 1/2}]$-subalgebra of $\mathcal K_t$ generated by 
the $\chi_{q,t}(W)$'s, 
where $W$ ranges over 
the set of the isomorphism classes of simple objects of $\mathcal C_J$.
Because  
the $[L(a,b)]$'s generate the $\Z[q^{\pm 1}]$-algebra $K(\mathcal T_J)$,
the  restriction $\tilde\psi$ of $\psi$ to $\mathcal K_{\C[t^{\pm 1/2}]}$
gives the surjective map
$$\tilde\psi\col \mathcal K_{\C[t^{\pm 1/2}]}\epito
\C[q^{\pm1/2}]\tens_{\Z[q^{\pm 1 }]}K(\mathcal T_J).$$
Since  $\tilde\psi$ gives an isomorphism 
$$\mathcal K_{\C[t^{\pm 1/2}]}/(t^{1/2}-1)\mathcal K_{\C[t^{\pm 1/2}]}\simeq
\C\otimes_{\Z}K(\mathcal C_J) \isoto \C\otimes_{\Z} 
\bl K(\mathcal T_J) /(q-1)K(\mathcal T_J))\br$$
after specializing at $t^{1/2}=1$, the homomorphism 
$\tilde\psi$, as well as $\psi$, is an isomorphism. 
\QED

\appendix
\section{Localization}
In this section, we shall recall  the basic facts on  the localization of tensor categories.
Since the materials here are more or less known or elementary, we omit
most of the proofs.
\subsection{Tensor category}
Let us recall a {\em tensor category} (often called a
{\em monoidal category}).
 In this paper, we mainly consider
additive tensor categories.

A {\em tensor category} consists of the following data:
\bni
\item a category $\T$,
\item a bifunctor $\scbul\tens\scbul\col\T\times\T\to\T$,
\item  an isomorphism $a({X,Y,Z})\col(X\tens Y)\tens Z\isoto X\tens(Y\tens Z)$ which is
functorial in $X,Y,Z\in\T$,
\item an object $\one\in\T$ (called a {\em unit object}),
\item an isomorphism $\eps\col \one\tens\one\isoto \one$
\ee
satisfying the following axioms:
\bna
\item (the Pentagon axiom) the  following diagram  is commutative for any $X,Y,Z,W\in\T$:
\eq\label{diag:pentadiag}
&&\qquad\ba{c}\xymatrix@C=8ex{
{((X\tens Y)\tens Z)\tens W}\ar[d]_-{a(X,Y,Z)\tens W}
\ar[rr]^-{a(X\tens Y,Z,W)}
&&{(X\tens Y)\tens(Z\tens W)}\ar[dd]^-{a(X,Y,Z\tens W)}\\
{(X\tens (Y\tens Z))\tens W}\ar[d]_-{a(X,Y\tens Z,W)}&&\\
           {X\tens ((Y\tens Z)\tens W)}\ar[rr]_-{X\tens a(Y,Z,W)}
                       &&{X\tens (Y\tens (Z\tens W)),}
}\ea\eneq
\item the functors from $\T$ to $\T$ given by
$X\mapsto \one\tens X$ and $X\mapsto X\tens \one$ are fully faithful.
\ee
We refer \cite{KS}, for example, for the fundamental properties
of tensor categories.

\hs{-1.5ex} Note that the isomorphism
$\one\tens \one\tens X\isoto[\eps\tens X]\one\tens X$
induces a  canonical  isomorphism $\one\tens X\isoto X$.
Similarly, there is a canonical isomorphism $X\tens \one\isoto X$.

Note that a unit object $\one$ is unique up to a unique
isomorphism. Namely, for an object $Z$ and an isomorphism $e\col
Z\tens Z\to Z$, if the functor  $X\mapsto Z\tens X$ is an
auto-equivalence of $\T$, then there exists a unique isomorphism
$\vphi\col Z\isoto \one$ such that the diagram
$$\xymatrix@C=8ex{Z\tens Z\ar[r]^-{e}\ar[d]^{\vphi\tens\vphi}
&Z\ar[d]^{\vphi}\\
\one\tens\one\ar[r]^-{\eps}&\one}$$
is commutative.

\medskip
Let $\T$ and $\T'$ be tensor categories. A functor $F\col\T\to \T'$
is called a {\em tensor functor} if it is endowed with an
isomorphism $F(X\tens Y)\isoto F(X)\tens F(Y)$ functorial in
$X,Y\in\T$ and an isomorphism $F(\one)\isoto \one$ which make the
following diagrams commutative:

\eq&&
\ba{c}\xymatrix{
F(X\tens Y\tens Z)\ar[r]\ar[d]&F(X\tens Y)\tens F(Z)\ar[d]\\
F(X)\tens F(Y\tens Z)\ar[r]&F(X)\tens F(Y)\tens F(Z),
}\\[2ex]
\\[1ex]
\xymatrix{
F(\one\tens\one)\ar[r]\ar[d]&F(\one)\tens F(\one)\ar[r]&\one\tens \one\ar[d]\\
F(\one)\ar[rr]&&\one. } \ea\label{dia:tensf} \eneq For $X\in\T$ and
$n\in\Z_{\ge0}$, we write $X^{\tens n}=\underbrace{X\tens
\cdots\tens X}_{\text{$n$-times}}$.

We say that an object $X$ is {\em invertible} if
the functors $Z\mapsto X\tens Z$ and $Z\mapsto Z\tens X$
are equivalences of categories. If $X$ is invertible,  then
there exist
an object $Y$ and isomorphisms $f\col X\tens Y\isoto\one$ and $g\col Y\tens X\isoto \one$ such that
the diagrams
\eqn&&\xymatrix@C=7ex{X\tens Y\tens X\ar[r]^-{f\tens X}\ar[d]_{X\tens g}&
\one\tens X\ar[d]\\
X\tens\one\ar[r]&X}\qtext{and}\quad
\xymatrix@C=7ex{Y\tens X\tens Y\ar[r]^-{g\tens Y}\ar[d]_{Y\tens f}&
\one\tens Y\ar[d]\\
Y\tens\one\ar[r]&Y}
\eneqn
are commutative.
 The triple  $(Y,f,g)$ is unique up to a unique isomorphism.
We write  $Y=X^{\tens(-1)}$ so that one may define  $X^{\tens n}$
for any integer $n$.

We say that a tensor category $\T$ is an {\em additive tensor category}
if $\T$ is additive and $\tens$ is an additive bifunctor.
In the sequel, we consider an additive tensor category.

\subsection{Adjunction and Quasi-adjunction}\label{sec:A}
\begin{definition}
Let $\T$ be a tensor category with a unit object $\one$.
Let $(X,Y)$ be a pair of objects and let $\eps \col X \tensor Y \to \one$ and
$\eta \col \one  \to Y \tensor X$
be morphisms.
\bnum
\item
We say that $(\eps,\eta)$ is an {\em adjunction} and that $X$ is a
{\em left dual} to $Y$ and $Y$ is a {\em right dual} to $X$ if the
conditions $(a)$ and $(b)$ below are satisfied: \bna
\item the composition $X \simeq X \otimes \one  \To[X \tensor\, \eta] X \tensor Y \tensor X
\To[\eps \tensor X] \one  \tensor X \simeq X$ is equal to the identity of $X$.

\item the composition $Y \simeq \one \otimes Y  \To[\eta\, \tensor Y ] Y \tensor X \tensor Y
\To[Y \tensor\, \eps] \one  \tensor Y \simeq Y$ is  equal to the identity of $Y$.
\ee

\item
If the composition $X \otimes \one  \To[X \tensor \eta] X \tensor Y \tensor X
\To[\eps \tensor X] \one  \tensor X$
and $\one \otimes Y  \To[\eta \tensor Y ] Y \tensor X \tensor Y
\To[Y \tensor \eps] \one  \tensor Y$ are isomorphisms,
then we say that $(\eps,\eta)$ is a {\em quasi-adjunction}.
\ee
\end{definition}

\Lemma
Let $\T$ be a tensor category with a unit object $\one$ and let $\eps \col X \tensor Y \to \one$ be a morphism in $\T.$
Then the following conditions are equivalent.
\bna
\item There exists a morphism $\eta\col \one \to Y \tensor X$ such that
 $(\eps,\eta)$ is an adjunction.
\item There exists a morphism $\eta \col \one \to Y \tensor X$ such that
$(\eps,\eta)$ is a quasi-adjunction.
 Namely, the compositions
\begin{align*}
&f\col X \To[X \tensor \eta] X \tens Y \tens X \To[\eps \tensor X] X
\quad \text{and}\quad
g\col Y \To[\eta \tensor Y] Y \tensor X \tensor Y \To[Y \tensor \eps] Y
\end{align*}
are isomorphisms.

\item For any $V, W \in \T$, the composition
\eqn
\Hom_\T(V, Y \tensor W) \to \Hom_\T(X \tensor V, X \tensor Y \tensor W)
 \To[\eps\tens W] \Hom_\T(X \tensor V, W)
\eneqn
is a bijection.
\item For any $V, W \in \T$, the composition
\eqn
\Hom_\T(V, W \tensor X) \to \Hom_\T(V \tensor Y , W \tensor X \tensor Y)
\To[W\tens\eps] \Hom_\T(V \tensor Y , W)
\eneqn
is a bijection.
\ee
In this case, the morphism $\eta$ in {\rm(a)} is unique.

Moreover, if $(\eps,\eta)$ satisfies {\rm(b)},
then the following statements hold.
\bni
\item
We have $(g^{-1} \tensor X) \circ \eta = (Y \tensor f^{-1}) \circ \eta$ and
the pair $\bl \eps, (g^{-1} \tensor X) \circ \eta \br$ is an adjunction,
\item We have $\eps \circ (X \otimes g^{-1})= \eps \circ (f^{-1} \otimes Y)$ and
the pair $\bl \eps \circ(X \otimes g^{-1}), \eta \br$ is an adjunction.
\ee
\enlemma

Hence for an object $X$ of $\T$, a left dual (resp.\ a right dual)
of $X$ is unique up to a unique isomorphism if it exists.

\Lemma \label{lem:quasi adjunctions}
Let $\T$ be an abelian tensor category
 such that $\scbul\tens\scbul$ is an exact bifunctor.
Let
\eqn
0 \to X' \to X \to X'' \to 0, \quad 0 \to Y'' \to Y \to Y' \to 0
\eneqn
be exact sequences and morphisms
\begin{align*}
  \eps' \col X' \tensor Y' \to\one, \quad \eps\col X \tensor Y \to \one, \quad \eps'' \col X'' \tensor Y'' \to  \one, \\
  \eta' \col \one \to Y' \tensor X', \quad \eta \col \one \to Y \tensor X , \quad \eta'' \col \one \to Y'' \tensor X''
\end{align*}
are given so that the following the diagrams are commutative{\rm:}
\vskip 0.5em
$${\xymatrix{
X'\tensor Y \ar[r] \ar[d] & X' \tensor Y'  \ar[d]^{\eps'} \\
X \tensor Y \ar[r]^-{\eps} & \one \\
X \tensor Y'' \ar[u] \ar[r] & X'' \tensor Y'', \ar[u]_{\eps''}}
 \qquad
\xymatrix{
Y'\tensor X' \ar[r]  & Y' \tensor X  \ar[d]\\
\one \ar[u]^{\eta'} \ar[d]_{\eta''} \ar[r]^-{\eta} & Y \tensor X \ar[d] \\
Y'' \tensor X'' \ar[r] & Y \tensor X''.}}$$

Assume further that
$(\eps', \eta')$ and $(\eps'', \eta'')$ are quasi-adjunctions.
Then the pair $(\eps, \eta)$ is also a quasi-adjunction.
\enlemma

\Proof
We shall only show that the composition
$X\tens \one\to X\tens Y\tens X\to \one \tens X$ is an isomorphism.
Consider the following diagram with exact rows:
$$\scalebox{0.85}{
\xymatrix@C=8ex{
0 \ar[r] & \ar @{} [ddr]|*+[o][F-]{A} X' \tensor \one \ar[r] \ar[d] & X \tensor \one \ar[r] \ar[d]
 \ar @{} [ddr]|*+[o][F-]{B} & X'' \tensor \one \ar[r] \ar[d] & 0 \\
         & X' \tensor Y' \tensor X' \ar[d] & X \tensor Y \tensor X  \ar[d] & X'' \tensor Y'' \tensor X''  \ar[d] &  \\
0 \ar[r] & \one \tensor X' \ar[r] & \one \tensor X \ar[r]  &  \one\tensor X''  \ar[r] & 0.
}}
$$
If we show that $\xymatrix@C=0ex{ \ar @{} [r]|*+[o][F-]{A} &}$ and
$\xymatrix@C=0ex{ \ar @{} [r]|*+[o][F-]{B} &}$ are commutative, then
the composition of the middle vertical arrows is an isomorphism
because the one in the right and the one in the left are
isomorphisms. Hence it is enough to show the commutativity of the
squares. For example, the square $\xymatrix@C=0ex{ \ar @{}
[r]|*+[o][F-]{B} &}$ is commutative, because we have the following
commutative diagram.
$$
\scalebox{0.9}{\xymatrix{
&X \tensor \one \ar[r] \ar[ddl]_{X \tensor \eta} \ar[d]^{X \tensor \eta''}  \ &
 X'' \tensor \one \ar[d]^{X'' \tensor \eta''}\\
& X \tensor Y '' \tensor X'' \ar[r] \ar[d]
 & X'' \tensor Y''  \tensor X ''  \ar[ddl]^{\eps'' \tensor X''}\\
X \tensor Y \tensor X \ar[r] \ar[d]^{\eps \tensor X}
& X \tensor Y \tensor X'' \ar[d]_{\eps \tensor  X''}& \\
\one \tensor X  \ar[r]& \one \tensor X''. &
}}
$$
The commutativity of $\xymatrix@C=0ex{ \ar @{} [r]|*+[o][F-]{A} &}$ can be shown in a similar way.
\QED

\subsection{Central objects}\label{sec:central}
Let $\T$ be a tensor category.
A {\em central object} of $\T$ is an object $P$ of $\T$
equipped with an isomorphism
$$ R_P(X)\col P\tens X\isoto X\tens P$$
functorial in $X\in\T$ such that
\eq
&&\text{$\xymatrix{
P\tens X\tens Y\ar[r]_-{R_P(X)}\ar@/^3ex/[rr]^{R_P(X\tens Y)}
&X\tens P\tens Y\ar[r]_-{R_P(Y)}
&X\tens Y\tens P}$ commutes for any $X,Y\in\T$.}
\label{cond:central}
\eneq
 Remark that we don't assume that $R_P(P)=\id_{P\tens P}$.
 If $(P, R_{P})$ is a central object, then
the following diagram is necessarily commutative:
$$\xymatrix@C=15ex@R=4ex{P\tens\one\ar[r]^\sim_{R_P(\one)}\ar[rd]_-\sim&\one\tens P\ar[d]^\bwr\\
&P.}$$

If $(P_1, R_{P_1})$ and $(P_2, R_{P_2})$ are central objects,
then $P_1\tens P_2$ is a central object with
$$ R_{P_1\tens P_2}(X) \col P_1\tens P_2\tens X\isoto[{P_1\tens R_{P_2}(X)}]
P_1\tens X\tens P_2\isoto[{R_{P_1}(X)\tens P_2}] X\tens P_1\tens P_2.$$

  The category $\T_c$ of central objects in $\T$
has a canonical structure of  a tensor category.

\subsection{Commuting family of objects}\label{subsec:comm}
Let $\T$ be a tensor category. Consider a family of object
$\{P_i\}_{i\in I}$ in $\T$ and a family of isomorphisms $\{
B_{i,j}\col P_i\tens P_j\isoto P_j\tens P_i\}_{i,j\in I}$.
\Def\label{def: commfam} We say that  $ (\{P_i\}_{i\in
I},\{B_{i,j}\}_{i,j\in I})$ is a {\em commuting family} if the
isomorphisms $B_{i,j} (i,j \in I)$ satisfy the following conditions:

\begin{minipage}{0.9\textwidth}
\bna
\item $B_{i,i}=\id_{P_i\tens P_i}$ for any $i\in I$,
\item $B_{j,i} \circ B_{i,j}=\id_{P_i\tens P_j}$ for any $i,j\in I$,
\item the isomorphisms $\{B_{i,j}\}_{i,j\in I}$ satisfies the Yang-Baxter equation;
namely, the following diagram is commutative for any $i,j,k\in I$:
$$\xymatrix@R=1ex{
&P_i\tens P_j\tens P_k\ar[dl]_{B_{i,j}}\ar[dr]^{B_{j,k}}\\
P_j\tens P_i\tens P_k\ar[dd]_{B_{i,k}}&&P_i\tens  P_k\tens P_j\ar[dd]^{B_{i,k}}\\
\\
P_j\tens P_k\tens P_i\ar[dr]_{B_{j,k}}&&P_k\tens P_i\tens P_j\ar[dl]^{B_{i,j}}\\
&P_k\tens P_j\tens P_i.
 }$$
\ee
\end{minipage}

\edf

Let us denote by $\{e_i\}_{i\in I}$ the canonical basis of $\Z^{\oplus I}$.
If $ (\{P_i\}_{i\in I},\{B_{i,j}\}_{i,j\in I})$ is a  commuting family,
then we can find
\bnum
\item
an object $P^\al$ of $\T$ for any $\al\in \Zp^{\oplus I}$,
\item
an isomorphism $ P_i\isoto P^{e_i}$ for any $i\in I$,
\item an isomorphism
$f_{\al,\beta}\col P^\al\tens P^\beta\isoto P^{\al+\beta}$
for any $\al,\beta \in \Zp^{\oplus I}$,
\ee
satisfying the following conditions:

\bna
\item  $P^0$ is isomorphic to $\one$,
\item the diagram
$$\xymatrix@C=7ex@R=4ex{
P^\al\tens P^\beta\tens P^\gamma\ar[r]^{f_{\al,\beta}}\ar[d]_{f_{\beta,\gamma}}
&P^{\al+\beta}\tens P^\gamma\ar[d]^{f_{\al+\beta,\gamma}}\\
P^\al\tens P^{\beta+\gamma}\ar[r]^{f_{\al,\beta+\gamma}}&P^{\al+\beta+\gamma}
}
$$ is commutative for any $\al,\beta,\gamma\in \Zp^{\oplus I}$,
\item the diagram
$$\xymatrix@C=6ex{
P_i\tens P_j\ar[r]^-\sim\ar[d]_{B_{i,j}}
&P^{e_i}\tens P^{e_j}\ar[dr]^{f_{e_i,e_j}}\\
P_j\tens P_i\ar[r]^-\sim&P^{e_j}\tens P^{e_i}\ar[r]_{f_{e_j,e_i}}& P^{e_i+e_j}
}$$
is commutative for any $i,j\in I$.
\ee

Moreover, such an
$(\{P^\al\}_{\al\in\Zp^{\oplus I}}, \{f_{\al,\beta}\}_{\al,\beta\in\Zp^{\oplus I}})$
is unique up to a unique isomorphism.

More generally, we have the following lemma. \Lemma\label{lem:mor
comm} Let  $ (\{P_i\}_{i\in I},\{B_{i,j}\}_{i,j\in I})$ and  $
(\{P'_i\}_{i\in I},\{B'_{i,j}\}_{i,j\in I})$ be two commuting
families, and let $(\{P^\al\}_{\al\in\Zp^{\oplus I}},
\{f_{\al,\beta}\}_{\al,\beta\in\Zp^{\oplus I}})$ and
$(\{{P'}^\al\}_{\al\in\Zp^{\oplus I}},
\{f'_{\al,\beta}\}_{\al,\beta\in\Zp^{\oplus I}})$ be the
corresponding families as above. Let $\vphi_i\col P_i\to P'_i$
$(i\in I)$ be a family of morphisms such that the diagram
$$\xymatrix@C=13ex{
P_i\tens P_j\ar[d]^{B_{i,j}}\ar[r]^{\vphi_i \tens \vphi_j }&P'_i\tens P'_j
\ar[d]^{B'_{i,j}}\\
P_j\tens P_i\ar[r]^{\vphi_j\tens\vphi_i}&P'_j\tens P'_i
}$$
is commutative for any $i,j\in I$.
Then there exists a unique family of morphisms
$\vphi_\al\col P^\al\to {P'}^\al$ $(\al\in\Zp^{\oplus I})$
such that the diagram
$$\xymatrix@C=15ex{
P^\al\tens P^\beta\ar[d]^{f_{\al,\beta}}\ar[r]^{\vphi_\al\tens\vphi_\beta}&{P'}^\al\tens {P'}^\beta
\ar[d]^{f'_{\al,\beta}}\\
P^{\al+\beta}\ar[r]^{\vphi_{\al+\beta}}&{P'}^{\al+\beta}}$$
is commutative for any $\al,\beta\in \Zp^{\oplus I}$,
and
$\vphi_{e_i}=\vphi_i$ for any $i\in I$.
\enlemma

\subsection{Localization}
Let $\shc$ be a category.
Then the category $\Fct(\shc,\shc)$ of endofunctors
has a structure of a tensor category by $F\tens G= F\cdot G$, the composition
of functors.
Let $\{\Phi_i\}_{i\in I}$ be a commuting family of objects of $\Fct(\shc,\shc)$.
Then we can define $\Phi^\al\in \Fct(\shc,\shc)$ for
$\al\in \Zp^{\oplus I}$ and $\Phi^\al\cdot\Phi^\beta\isoto \Phi^{\al+\beta}$
as in the preceding subsection.

We define the category $\tC$ as follows.
The objects of $\tC$ are pairs $(X,\al)$ of
$X\in\shc$ and $\al\in\Z^{\oplus I}$.
The homomorphisms are defined by

$$\Hom_{\tC}\bl(X,\al),(Y,\beta) \br
=\indlim_{\substack{\gamma\in\Zp^{\oplus I}, \\
\gamma+\al,\,\gamma+\beta\in\Zp^{\oplus I}}}
\Hom_\shc(\Phi^{\gamma+\al}(X),
\Phi^{\gamma+\beta}(Y)).$$

Note that we have a well-defined inductive system in the above definition,
since $\{\Phi_i\}_{i\in I}$ is a commuting family.

The composition of morphisms in $\tC$ is defined in an evident way.

We define a functor
$\Upsilon\col \shc\to \tC$ by $X\mapsto (X,0)$.
For $\al\in \Z^{\oplus I}$, we define a functor
$\tPhi^\al\col\tC\to\tC$ by
$$(X,\beta)\mapsto (X,\beta+\al).$$
Then all the functors $\tPhi^\al$ are auto-equivalences.
Moreover, the diagram
$$\xymatrix{
{\shc}\ar[r]^{\Phi^\al}\ar[d]&{\shc}\ar[d]\\
{\tC}\ar[r]^{\tPhi^\al}&{\tC}}$$
is quasi-commutative for any $\al\in \Zp^{\oplus I}$.
We call $\tC$ the {\em localization of $\shc$ by
the commuting family $\{\Phi_i\}_{i\in I}$} and denote it by
$\shc[\Phi_i^{-1}\mid {i\in I}]$.

The following lemma  can be  easily verified.
\Lemma
Assume that $\shc$ is an abelian category and
the $\Phi_i$'s are exact functors. Then
\bnum
\item
$\shc[\Phi_i^{-1}\mid {i\in I}]$ is an abelian category and
the functor $\Upsilon \col\shc\to\shc[\Phi_i^{-1}\mid {i\in I}]$ is an exact functor.
\item For $X\in\shc$,
$\Upsilon(X) \simeq 0$ if and only if there exists $\al \in\Zp^{\oplus I}$
such that $\Phi^\al(X)\simeq 0$.
\ee
\enlemma

\subsection{Localization of tensor categories}
\label{app:Local_tensor_categories}

Now let $\T$ be a tensor category and let $\{(P_i, R_{P_i})\}_{i\in I}$
be a family of central objects in $\T$.
Set $$B_{i,j}=R_{P_i}(P_j)\col P_i\tens P_j\isoto P_j\tens P_i$$
for $i,j\in I$.
If $(\{P_i\}_{i\in I},\{B_{i,j}\}_{i,j\in I})$ is a commuting family of objects, we say that $\{(P_i, R_{P_i})\}_{i\in I}$ is a {\em commuting family of central objects}
(see also \cite{KKOP19A} for its generalization). 
 Note that it means that $\{(P_i, R_{P_i})\}_{i\in I}$ satisfies the conditions:
\bna
\item for any $i\in I$, $R_{P_i}$ satisfies \eqref{cond:central},
\label{cond:comm.f}
\item  $R_{P_j}(P_i)\circ R_{P_i}(P_j)=\id_{P_i\tens P_j}$  for any $i,j\in I$,
\label{cond2:comm.f}
\item  $R_{P_i}(P_i)=\id_{P_i\tens P_i}$ for any $i\in I$.
\ee


For a commuting family $\{(P_i, R_{P_i})\}_{i\in I}$ of central objects,
let $\Phi_i\in\Fct(\T,\T)$ be the endofunctor defined by $X\mapsto X\tens P_i$.
We define the isomorphism $B^\Phi_{i,j}\col \Phi_i\Phi_j\isoto \Phi_j\Phi_i$
by $$\Phi_i\Phi_j(X)=X\tens P_j\tens P_i
\isoto[B_{j,i}]X\tens P_i\tens P_j=\Phi_j\Phi_i(X).$$
Then it is easy to see that
$\{\Phi_i\}_{i\in I}$ becomes a  commuting family of endofunctors.

Let $\tT=\T[P_i^{\tens -1}\mid i\in I]$ be the localization of $\T$ by $\{\Phi_i\}_{i\in I}$.
Hence we have
$\Ob(\tT)=\Ob(\T)\times \Z^{\oplus I}$
and
$$\Hom_{\tT}((X,\al),(Y,\beta))
=\indlim_{\substack{\gamma+\al,\,\gamma+\beta\in\Zp^{\oplus I},\\
\gamma\in\Zp^{\oplus I}}} \Hom_{\T}(X\tens
P^{\al+\gamma},Y\tens P^{\beta+\gamma}).$$

\bigskip
For any $\al\in \Zp^{\oplus I}$,
 we can define an isomorphism which is  functorial in $X$
$$ R^\al(X)\col P^\al \tens X\isoto X\tens P^\al$$
such that the  following diagrams are commutative  for any
$X,Y\in\T$:
$$
\xymatrix@C=10ex{
P^\al\tens X\tens Y\ar[r]_-{R^\al(X)  \tens Y }\ar@/^3ex/[rr]^{R^\al(X\tens Y)}
&X\tens P^\al\tens Y \ar[r]_-{ X\tens R^\al(Y)}
&X\tens Y\tens P^\al},$$
\vs{.5ex}
$$\xymatrix@C=10ex{
P^\al\tens P^\beta\tens X\ar[d]\ar[r]_-{P^\al\tens R^\beta(X)}
&P^\al\tens X\tens P^\beta\ar[r]_{R^\al(X)\tens P^\beta}&X\tens P^\al\tens P^\beta\ar[d]\\
P^{\al+\beta}\tens X\ar[rr]^{R^{\al+\beta}(X)}&&X\tens P^{\al+\beta},}$$
$$
\xymatrix@C=10ex{
P_i\tens X\ar[d]^{\bwr}\ar[r]^{R_{P_i}(X)}& X\tens P_i\ar[d]^{\bwr}\\
P^{e_i}\tens X\ar[r]^{R^{e_i}(X)}& X\tens P^{e_i}.}$$
Moreover, such  isomorphisms  $R^\al$ are unique.

 Indeed, $\{(P_i, R_{P_i})\}_{i\in I}$ is a commuting family in
the tensor category $\T_c$ of central objects of $\T$.

\smallskip
The category $\tT$ has a structure of tensor category as follows.

\noindent For $\al,\beta\in \Z^{\oplus I}$ and $X,Y\in\T$, we define
$$(X,\al)\tens (Y,\beta)=(X\tens Y,\al+\beta).$$
For $\al',\beta'\in \Z^{\oplus I}$ and $X',Y'\in\T$,
we define the map
\eqn
&&\Hom_{\tT}\bl(X,\al),(X',\al')\br\times
\Hom_{\tT}\bl(Y,\beta),(Y',\beta')\br\\
&&\hs{10ex}\to\Hom_\tT\bl(X\tens Y,\al+\beta),(X'\tens Y',\al'+\beta')\br\eneqn
by taking the inductive limit of the composition of the morphisms
below with respect to $\gamma,\gamma'\in\Zp^{\oplus I}$
\eqn
&&\Hom_\T( X\tens P^{\al+\gamma},X'\tens P^{\al'+\gamma})
\times
\Hom_\T(Y\tens P^{\beta+\gamma'},Y'\tens P^{\beta'+\gamma'})\\
&&\hs{10ex}\to
\Hom_\T(X\tens P^{\al+\gamma}\tens Y\tens P^{\beta+\gamma'},
X'\tens P^{\al'+\gamma}\tens Y'\tens P^{\beta'+\gamma'})\\
&&\hs{10ex}\simeq
\Hom_\T\bl X\tens  Y\tens  P^{\al+\gamma}\tens P^{\beta+\gamma'},
X'\tens  Y'\tens  P^{\al'+\gamma}\tens P^{\beta'+\gamma'}\br\\
&&\hs{10ex}\simeq
\Hom_\T\bl X\tens  Y\tens  P^{\al+\al'+\gamma+\gamma'},
X'\tens  Y'\tens  P^{\al'+\beta'+\gamma+\gamma'}\br\\
&&\hs{10ex}\to \Hom_{\tT}\bl(X\tens Y,\al+\beta),(X'\tens Y',\al'+\beta')\br.
\eneqn
It is easy to verify that $\tT$ becomes a tensor category.
Moreover, $X\mapsto (X,0)$ gives a tensor functor
$\Upsilon\col\T\to \tT$ such that the image of $P_i$ is an invertible object of $\tT$
for any $i\in I$.
We write $\T[P_i^{\tens-1}\mid i\in I]$ for
$\tT$.

\Lemma \label{lem:Pi_inv} Let $\T$ be a tensor category and let
$\{(P_i, R_{P_i})\}_{i\in I}$ be a commuting family of central
objects of $\T$. Let $\shc$ be another tensor category and
$\Psi\col\T\to \shc$ a tensor functor. Assume that $\Psi(P_i)$ is
invertible for any $i \in I$.
Then the functor $\Psi$ factors
through  $\T\To[\Upsilon]\T[P_i^{\tens-1}\mid i\in I]\To[\Psi']\shc$
with a tensor functor $\Psi'$.
Moreover, such a $\Psi'$ is unique up to a unique isomorphism.
\enlemma

\Prop \label{prop:abelian finite}
Let $(\T, \tens)$ be a tensor category and
let $\{(P_i, R_{P_i})\}_{i\in I}$
be a commuting family of central objects of $\T$.
Consider the following conditions.
\bna
\item $\T$ is an abelian category.
\item $ \tens$ is an exact bifunctor.
\item Any object of $\T$ has a finite length.
\item If $X$ is a simple object of $\T$, then $X \tens P_i$ is a simple object for any $i \in I$.
\ee
Set $\tT=\T[P_i^{\tens-1}\mid i\in I]$.
Then the following statements hold.
\bni
\item If  $(\T, \tens)$ satisfies   {\rm(a)} and {\rm(b)}, then $\tT$ is an abelian category and the functor $\Upsilon \col\T \to \tT$ is exact.
\item If  $(\T, \tens)$ satisfies   {\rm(a)--(d)}, then  $(\tT, \tens)$  satisfies {\rm(a)--(c)},
and the functor $\Upsilon  \col \T \to \tT$
sends simple objects to  simple objects.
 Conversely, every simple object of
$\tT$ is isomorphic to $\Upsilon(S) \tens P^\al$
for some $\al\in \Z^{\oplus I}$ and a simple
object $S$ of $\T$.
\ee
\enprop

For easy reference, we record the following lemma.
\Lemma \label{lem:finite Hom}
Let $\cor$ be a field and let $\mathcal C$ be a $\cor$-linear abelian category.
Assume that
\bna
\item any object of $\mathcal C$ has finite length,
\item $\dim_\cor \Hom_\shc (S,S)<\infty$
for any simple object $S$ in $\mathcal C$.
\ee
Then we have $\dim_\cor \Hom_\mathcal C(X,Y) < \infty$  for all  $X,Y \in \mathcal C$.
\enlemma

\subsection{Graded case}\label{app:graded}
Let $L$ be a $\Z$-module.
An additive tensor category $\T$ is called {\em$L$-graded}
if $\T$ has a decomposition $\T=\soplus\nolimits_{\la\in L}\T_\la$
such that $\tens$   induces a  bifunctor  $\T_\la\times \T_\mu\to \T_{\la+\mu}$
for any $\lam,\mu\in L$ and that $\one\in\T_0$.

Let  $\{(P_i, R_{P_i})\}_{i\in I}$
be a commuting family of central objects of $\T$
such that $P_i\in\T_{\la_i}$ for $\la_i\in L,  i  \in  I $.
Let $\ell\col\Z^{\oplus I}\to L$ be a homomorphism given by $\ell(e_i)=\la_i$ ($i \in I$).
Hence $P^\al$ belongs to $\T_{\ell(\al)}$ for any $\al\in\Zp^{\oplus I}$.

\medskip
Now we assume that $\ell\col \Z^{\oplus I}\to L $ is injective.
We will define a  tensor category $\T'$
and a tensor functor
$$ \Omega \col \T\to\T'$$
such that $ \Omega (P_i)\simeq \one$ for $i \in I$.
 We take $ \Ob(\T')=\Ob(\T)$  and
$$\Hom_{\T'}(X,Y)
=\indlim_{\substack{\al,\beta\in \Zp^{\oplus I},\\\la+\ell(\al)=\mu+\ell(\beta)}}
\Hom_{\T}(X\tens P^{\al}, Y\tens P^{\beta})$$
for $X\in \T_\la$ and $Y\in \T_\mu$.
If $\la-\mu$ is not in the image of $\ell\col \Z^{\oplus I}\to L$,
then we understand that
$\Hom_{\T'}(X,Y)=0$.
The tensor product of $X,Y\in\T'$ is  the same as  the one in $\T$.
 Then $ \cdot \tensor \cdot $ becomes a bifunctor  on $\T'$   as in the case of
$\tT=\T[P_i^{\tens-1}\mid i\in I]$.

Note that the category $\T'$ has a decomposition
$\T'=\soplus_{a\in \Coker(\ell)}\T'_a$.
We write
$\T'=\T[P_i\simeq\one\mid i\in I]$.

\Lemma\label{lem:P_i=1_exact sequences}
 Let $\T$ be an abelian $L$-graded tensor category and  $\{(P_i, R_{P_i})\}_{i\in I}$
be a commuting family of central objects in $\T$ as above.
Assume that the functor $\T\ni X\mapsto P_i\tens X$ is an exact functor
for  all  $i \in I$.
 Then the following statements hold.
\bnum
\item The functor $\Omega \col \T\to\T[P_i\simeq\one\mid i\in I]$
is exact.
\item $\Omega(P_i)$ is isomorphic to $\one$ for any $i\in I$.
\item Every exact sequence in $\T[P_i\simeq\one\mid i\in I]$
is isomorphic to the image of an exact sequence in $\T$.
\item The functor $ \Omega$  is decomposed into
$$\T\To[ \Upsilon ] \T[P_i^{\tens-1}\mid i\in I]\To[\Xi]\T[P_i\simeq\one\mid i\in I],$$
where $\Xi(X,\al)=X$  for $X \in \T, \al \in \Z^{\oplus I}$.
\ee
\enlemma

We have the  similar results to the one in
Proposition~\ref{prop:abelian finite}
for the category
$\T'=\T[P_i\simeq\one\mid i\in I]$ and the functor
$ \Omega  \col \T \to \T'$.

\Prop \label{prop:abelian finite2}
Let $(\T, \tens)$ and $\{(P_i, R_{P_i})\}_{i\in I}$
be as in {\rm Lemma~\ref{lem:P_i=1_exact sequences}}.
Consider the following conditions.
\bna
\item $\T$ is an abelian category.
\item $ \tens$ is an exact bifunctor.
\item  Every  object of $\T$ has  finite length.
\item If $X$ is a simple object of $\T$, then $X \tens P_i$ is a simple object for  all  $i \in I$.
\ee
Set $\T'=\T[P_i\simeq\one\mid i\in I]$.
Then the following statements hold.
\bni
\item If  $(\T, \tens)$ satisfies  {\rm(a)} and {\rm(b)}, then $\T'$ is an abelian category and the functor
$\Omega \col  \T \to \T'$ is exact.
\item If  $(\T, \tens)$ satisfies  {\rm(a)--(d)}, then $(\T', \tens)$  satisfies {\rm(a)--(c)},
and the functor $ \Omega  \col \T \to \T'$
sends simple objects to  simple objects.
 Conversely, every simple object of
 $\T'$ is the image of a simple object of $\T$  under $\Omega$.
\ee
\enprop

 The following proposition gives a characterization of
$\T[P_i\simeq\one\mid i\in I]$.

\Prop\label{prop:P=1} Let $\T$ be a tensor category and let
$\{(P_i,R_{P_i})\}_{i\in I}$
be a commuting family of central objects of $\T$
as in {\rm Lemma~\ref{lem:P_i=1_exact sequences}}.
Let $\shc$ be another tensor category and $\Psi\col\T\to \shc$
a tensor functor.
Assume that, for any $i\in I$, there exists  an isomorphism
$ g_i\col\Psi(P_i)\isoto\one$ such that the
following diagrams are commutative for any $X\in\T$:
\eq&&\hs{4ex}
\ba{l}\xymatrix@C=10ex{
\Psi(P_i\tens X)\ar[d]^{\Psi(R_{ P_i  }(X))}\ar[r]^-\sim&\Psi(P_i)\tens\Psi(X)
\ar[r]^-{g_i\tens \Psi(X)}
&\one\tens \Psi(X)\ar[dr]\\
\Psi(X\tens P_i)\ar[r]^-\sim&\Psi(X)\tens\Psi(P_i)\ar[r]^-{\Psi(X)\tens g_i}
&\Psi(X)\tens\one\ar[r]&\Psi(X).
}\ea\label{eq:Psi}
\eneq
Then $\Psi$ factors
as $ \T\To[\Omega]\T[P_i\simeq\one\mid i\in I]\To[\Psi']\shc$
with a tensor functor $\Psi'$.
Moreover, such a $\Psi'$ is unique up to a unique isomorphism.
\enprop

\Proof
Let $(\{P^\al\}_{\al\in\Zp^{\oplus I}}, \{f_{\al,\beta}\}_{\al,\beta\in\Zp^{\oplus I}})$
be as in \S\;\ref{subsec:comm}.
Taking $X=P_j$ in \eqref{eq:Psi}, the diagram
$$\xymatrix@C=10ex{
\Psi(P_i\tens P_j)\ar[r]^-\sim\ar[d]^{\Psi(R_{P_i}(P_j))}
&\Psi(P_i)\tens \Psi(P_j)
\ar[r]^-{g_i\tens g_j}&\one\tens\one\ar[d]^\id\\
\Psi(P_j\tens P_i)\ar[r]^-\sim
&\Psi(P_j)\tens \Psi(P_i)
\ar[r]^-{g_j\tens g_i}&\one\tens\one
}$$
is commutative.
Hence, Lemma~\ref{lem:mor comm} implies that there exists
a family of  morphisms
$g_\alpha\col\Psi(P^\alpha)\isoto \one$ ($\al\in \Zp^{\oplus I}$) such that
the following diagram commutes:
\eq&&\ba{l}\xymatrix@C=10ex{
\Psi(P^\al\tens P^\beta)\ar[r]^-\sim\ar[d]^{\Psi(f_{\al,\beta})}
&\Psi(P^\al)\tens \Psi(P^\beta)
\ar[r]^-{g_\al\tens g_\beta}&\one\tens\one\ar[d]^\sim\\
\Psi(P^{\al+\beta})\ar[rr]^{g_{\al+\beta}}&&\one.
}\ea\label{eq:Pone}
\eneq
We can also check that the diagram
\eq&&\hs{4ex}
\ba{l}\xymatrix@C=10ex{
\Psi(P^\al\tens X)\ar[d]^{\Psi(R_{ P^\al}(X))}\ar[r]^-\sim&\Psi(P^\al)
\tens\Psi(X)
\ar[r]^-{g_\al\tens \Psi(X)}
&\one\tens \Psi(X)\ar[dr]\\
\Psi(X\tens P^\al)\ar[r]^-\sim&\Psi(X)\tens\Psi(P^\al)
\ar[r]^-{\Psi(X)\tens g_\al} &\Psi(X)\tens\one\ar[r]&\Psi(X)
}\ea\label{eq:Psi1} \eneq
is commutative for any $X\in\T$ and $\al\in
\Zp^{\oplus I}$.

We shall define the functor $\Psi'\col \T'\seteq\T[P_i\simeq\one\mid
i\in I]\To\shc$ as follows. For $X\in\T$, we set $\Psi'(X)=\Psi(X)$.
For $X\in\T_\la$ and $Y\in\T_\mu$, we define $\Hom_{\T'}(X,Y)\to
\Hom_\shc(\Psi'(X),\Psi'(Y))$ as the inductive limit of \eqn
&&\Hom_\T(X\tens P^\al,Y\tens P^\beta)
\to \Hom_\shc\bl\Psi(X\tens P^\al),\Psi(Y\tens P^\beta)\br\\
&&\hs{7ex}
\simeq\Hom_\shc\bl\Psi(X)\tens \Psi(P^\al),\Psi(Y)\tens \Psi(P^\beta)\br
\isoto \Hom_\shc\bl\Psi(X)\tens \one,\Psi(Y)\tens \one\br.
\eneqn
Here the limit is taken over $\al,\beta\in \Zp^{\oplus I}$ such that
$\la+\ell(\al)=\mu+\ell(\beta)$.
It is easy to verify that $\Psi'$ is a well-defined functor.
For $X,Y\in\T$, we have an isomorphism
$$\Psi'(X\tens Y)=\Psi(X\tens Y)\isoto\Psi(X)\tens \Psi(Y)
=\Psi'(X)\tens\Psi'(Y).$$
Let us show that it is an isomorphism of
functors. In order to see this, it is enough to show that for
$X\in\T_\la$, $Y\in\T_\mu$ $X'\in\T_{\la'}$, $Y'\in\T_{\mu'}$, and
$f\in\Hom_{\T'}(X,X')$, $g\in\Hom_{\T'} (Y,Y')$, the diagram \eq
&&\ba{l}\xymatrix@C=15ex{
\Psi'(X\tens Y)\ar[r]^-{\Psi'(f\tens g)}\ar[d]^\sim&\Psi'(X'\tens Y')\ar[d]^\sim\\
\Psi'(X)\tens\Psi'(Y)\ar[r]^-{\Psi'(f)\tens\Psi'(g)}&\Psi'(X')\tens\Psi'(Y')
}\ea\label{eq:Psi'}
\eneq
is commutative.

Assume that $f$ is given by $\tf\in \Hom_\T(X\tens P^\al,X'\tens P^{\al'})$
and $g$ is given by $\tg\in \Hom_\T(Y\tens P^\beta,Y'\tens P^{\beta'})$.
We have two sequences of isomorphisms
\eqn
&&\Psi(X\tens P^\al\tens Y\tens P^\beta)
\isoto[R_{P^\al}(Y)]\Psi(X\tens Y\tens P^\al\tens P^\beta)
\isoto\Psi(X\tens Y\tens P^{\al+\beta})\\
&&\hs{40ex}\isoto\Psi(X\tens Y)\tens \Psi(P^{\al+\beta})
\isoto\Psi(X\tens Y)
\eneqn
and
\eqn
&&\Psi(X\tens P^\al\tens Y\tens P^\beta)
\isoto\Psi(X)\tens \Psi(P^\al)\tens \Psi(Y)\tens \Psi(P^\beta)\\
&&\hs{30ex}\isoto\Psi(X)\tens \one\tens \Psi(Y)\tens\one
\isoto\Psi(X)\tens\Psi(Y).
\eneqn
We denote the first composition by
$\vphi\col\Psi(X\tens P^\al\tens Y\tens P^\beta)\isoto\Psi'(X\tens Y)$
and the second one by
$\psi\col \Psi(X\tens P^\al\tens Y\tens P^\beta)\isoto\Psi'(X)\tens \Psi'(Y)$.

Similarly, we have two isomorphisms $\vphi'\col\Psi(X'\tens
P^{\al'}\tens Y'\tens P^{\beta'})\isoto\Psi'(X'\tens Y') $ and
$\psi'\col \Psi(X'\tens P^{\al'}\tens Y'\tens
P^{\beta'})\isoto\Psi'(X') \tens \Psi'(Y')$. We can easily see that
the following diagram is commutative:
$$\xymatrix{
\Psi'(X\tens Y)\ar[d]^{\Psi'(f\tens g)}&\Psi(X\tens P^\al\tens Y\tens P^\beta)
\ar[r]^-\sim_-{\psi}\ar[l]_-\sim^-\vphi\ar[d]^{\Psi(\tf\tens\tg)}
&\Psi'(X)\tens \Psi'(Y)
\ar[d]^{\Psi'(f)\tens\Psi'(g)}\\
\Psi'(X'\tens Y')&\Psi(X'\tens P^{\al'}\tens Y'\tens P^{\beta'})
\ar[r]^-\sim_-{\psi'}\ar[l]_-\sim^-{\vphi'}&\Psi'(X')\tens
\Psi'(Y'). }$$ On the other hand, \eqref{eq:Pone} and
\eqref{eq:Psi1} imply that $\psi$ coincides with the composition
$\Psi(X\tens P^\al\tens Y\tens P^\beta) \isoto[\vphi]\Psi(X\tens
Y)\isoto\Psi(X)\tens \Psi(Y)$, and a similar relation holds for
$\vphi'$ and $\psi'$. Hence we obtain the commutativity of
\eqref{eq:Psi'}. \QED
 \Lemma \label{lem:P=1_exact} Under the
conditions in the above  proposition, we  further assume that
$\shc$ is an abelian category and $\Psi$ is an exact functor. Then
the functor $\Psi'\col\T[P_i\simeq\one\mid i\in I]\to \shc$ is
exact. \enlemma
 \Rem
The commutativity of \eqref{eq:Psi}
does not depend on the choice of an isomorphism $g_i$  ($i \in I$).
Indeed, in a tensor category $\T$,
the diagram
$$\xymatrix{
X\tens\one\ar[r]^-{\sim}\ar[d]_{X\tens\vphi}
&X\ar[r]^-{\sim}&\one\tens X\ar[d]_{\vphi\tens X}\\
X\tens\one\ar[r]^-{\sim}&X\ar[r]^-{\sim}&\one\tens X
}
$$
 is  commutative for any $\vphi\in\End_\T(\one)$.
\enrem

\subsection{Variant} \label{sec:Variant}
 Let $\T$ be an abelian $L$-graded tensor category and  $\{(P_i, R_{P_i})\}_{i\in I}$
be a commuting family of central objects in $\T$ as in \S\,\ref{app:graded}.
Let
$(Q, R_Q)$ be an invertible central object of $\T$
such that $Q\in\T_0$.
We assume that
\eq&&
\parbox{70ex}{$R_Q(P_i)\col Q\tens P_i\isoto P_i\tens Q$ and
$R_{P_i}(Q)\col P_i\tens Q\isoto Q\tens P_i$ are inverse to each other
for any $i\in I$.}
\eneq
It is equivalent to saying that $\{(Q, R_Q),(P_i, R_{P_i})\mid i\in I\}$
is a commuting family of central objects.

Given a family of integers  $(s_i)_{i\in I}$, let
$s\col \Z^{\oplus I}\to \Z$ be the linear map defined by $s(e_i)=s_i$.
Then we can define  a category $\T'$ as follows.
We take $\Ob(\T')=\Ob(\T)$.
For $X\in\T_\la$ and $Y\in \T_\mu$,
$$\Hom_{\T'}(X,Y)
=\indlim_{\substack{\al,\beta\in \Zp^{\oplus I},\\\la+\ell(\al)=\mu+\ell(\beta)}}
\Hom_{\T}(X\tens P^{\al}\tens Q^{\tens-s(\al)}, Y\tens P^{\beta}\tens Q^{\tens-s(\beta)}).$$
Then we can easily see that $\T'$ is also a tensor category,
and there exists a tensor functor
$\Omega\col \T\to\T'$.
We have
$$\Omega(P_i)\simeq \Omega(Q^{\tens s_i}).$$
We denote $\T'$ by
$$\T\,[P_i\simeq Q^{\tens s_i}\mid i\in I].$$
Indeed, $\widetilde{P}_i\seteq P_i\tens Q^{\tens-s_i}$ form
a commuting family of central objects and
$\T'\simeq \T\,[\widetilde{P}_i\simeq \one\mid i\in I]$.

\subsection{Twisting of tensor structure}\label{app:twist}
\ As in  \S\,\ref{app:graded},
let $\T= \soplus\nolimits_{\la \in L} \T_\la$ be an
$L$-graded additive tensor category.
Let $(Q, R_Q)$ be an invertible central object of $\T$
such that $Q\in\T_0$.
Then $Q^{\otimes n}$ is a central object for all $n\in\Z$.

Let $B$ be a $\Z$-valued bilinear form on $L$.
We define an additive bifunctor
$\ntens$  on $\T$  by
$$X\ntens Y=Q^{\tens B(\la,\mu)}\tens X\tens Y\qtext{for $X\in\T_\la$ and $Y\in \T_\mu$.}$$
Then it is easy to see that $\ntens$ gives
a new tensor category structure on $\T$, where
the associativity is given by
\eqn
(X\ntens Y)\ntens Z
&\simeq& Q^{\tens B(\la+\mu,\zeta)}\tens Q^{\tens B(\la,\mu )}
\tens X\tens Y\tens Z\\
&\simeq& Q^{\tens B(\la,\mu)+B(\la+\zeta)+B(\mu,\zeta)}
\tens X\tens Y\tens Z\\
&\simeq& Q^{\tens B(\la,\mu+\zeta)}\tens
X\tens Q^{\tens B(\mu,\zeta)}
\tens Y\tens Z
\simeq X\ntens (Y\ntens Z)
\eneqn
for $X\in\T_\la$, $Y\in\T_\mu$, $Z\in \T_\zeta$.

We say that $(\T,\ntens)$ is the tensor category twisted by $Q$ and $B$.

\section{Quotient categories}\label{app:Serre}

\subsection{Serre category}

In this appendix, we recall the notion of the quotient category of
an abelian category by a subcategory. For more  details, see \cite[\S\;4.3]{Pop}.
 Let $\As$ be
an abelian category and let $\Ss$ be a {\em Serre subcategory} of $\As$\,; i.e.,
\eqn&&\parbox{80ex}{
 (i) $\Ss$ is a full subcategory of $\As$,

 (ii) $\Ss$ is stable  under  taking subobjects, quotients and extensions,
namely, for any exact sequence
$0 \rightarrow X' \rightarrow X \rightarrow X''\rightarrow 0$ in $\As$,
the middle term $X$ is in $\Ss$ if and only if $X'$ and $X''$ are in $\Ss$.
}
\eneqn
For two objects $X$ and $Y$ of $\As$, we have a directed set
$$\mathfrak L (X,Y) = \{(X',Y') \, | \, X' \subset X, Y' \subset Y,
X/X' \in \Ss, Y' \in \Ss\}$$
with the order given by
$$(X'_1,Y'_1) \leq (X', Y') \Leftrightarrow X' \subset X'_1, Y'_1 \subset Y'.$$
If $(X'_1,Y'_1) \leq (X', Y')$, we have a canonical homomorphism of abelian groups
$$\Hom_{\As}(X'_1, Y/ Y'_1) \rightarrow \Hom_{\As}(X', Y/ Y').$$

\noindent The {\em quotient category} $\As / \Ss$ is defined as follows:
\begin{enumerate}
\item The objects of $\As / \Ss$ are the same as the objects of $\As$.
\item For two objects $X$, $Y$ of $\As / \Ss$, the morphisms are given by
\eqn&&
  \Hom_{\As / \Ss}(X, Y) \seteq
  \varinjlim_{(X',Y') \in \mathfrak L (X,Y)} \Hom_{\As}(X', Y/Y').
\eneqn

\item
We refer the reader to \cite[\S\;4.3]{Pop} for
the definition of the composition
$\Hom_{\As / \Ss}(X, Y)\times \Hom_{\As / \Ss}(Y, Z)\to\Hom_{\As / \Ss}(X, Z)$
for $X,Y,Z\in\Ob(\As/  \Ss)$.
\end{enumerate}

 Then we  can define the functor $\mathcal Q \col \As \rightarrow \As / \Ss$.

\begin{theorem}[{\cite[\S\;4.3]{Pop}}] \label{thm:quotient category}
  Let $\Ss$ be a Serre subcategory of an abelian category $\As$.
  Then the following statements hold.
  \bnum
    \item The quotient category $\As / \Ss$ is abelian.
\item  For an object $X \in \As$,  $\mathcal Q(X)\simeq 0$
if and only if $X \in \Ss$.

    \item The functor $\mathcal Q \col \As \rightarrow \As / \Ss$ is exact.
\item Every  exact sequence in $\As / \Ss$ is isomorphic to the image of
an exact sequence in $\As$ under  $\mathcal Q$.
\item
 Let  $\mathcal B$ be an abelian category and let $\mathcal H\col \As \rightarrow \mathcal B$
be an exact functor such that
   $\mathcal H(X)\simeq0$ for all $X$ in  $\Ss$.
Then $\mathcal H $ factors through $\As / \Ss$ with an exact  functor
    $\overline{\mathcal H} \col \As / \Ss \rightarrow \mathcal B$.
Moreover, such an $\overline{\mathcal H}$ is unique up to isomorphism.
\label{q6}
   \item An additive functor
$\mathcal G\col \As /\Ss \rightarrow \mathcal B$ is exact
if and only if $\mathcal G \circ \mathcal Q
\col \As \rightarrow \mathcal B$ is exact.
  \end{enumerate}
\end{theorem}

As for simple objects in $\As / \Ss$, we have the following proposition.
\begin{prop} \label{prop:simples in quotient cat.}
Let $\Ss$ be a Serre subcategory of an abelian category $\As$.
\bnum
\item If $X$ is simple in $\As$ and $X \notin \Ss$, then $\mathcal Q(X)$ is simple in $\As / \Ss$.
\item Assume that every object in $\As$ has finite length.
Then  every simple object $Y$ in $\As /\Ss$ is isomorphic to
$\mathcal Q(X)$ for a simple object $X$ in $\As$.
\item If $X_1$ and $X_2$ are simple objects of $\As$ and $\mathcal Q(X_1) \simeq \mathcal Q(X_2)\not\simeq0$ in $\As /\Ss$,
    then $X_1 \simeq X_2$ in $\As$.
\end{enumerate}
\end{prop}
\noindent Since the proofs are elementary, we omit them.

The following lemma is also elementary.

\Lemma
Let $\cor$ be a field and let $\As$ be a $\cor$-linear abelian category.
Assume the following conditions:
\bna
\item $\dim_\cor\Hom_{\As}(X,Y)<\infty$ for any $X,Y\in\As$,
\item any object of $\As$ has finite length.
\ee
Then, for any Serre subcategory $\Ss$ of $\As$,
the quotient category $\As/\Ss$ is a $\cor$-linear abelian category
satisfying $(a)$ and $(b)$.
\enlemma

The following proposition is easy to verify.
\Prop\label{prop:quot tensor}
Let $\As$ be an abelian tensor category such that
$\tens$ is an exact bifunctor
and let $\Ss$ be a Serre subcategory of $\As$.
 Assume the following condition:
\eq&&\text{
for any $X\in\As$ and $Y\in\Ss$,
$X\tens Y$ and $Y\tens X$ belong to $\Ss$.}
\eneq
Then $\As/\Ss$ has a structure of a tensor category such that
$\mathcal Q\col \As\to\As/\Ss$ is a tensor functor.
\enprop

\end{document}